\documentclass{amsart}
\usepackage{amsmath,amssymb,dsfont,amsthm}
\usepackage{graphics}
\usepackage{xypic}
\usepackage{latexsym}
\usepackage{amsmath}
\usepackage{amsfonts}
\usepackage{geometry}
\usepackage{amssymb}
\usepackage{tensor}
\usepackage{bm, amsbsy}

\date{\today}

\title[Analytic torsion of cones]{The analytic torsion of the finite metric cone over a compact manifold}


\author{L. Hartmann  and M. Spreafico}
\address[Luiz Hartmann]{\tt UFSCar, Universidade Federal de S\~{a}o Carlos,  Brazil.  Partially supported by CNPq and FAPESP 2013/04396-6}
\email{hartmann@dm.ufscar.br}
\address[Mauro Spreafico]{\tt Dipartimento di matematica e fisica, Universit\`a del Salento, Lecce, Italia.}
\email{mauro.spreafico@unisalento.it}

\numberwithin{equation}{section}

\newtheorem{theo}{Theorem}[section]
\newtheorem{lem}{Lemma}[section]
\newtheorem{corol}{Corollary}[section]
\newtheorem{defi}{Definition}[section]

\newtheorem{prop}{Proposition}[section]
\newtheorem{rem}{Remark}[section]

\newcommand{\beq}{\begin{equation}}
\newcommand{\eeq}{\end{equation}}

\renewcommand{\Re}{{\rm Re}}

\renewcommand{\b}{{\partial}}

\newcommand{\rk}{{\rm rk}}
\newcommand{\da}{{\dagger}}
\renewcommand{\det}{{\rm det\hspace{1pt}}}

\newcommand{\bu}{{\bullet}}
\newcommand{\Sp}{{\rm Sp}}
\newcommand{\sgn}{{\rm sgn}}
\newcommand{\Det}{{\rm Det\hspace{1pt}}}

\newcommand{\Z}{{\mathds{Z}}}
\newcommand{\R}{{\mathds{R}}}
\newcommand{\C}{{\mathds{C}}}

\newcommand\e{{\rm e}}

\newcommand{\es}{\mathsf{e}}
\newcommand{\gs}{\mathsf{g}}

\newcommand{\vs}{\mathsf{v}}

\newcommand{\alphas}{\bm{\alpha}}

\renewcommand{\P}{\mathcal{P}}
\renewcommand{\H}{\mathcal{H}}
\newcommand{\T}{{\mathcal{T}}}
\renewcommand{\SS}{{\mathcal{S}}}
\renewcommand{\AA}{\mathcal{A}}

\DeclareMathOperator*{\As}{A\hspace{-5.5pt}\slash}
\newcommand{\A}{{\As}}
\DeclareMathOperator*{\Bs}{B\hspace{-5.5pt}\slash}
\newcommand{\B}{{\Bs}}

\DeclareMathOperator*{\Rz}{Res_0}
\DeclareMathOperator*{\Ru}{Res_1}

\begin{document}


\maketitle

\begin{abstract}  We give an explicit formula for the $L^2$ analytic torsion of the finite metric cone over an oriented compact connected Riemannian manifold. 
We provide an interpretation of the different factors appearing in this formula. We prove that the analytic torsion of the cone is the finite part of the limit obtained collapsing one of the boundaries, of the ratio of the analytic torsion of the frustum to a regularising factor. We show that the regularising factor comes from the set of the non square integrable eigenfunctions of  the Laplace Beltrami operator on the cone.
\end{abstract}


\section{Introduction and statement of the main result}
\label{s0}

Let $(M,g)$ be  a compact connected oriented Riemannian manifold  without boundary of dimension $n$ with metric $g$. Let $\Delta$ denotes the Hodge-Laplace operator  on $M$ in the metric $g$. Then, $\Delta$ has a non  negative discrete spectrum $\Sp \Delta$ and the {\it zeta function} of $\Delta$ is well defined by 
\[
\zeta(s,\Delta^{(q)})=\sum_{\lambda\in \Sp_+\Delta^{(q)}} \lambda^{-s},
\]
for $\Re(s)>\frac{n}{2}$, and by analytic continuation elsewhere, and is regular at $s=0$. The {\it analytic torsion} $T(M,g)$ of the pair $(M,g)$ is  defined by
\[
\log T(M,g)=\frac{1}{2}\sum_{q=1}^{n} (-1)^q q \zeta'(0,\Delta^{(q)}).
\]

If the manifold $M$ has a boundary $\b M$, the Laplace operator is assumed to be defined by suitable  boundary conditions BC \cite[Section 3]{RS}. In such a case, it is convenient to split the logarithm of the analytic torsion into two parts, the first being a global term and the second a local one, defined on a neighborhood of the boundary \cite[Section 3]{Che1}.  The first term coincides with the Reidemeister torsion \cite{Mil} of either $M$ of the pair $(M,\b M)$, with the Ray and Singer homology basis \cite{RS}, by the Cheeger M\"{u}ller theorem \cite{Che1} \cite{Mul}, so
\beq\label{e000}
\begin{aligned}
\log T_{\rm abs}(M,g)&=\log \tau(M,g)+\log T_{\rm bound, abs}(\b M),\\
\log T_{\rm rel}(M,g)&=\log \tau((M,\b M),g)+\log T_{\rm bound, rel}(\b M),
\end{aligned}
\eeq
the second term splits as
\[
\log T_{\rm bound, BC}(\b M)=\frac{1}{4}\chi(\b M)\log 2+A_{\rm BM, BC}(\b M),
\]
where $\chi$ is the Euler characteristic, and the last term is called the anomaly boundary term, and was described in the more general case in \cite{BM1} and \cite{BM2}.

It is clear that what is necessary in order to define the analytic torsion $T(M,g)$ is that the spectrum of the Hodge-Laplace operator satisfies some assumptions that guarantee the possibility of defining the zeta function and of proving its regularity at $s=0$. It is also clear that this follows by some spectral properties of the Hodge-Laplace operator. There are several approaches to describe these properties. We will follow the one of Spreafico, introduced in \cite{Spr4} and \cite{Spr9}: so we require  that $\Sp_+(\Delta)$ is a graded regular sequence of spectral type (of non positive order), as defined in \cite[Definitions 2.1, 2.6]{Spr9}. The fact that this is true for the Hodge-Laplace operator on a compact manifold is well know.

The given definition of the analytic torsion 
extends easily  considering forms with values in some vector bundle $V_\rho$ associated to some orthogonal representation $\rho$ of the fundamental group of $W$ \cite[Section 1]{RS}. Under our approach what is necessary is that the spectrum of the resulting operator is a regular sequence of spectral type of non positive order, and again this is well known. Since the results of this paper are independent of these extensions, we will consider the simpler case of the Hodge-Laplace operator itself.

An other possible generalization of the given definition of analytic torsion, and this is the case that we will consider here,  is when the underlying space is no longer a compact manifold, but some type of open manifold, or manifold with singularities. In this paper we consider the case of the cone over a manifold $C(W)$, as defined below. 

\begin{defi} \label{metric} Let $(W^{(m)},g)$ be an oriented compact connected Riemannian manifold of dimension $m$ without boundary with metric $g$. 
Let $0\leq l_1\leq l_2$ be real numbers. Consider the space $C_{[l_1,l_2]}(W)=[l_1,l_2]\times W$ with the  metric (defined for $x>0$ when $l_1=0$)
\[
dx\otimes dx+x^2 g.
\]

We call $C_{[l_1>0,l_2]}(W)$ the {\it finite metric frustum} over $W$;   we call $C_{(0,l]}(W)$ the {\it finite metric cone} over $W$, and we denote it by $C_l(W)$.


\end{defi}

It is clear that in order to obtain a suitable extension of the definition of the analytic torsion to the cone, we need first a suitable definition of the Hodge-Laplace operator. Spectral analysis on cones was developed by Cheeger \cite{Che2} \cite{Che3}, that in particular showed that the formal Hodge-Laplace operator $\Delta$ has a self adjoint extension on the space of square integrable forms.  A  complete set of solutions of the eigenvalues equation for $\Delta$ can be described in terms of a complete  discrete resolution of the Hodge-Lapalce operator on the section of the cone, see Lemma \ref{l2} \cite[Section 3]{Che2}. Considering square integrable forms and applying the boundary conditions, we obtain an explicit description of the spectrum of $\Delta$ in terms of the spectrum of the Hodge-Laplace operator on the section, see Lemmas \ref{l3}. 

\begin{rem}\label{ideal} It is important to observe here that beside the usual (either absolute or relative) boundary conditions, due to the presence of a non empty boundary, when the section $W$ has even dimension $m=2p$, further boundary conditions, called  {\rm ideal boundary conditions}, where introduced by Cheeger \cite{Che2}, as we recall here. Assume that there exists a decomposition $\H^p(W)=V_a\oplus V_r$, where $V_a$ and $V_r$ are  maximal self annihilating subspaces for the cup product paring. Then, a $p$ dimensional form belongs to the domain of the Laplace operator if its components in $V_a$ and $V_r$ satisfy Neumann and Dirichlet conditions respectively at $x=0$ (see \cite{Che2} pg. 580 for details).  Observe that our determination of the spectrum of the Laplace operator on  the cone is obtained assuming this decomposition and these conditions. Moreover, where ever not explicitly stated, the above decomposition and the ideal conditions will be assumed. We conclude this remark recalling that ideal boundary conditions are necessary to guarantee Poincar\'e duality on the cone.
\end{rem}

We can then prove that $\Sp_+(\Delta)$ is a regular sequence of spectral type and therefore the analytic torsion of the cone is well defined. 
We are in the position of extending the above definition of the analytic torsion to this setting, and  we call the resulting object the $L^2$ analytic torsion of the cone, $T_{\rm abs, ideal}(C_l(W))$ (we use the same notation and we  restrict to absolute BC, since the relative torsion follows by Poincar\'e duality \cite[Section 4]{HS2}). Of course it is not obvious at all if the invariant obtained in this way has some geometric or topological meaning. Since the spaces of the $L^2$ harmonic forms on $W$ are proved to be isomorphic to the intersection cohomology of $W$, the natural candidate for the $L^2$ analytic torsion is the intersection torsion. It was proved in \cite{HS3} that indeed $L^2$ analytic torsion and intersection torsion of the cone over an odd dimensional manifold coincide. The case of an even dimensional section is not clear yet, due to some difficulty in producing a natural definition for the intersection torsion in this case. There is work in progress in this direction. Here we tackle the analytic side of the problem.

The main purpose of this work is to compute $T_{\rm abs, ideal}(C_l(W))$, or more precisely to give formulas for it in terms of other either geometric or spectral invariants, and in particular invariants of the section $(W,g)$. A key point here is the following. The description of the spectrum of the Hodge-Laplace operator on the cone in terms of the spectrum of the Hodge-Laplace operator on the section makes possible to express (or decompose) all the spectral functions (in particular the zeta function and the logarithmic Gamma function, see Subsection \ref{s3}) on the cone in terms of the spectral functions on the sections, and to apply the technique introduced in \cite{Spr9} in order to tackle the derivative at zero of a class of double series. This is the main technical point, and in fact this is the reason that permits to obtain the final results. We did follow the same approach first in \cite{HS1}, where we gave the formula for the torsion of the cone over a sphere, and then in \cite{HS2}, where we considered as section any compact connected manifold of odd dimension. These results are superseded by the formulas obtained in the present work, where the section can have any dimension. We stress the fact that, as observed in \cite{HS2}, the calculation are exactly the same and there is no more difficulty to deal with the even dimensional section case that with  the odd dimensional one. Moreover, while the formulas in the odd dimensional section case have a clear geometric interpretation in terms of intersection torsion, due to \cite{HS4}, the even dimensional section case is still quite obscure. However, due to the recent interest in the subject (see \cite{MV}), we decided to present the formulas for the general case and the details of the calculation. Indeed, the presentation followed in this paper adds some insights in the possible interpretation of the result, as we explain at the end of this section, but first we present our first main result in the following theorem.

\begin{theo}\label{t0} Let $(W^{(m)},g)$ be an oriented compact connected Riemannian manifold of dimension $m$ with metric $g$. The $L^2$ analytic torsion $T_{\rm abs, ideal}(C_l(W))$ of the cone over $W$ with absolute  and ideal  boundary conditions is as follows, where $r_q=\rk H_q(W)$, and  $m_{{\rm cex},q,n}$, $\alpha_q$ and $\mu_{q,n}$ are defined in  Lemma \ref{l2}, $p\geq 1$:
\begin{align*}
\log T_{\rm abs}(C_l (W^{(2p-1)}))
=&\frac{1}{2}\sum_{q=0}^{p-1} (-1)^q (2p-2q)r_q\log l+\frac{1}{2}\log T(W,g)-\frac{1}{2} \sum_{q=0}^{p-1} (-1)^{q}  r_q\log(2(p-q))\\&+A_{\rm BM, abs}(\b C_l (W)),
\end{align*}

\begin{align*}
\log T_{\rm abs, ideal}(C_l(W^{(2p)}))
=&\left(\frac{1}{2}\sum_{q=0}^{p-1} (-1)^{q}  (2p-2q+1)r_q+(-1)^{p} \frac{1}{4}r_p\right)\log l\\
&-\frac{1}{2} \sum_{q=0}^{p-1} (-1)^{q}  r_q\log (2p-2q+1)((2p-2q-1)!!)^2\\
&+\frac{1}{2}\sum^{p-1}_{q=0} (-1)^q \sum_{n=1}^\infty m_{{\rm cex},q,n}\log\frac{1+\frac{\alpha_q}{\mu_{q,n}}}{1-\frac{\alpha_q}{\mu_{q,n}}}\\
&+\frac{1}{2}\chi( W)\log 2+A_{\rm BM, abs}(\b C_l (W)).
\end{align*}

\end{theo}

Next, we give an interpretation  of the formulas given in Theorem \ref{t0}. This is done in two steps. In the first, described in whole details in Section \ref{sgeo}, we show that  all the sums appearing in the formula of the torsion in the odd case, and some of these sums in the even case, coincide with the determinant of the change of basis between the basis of the harmonic forms on the cone and the basis of the harmonic forms on the sections, see Theorem \ref{t1} of Section \ref{sgeo}. This is a classical geometric contribution in the Reidemeister torsion of the cone (see \cite{HS4}), and therefore this interpretation clarifies completely the appearance of these sums. In particular, this describes completely the $L^2$-analytic torsion of a cone over an odd dimensional manifold, and also suggests that the Euler part of the boundary contribution is anomalous in the analytic torsion (see Remark \ref{euler} at the end of Section \ref{sgeo}). In the other case, i.e. for a cone over an even dimensional manifold, other sums appear in the formula for the analytic torsion. An interpretation of these further sums, collected into the two terms $B_1$ and $B_2$ in Theorem \ref{t1}, is the second step of our interpretation of the formula for the analytic torsion of the cone, and is presented in the final section of the paper. The interpretation is suggested  after proving that the formula for the torsion of the cone can be obtained as some limit case of the formula for the torsion of the frustum. We postpone the discussion of these results to the last section. 

We conclude this introductory section with a few remarks on the anomaly term appearing in the even case.
We observe that the anomaly term can be also written in terms of residues of some spectral functions, see Subsection \ref{singconesing}, or Theorem 1.1 of \cite{HS1}. Formulas for particular sections, as spheres of discs, are given in \cite{HS1} and \cite{HMS} \cite{Har}.

\begin{rem} By duality:
\[
\frac{1}{2}\sum^{p-1}_{q=0} (-1)^q \sum_{n=1}^\infty m_{{\rm cex},q,n}\log\frac{1+\frac{\alpha_q}{\mu_{q,n}}}{1-\frac{\alpha_q}{\mu_{q,n}}}
=\frac{1}{2}\sum^{2p-1}_{q=0} (-1)^q \sum_{n=1}^\infty m_{{\rm cex},q,n}\log\left(1+\frac{\alpha_q}{\mu_{q,n}}\right).
\]

\end{rem}

\begin{rem} Using Proposition 2.9 of \cite{Spr4} the first term in the last line of the second formula in the theorem reads
\begin{align*}
\frac{1}{2}\sum^{p-1}_{q=0} (-1)^q \sum_{n=1}^\infty m_{{\rm cex},q,n}\log\frac{1+\frac{\alpha_q}{\mu_{q,n}}}{1-\frac{\alpha_q}{\mu_{q,n}}}
&=\frac{1}{2}\sum^{p-1}_{q=0} (-1)^q\left(\zeta_q'(0,\alpha_q)-\zeta_q'(0,-\alpha_q)\right),\\
&=\frac{1}{2}\sum^{p-1}_{q=0} (-1)^q\log\frac{\det_\zeta \left((\Delta_{(W,g)}^{(q)}+\alpha_q^2)^\frac{1}{2}-\alpha_q\right)}{\det_\zeta\left((\Delta_{(W,g)}^{(q)}+\alpha_q^2)^\frac{1}{2}+\alpha_q\right)},
\end{align*}
where $\Delta_{(W,g)}$ is the Hodge-Laplace operator on the section of the cone $(W,g)$, and
\[
\zeta_q(s,x)=\sum_{n=1}^\infty m_{{\rm cex},q,n} (\mu_{q,n}+x)^{-s}.
\]

\end{rem}

\section{A geometric interpretation of the analytic torsion}
\label{sgeo}

The spaces of harmonic forms on the frustum (with absolute an mixed BC) were computed in \cite{HS3}. The space of harmonic forms on the cone (with absolute and relative BC) were compute \cite{HS4}. In particular,  in the even case $m=2p$ the result changes if we assume the Cheeger ideal BC, as described in the introduction. A simple calculation, proceeding as in the proof of Lemma 4.1 of \cite{HS4}, gives the following result.

\begin{lem} We have the following isomorphisms of vector spaces induced by the extension of the inclusion of the forms:  
\begin{align*}
\H^q_{\rm abs}(C_l (W))&=\begin{cases}\H^q(W), &0\leq q\leq p-1,\\
 \{0\}, & p\leq q\leq 2p-1,\end{cases}&\dim W=2p-1,\\
\H^q_{\rm abs}(C_l (W))&=\begin{cases}\H^q(W), &0\leq q\leq p,\\
 \{0\}, & p+1\leq q\leq 2p+1,\end{cases}&\dim W=2p,\\
\H^q_{\rm abs, ideal}(C_l W)&=
\begin{cases} 
\H^q (W) & {\text{if}}\; 0\leq q\leq p-1,\\
V_{\rm a} & {\text{if}}\; q = p,\\
0 & {\text{if}}\; p+1\leq q\leq 2p+1,
\end{cases}&\dim W=2p,\\
\H^q_{\rm abs}(C_{[l_1,l_2]} (W))&=\H^q(W).
\end{align*}
\end{lem}

Proceeding as in Section 3.5 of \cite{HS3}, we have a the following commutative diagram of isomorphisms of vector spaces (we give here the diagram for the cone, the one for the frustum is in \cite{HS3}), for $0\leq q\leq \left[\frac{m}{2}\right]$,

\centerline{
\xymatrix{\H^q_{\rm abs}(C_l (W))\ar[r]^{\star_{C_l(W)}}&\H^{m-q+1}_{\rm rel}(C_l(W))\ar[r]^{\AA_{C_l(W), {\rm rel}}^{m-q+1}}&H^{m-q+1}(\widehat{C_l (W)}, \widehat{\b C_l (W)})&H_q(C_l(W))\ar[l]_{\hspace{40pt}\P}\\
\H^q( W)\ar[r]^{\star_{W}}\ar[u]_{\frac{\ddot{(\_)}}{\sqrt{\gamma_q}}}&
\H^{m-q}( W)\ar[r]^{\AA^{m-q}_{W}}\ar[u]_{\frac{(-1)^q x^{m-2q}dx\wedge \_}{\sqrt{\gamma_q}}}&H^{m-q}(\widehat{ W})\ar[u]&H_q( W)\ar[l]_{\P}\ar[u]_{(-1)^q\sqrt{\gamma_q}}
}}
\noindent where $\AA$ is the de Rham map on the cone, $\P$ Poincar\`e duality, $\star$ the Hodge isomorphism, and the factor $\gamma_q$ is the ratio between the $L^2$ norm of the constant extension of a form $\omega$ on the cone and the $L^2$ norm of $\omega$ (the double dot indicates the costant extension):
\[
\gamma_q=\frac{\|\ddot\omega_{C_l}\|^2_{C_l(W)}}{\|\omega\|^2_W}=\int_0^l x^{m-2q} dx=\frac{l^{m-2q+1}}{m-2q+1}.
\]

The same analysis on the frustum gives (see \cite[Section 3.5]{HS3} for all definitions and details)
\[
\Gamma_q=\frac{\|\ddot\omega_{F}\|^2_{C_{[l_1,l_2]}(W)}}{\|\omega\|^2_W}=\int_{l_1}^{l_2} x^{m-2q} dx=\left\{
\begin{array}{cc}
\frac{1}{m+1-2q}\left(l_2^{m+1-2q}-l_1^{m+1-2q}\right)& {\rm if} \; m+1-2q\neq0, \\
\ln\frac{l_2}{l_1}, & {\rm if} \; m+1-2q=0.
\end{array}
\right.
\]

Note that for all $q< \left[\frac{m}{2}\right]$,
\[
\lim_{l_1\to 0^+, l_2=l} \Gamma_q=\gamma_q.
\]

Recall that for  a real vector space $V$ we denote $\Lambda^{\dim V} V$ by $\det V$, and we call this line the  determinant line of $V$ ($\det 0:=\R$). 
If $\vs=\{v_1, \dots, v_{\dim V}\}$ is a basis for $V$, we use the notation $\det \vs$ for $v_1\wedge\dots\wedge v_{\dim V}$. For a finite dimensional graded vector space $V_\bu=( V_q)_{q\in \Z}=\bigoplus_{q\in \Z} V_q$, set
\[
\Det^{m}_n V_\bu:=\bigotimes_{q=n}^m (\det V_q)^{\da^q},
\]
where $V^{\da^{2p}}=V$ and $V^{\da^{2p+1}}=V^\da$. This one dimensional vector space is called the determinant line of $V_\bu$. 

With this notation, if $\alphas$ is a (graded) basis for the harmonic forms on $W$, we have that the ratio between the determinant of the homology graded basis  induced by an harmonic basis of the frustum and of its section, and of the cone and of this section are, respectively:
\begin{align*}
\frac{{\rm Det} \ddot\alphas_F}{{\rm Det} \alphas}&=\prod_{q=0}^m \Gamma_q^{\frac{(-1)^q r_q}{2}},& &{\rm absolute~ BC}\\
\frac{{\rm Det}^{\left[\frac{m-1}{2}\right]}_0 \ddot\alphas_{C_l}}{{\rm Det}^{\left[\frac{m-1}{2}\right]}_0 \alphas}&=\prod_{q=0}^{\left[\frac{m-1}{2}\right]} \gamma_q^{\frac{(-1)^q r_q}{2}}, & &{\rm absolute~ BC}\\
\frac{{\rm Det}^{p-1}_0 \ddot\alphas_{C_l}}{{\rm Det}^{p-1}_0 \alphas}&=\gamma_p^{\frac{(-1)^p r_p}{4}}\prod_{q=0}^{p-1} \gamma_q^{\frac{(-1)^q r_q}{2}}, & m=2p,~& {\rm absolute, ideal~ BC}.
\end{align*}

For completeness we gave here also the result without ideal BC here, however in the following the formulas are all given assuming ideal BC. A simple calculation completes the proof of the following corollary of Theorem \ref{t0} (a similar formula holds in the case $m=0$).

\begin{theo}\label{t1} The  analytic torsion of the cone reads:
\begin{align*}
\log T_{\rm abs, ideal}(C_l(W^{(m)}))=&\frac{1}{2}\log T(W,g)+\log \frac{{\rm Det}^{\left[\frac{m-1}{2}\right]}_0 \ddot\alphas}{{\rm Det}^{\left[\frac{m-1}{2}\right]}_0 \alphas}+\frac{1}{4}\chi(W)\log 2+A_{\rm BS,abs}(\b C_l(W))\\
&+B^{(m)}_1(C_l(W))+B^{(m)}_2(C_l(W)),
\end{align*}
where $B^{(m)}_{1/2}(C_l(W)$ are the following anomaly terms (vanishing when $m=2p-1$ is odd):
\begin{align*}
B^{(2p)}_1(C_l(W))=& - \sum_{q=0}^{p-1} (-1)^{q}  r_q\log (2p-2q-1)!!+\frac{1}{2}\sum^{p-1}_{q=0} (-1)^q \sum_{n=1}^\infty m_{{\rm cex},q,n}\log\frac{1+\frac{\alpha_q}{\mu_{q,n}}}{1-\frac{\alpha_q}{\mu_{q,n}}},\\
B^{(2p)}_2(C_l(W))=&\frac{1}{4}\chi(W)\log 2.
\end{align*}

The  analytic torsion of the frustum reads:
\[
\log T_{\rm abs}(C_{[l_1,l_2]}(W))=\log T(W,g)+\log \frac{{\rm Det} \ddot\alphas}{{\rm Det} \alphas}+\frac{1}{2}\chi(W)\log 2+A_{\rm BS,abs}(\b C_{[l_1,l_2]}(W)).
\]

\end{theo}

\begin{rem}\label{euler} As announced in the introduction, the formula in Theorem \ref{t1} completely describes the analytic torsion of a cone over an odd dimensional manifold in terms of topological and geometric quantities. In particular, the splitting between global and boundary terms is evident (compare with the formula in equation (\ref{e000}) for a regular manifold, and with the discussion in Section \ref{s4}). In the even case, the same analysis suggests the description of the torsion given in the formula, with two new anomaly terms, called $B_1$ and $B_2$. If now we read the formula for the torsion in the even case looking for the global and the boundary contributions, we find out that part of the boundary contribution can be interpret as anomalous. 
This is justified as follow. Recall that the boundary term  splits into two parts: one is the classical part, see the work of L\"{u}ck \cite{Luc} and depends on the Euler characteristic of the boundary, the other is the anomaly boundary term, computed by Br\"{u}ning and Ma \cite{BM1,BM2}. Then, we realise that (collecting together all the terms where the Euler characteristic appears) the classical boundary term in the cone coincides with the one of the frustum, even if the boundary is different. In other words, it seems that the analytic torsion of the cone does not see that the boundary at $x=0$ collapses to a point. Following this interpretation, we split this boundary term into two parts: the first is
\[
\frac{1}{4}\chi(W)\log 2,
\]
and is the classical contribution of the boundary, the second one is $B_2(W)$, and is understood as an anomaly boundary term.
\end{rem}

\section{Spectral properties of the Hodge-Laplace operator}
\label{s1}

Consider the metric in Definition \ref{metric} either on the cone $C_l(W)$ or on the frustum $C_{[l_1>0,l_2]}(W)$. Let  $\omega\in \Omega^{q}(C_l (W))$ ( $\omega\in \Omega^{q}(C_{[l_1>0,l_2]}(W))$), set
\[
\omega(x,y)=f_1(x)\omega_1(y)+f_2(x)dx\wedge \omega_2(y),
\]
with smooth functions $f_1$ and $f_2$, and $\omega_j\in \Omega( W)$. Then (we denote operators acting on the section by a tilde),
\begin{align*}
\label{f1}\star \omega(x,y)&= x^{m-2q+2} f_2(x)\tilde\star \omega_2(y)+(-1)^q x^{m-2q}f_1(x) dx\wedge\tilde\star \omega_1(y),
\end{align*}
\[
\begin{aligned}d \omega(x,y)   &= f_1(x)\tilde d \omega_1(y) + \b_x f_1(x) dx \wedge \omega_1(y) - f_2(x) dx \wedge d\omega_2(y),\\
d^\dagger \omega(x,y)&= x^{-2} f_1(x)\tilde d^{\dag}\omega_1 (y) -\left((m-2q+2)x^{-1}f_2(x) + \b_x f_2(x)\right)\omega_2(y)\\
&-x^{-2}f_2(x) dx \wedge \tilde d^{\dag} \omega_2(y),
\end{aligned}
\]

\begin{align*}
\Delta\omega(x,y)&= \left(-\b_x^2 f_1(x) -(m-2q)x^{-1}\b_x f_1(x)\right)\omega_1(y) + x^{-2}f_1(x)\tilde
\Delta\omega_1(y)-2x^{-1}f_2(x)\tilde d\omega_2(y)\\
&+dx\wedge \left(x^{-2}f_2(x) \tilde\Delta\omega_2(y)+\omega_2(y)\left(-\b^2_x f_2(x) -(m-2q+2)x^{-1}\b_x f_2(x)\right.\right.\\
&\left.\left. + (m-2q+2)x^{-2}f_2(x) \right) -2x^{-3}f_1(x)\tilde d^{\dag}\omega_1(y)\right).
\end{align*}

\begin{lem}[Cheeger \cite{Che2}]\label{l2}
Let $\{\varphi_{{\rm har}}^{(q)},\varphi_{{\rm cex},n}^{(q)},\varphi_{{\rm ex},n}^{(q)}\}$ be an orthonormal  basis of $\Omega^q(W)$
consisting of harmonic,  coexact and exact eigenforms of $\tilde\Delta^{(q)}$. Let $\lambda_{q,n}$
denotes the  eigenvalue of $\varphi_{{\rm cex},n}^{(q)}$ and $m_{{\rm cex},q,n}$ its multiplicity. 
Let $J_\nu$ be the Bessel function of index $\nu$. Define 
\begin{align*}
\alpha_q &=\frac{1}{2}(1+2q-m), &\mu_{q,n} &= \sqrt{\lambda_{q,n}+\alpha_q^2}.
\end{align*}

Then, assuming that $\mu_{q,n}$ is not an integer,  the solutions of the equation $\Delta u=\lambda^2 u$,
with $\lambda\not=0$, are  of the following six types:
\begin{align*}
\psi^{(q)}_{\pm, 1,n,\lambda} =& x^{\alpha_q} J_{\pm\mu_{q,n}}(\lambda x) \varphi_{{\rm cex},n}^{(q)},\\
\psi^{(q)}_{\pm, 2,n,\lambda} =& x^{\alpha_{q-1}} J_{\pm\mu_{q-1,n}}(\lambda x) \tilde d\varphi_{{\rm cex},n}^{(q-1)} +
\b_x(x^{\alpha_{q-1}}J_{\pm\mu_{q-1,n}}(\lambda x)) dx \wedge \varphi_{{\rm cex},n}^{(q-1)}\\
\psi^{(q)}_{\pm,3,n,\lambda} =& x^{2\alpha_{q-1}+1}\b_x(x^{-\alpha_{q-1}}J_{\pm\mu_{q-1,n}}(\lambda x) )\tilde d\varphi_{{\rm cex},n}^{(q-1)}\\
\nonumber &+ x^{\alpha_{q-1}-1}J_{\pm\mu_{q-1,n}}(\lambda x) dx \wedge \tilde d^{\dag} \tilde d \varphi_{{\rm cex},n}^{(q-1)}\\
\psi^{(q)}_{\pm,4,n,\lambda} =& x^{\alpha_{q-2}+1}J_{\pm\mu_{q-2,n}}(\lambda x) dx \wedge \tilde d \varphi_{{\rm cex},n}^{(q-2)}\\
\psi^{(q)}_{\pm,E,\lambda} =& x^{\alpha_{q}} J_{\pm|\alpha_{q}|}(\lambda x) \varphi_{{\rm har}}^{(q)}\\
\psi^{(q)}_{\pm,O,\lambda} =& \b_x(x^{\alpha_{q-1}}J_{\pm|\alpha_{q-1}|}(\lambda x)) dx \wedge \varphi_{{\rm har}}^{(q-1)}.
\end{align*}

When the  index is an integer the $-$ solutions must be modified including some logarithmic term (see for example \cite{Wat} for a set of linear independent solutions of the Bessel equation).
\end{lem}

Following \cite{Che2} and \cite{BS2}, the formal Hodge-Laplace operator in equation (\ref{l3}) defines a concrete self adjoint operator with domain in the space of the  square integrable forms on the cone (see \cite{HS1} for details). This operator (that we denote by the same symbol $\Delta$) has a pure point spectrum as described in the following lemmas, whose proof follows applying BC and square integrability to the solutions described in Lemma \ref{l2} (see \cite{HS4} for details on the proof, and observe that more care is necessary in the even dimensional case, where we must take into account also the ideal BC). 

\begin{lem}\label{l3} The positive part of the spectrum of the Hodge-Laplace operator  on $C_l(W)$, with
absolute boundary conditions on $\b C_l (W)$ is as follows, where $0\leq q\leq m+1$. If $m=\dim W=2p-1$, $p\geq 1$:
\begin{align*}
\Sp_+ \Delta_{\rm abs}^{(q)} &= \left\{m_{{\rm cex},q,n} : \hat j^{2}_{\mu_{q,n},\alpha_q,k}/l^{2}\right\}_{n,k=1}^{\infty}
\cup
\left\{m_{{\rm cex},q-1,n} : \hat j^{2}_{\mu_{q-1,n},\alpha_{q-1},k}/l^{2}\right\}_{n,k=1}^{\infty} \\
&\cup \left\{m_{{\rm cex},q-1,n} : j^{2}_{\mu_{q-1,n},k}/l^{2}\right\}_{n,k=1}^{\infty} \cup \left\{m _{{\rm cex},q-2,n} :
j^{2}_{\mu_{q-2,n},k}/l^{2}\right\}_{n,k=1}^{\infty} \\
&\cup \left\{m_{{\rm har},q}:\hat j^2_{|\alpha_q|,\alpha_q,k}/l^{2}\right\}_{k=1}^{\infty} \cup \left\{ m_{{\rm har},q-1}:\hat
j^2_{|\alpha_{q-1}|,\alpha_q,k}/l^{2}\right\}_{k=1}^{\infty}.
\end{align*}

If $m=\dim W=2p$, $p\geq 1$:
\begin{align*}
\Sp_+ \Delta_{\rm abs, ideal}^{(q\not= p, p+1)} &= \left\{m_{{\rm cex},q,n} : \hat j^{2}_{\mu_{q,n},\alpha_q,k}/l^{2}\right\}_{n,k=1}^{\infty}
\cup
\left\{m_{{\rm cex},q-1,n} : \hat j^{2}_{\mu_{q-1,n},\alpha_{q-1},k}/l^{2}\right\}_{n,k=1}^{\infty} \\
&\cup \left\{m_{{\rm cex},q-1,n} : j^{2}_{\mu_{q-1,n},k}/l^{2}\right\}_{n,k=1}^{\infty} \cup \left\{m _{{\rm cex},q-2,n} :
j^{2}_{\mu_{q-2,n},k}/l^{2}\right\}_{n,k=1}^{\infty} \\
&\cup \left\{m_{{\rm har},q}:\hat j^2_{|\alpha_q|,\alpha_q,k}/l^{2}\right\}_{k=1}^{\infty} \cup \left\{ m_{{\rm har},q-1}:\hat j^2_{|\alpha_{q-1}|,\alpha_{q-1},k}/l^{2}\right\}_{k=1}^{\infty},
\end{align*}
\begin{align*}
\Sp_+ \Delta_{\rm abs, ideal}^{(p)} &= \left\{m_{{\rm cex},p,n} : \hat j^{2}_{\mu_{p,n},\alpha_p,k}/l^{2}\right\}_{n,k=1}^{\infty}
\cup
\left\{m_{{\rm cex},p-1,n} : \hat j^{2}_{\mu_{p-1,n},\alpha_{p-1},k}/l^{2}\right\}_{n,k=1}^{\infty} \\
&\cup \left\{m_{{\rm cex},p-1,n} : j^{2}_{\mu_{p-1,n},k}/l^{2}\right\}_{n,k=1}^{\infty} \cup \left\{m _{{\rm cex},p-2,n} :
j^{2}_{\mu_{p-2,n},k}/l^{2}\right\}_{n,k=1}^{\infty} \\
&\cup \left\{\frac{1}{2}m_{{\rm har},p}: j^2_{\frac{1}{2}}/l^{2}\right\}_{k=1}^{\infty} \cup \left\{\frac{1}{2}m_{{\rm har},p}: j^2_{-\frac{1}{2}}/l^{2}\right\}_{k=1}^{\infty}\cup \left\{ m_{{\rm har},p-1}:\hat j^2_{|\alpha_{p-1}|,\alpha_{p-1},k}/l^{2}\right\}_{k=1}^{\infty},
\end{align*}
\begin{align*}
\Sp_+ \Delta_{\rm abs, ideal}^{(p+1)} &= \left\{m_{{\rm cex},p+1,n} : \hat j^{2}_{\mu_{p+1,n},\alpha_{p+1},k}/l^{2}\right\}_{n,k=1}^{\infty}
\cup
\left\{m_{{\rm cex},p,n} : \hat j^{2}_{\mu_{p,n},\alpha_{p},k}/l^{2}\right\}_{n,k=1}^{\infty} \\
&\cup \left\{m_{{\rm cex},p,n} : j^{2}_{\mu_{p,n},k}/l^{2}\right\}_{n,k=1}^{\infty} \cup \left\{m _{{\rm cex},p-1,n} :
j^{2}_{\mu_{p-1,n},k}/l^{2}\right\}_{n,k=1}^{\infty} \\
&\cup \left\{ m_{{\rm har},p+1}:\hat j^2_{|\alpha_{p+1}|,\alpha_{p+1},k}/l^{2}\right\}_{k=1}^{\infty} \cup \left\{\frac{1}{2}m_{{\rm har},p}: 
j^2_{-\frac{1}{2}}/l^{2}\right\}_{k=1}^{\infty} \cup \left\{\frac{1}{2}m_{{\rm har},p}: j^2_{\frac{1}{2}}/l^{2}\right\}_{k=1}^{\infty},
\end{align*}
where  the $j_{\mu,k}$ are the positive zeros of the Bessel function $J_{\mu}(x)$,  the $\hat j_{\mu,c,k}$ are the positive zeros of
the function $\hat J_{\mu,c}(x) = c J_\mu (x) + x J'_\mu(x)$,  $c\in \R$, $\alpha_q$ and $\mu_{q,n}$ are defined in Lemma \ref{l2}.
\end{lem}

\begin{lem}\label{l3b}
The positive part of the spectrum of the Hodge-Laplace operator  on $C_{[l_1,l_2]}(W)$, $l_1>0$, with
relative boundary conditions on $\b_1 C_{[l_1,l_2]}(W)$ and absolute boundary conditions on $\b_2 C_{[l_1,l_2]}(W)$ is ($0\leq q\leq m+1$):
\begin{align*}
\Sp_+ \Delta_{\rm rel~\b_1,abs~\b_2}^{(q)}  &= \left\{m_{{\rm cex},q,n} : \hat f^{2}_{\mu_{q,n},\alpha_q,k}(l_1,l_2)\right\}_{n,k=1}^{\infty}
\cup
\left\{m_{{\rm cex},q-1,n} : \hat f^{2}_{\mu_{q-1,n},\alpha_{q-1},k}(l_1,l_2)\right\}_{n,k=1}^{\infty} \\
&\cup \left\{m_{{\rm cex},q-1,n} : \hat f^{2}_{\mu_{q-1,n},-\alpha_{q-1},k}(l_2,l_1)\right\}_{n,k=1}^{\infty} \cup \left\{m _{{\rm cex},q-2,n} :
\hat f^{2}_{\mu_{q-2,n},-\alpha_{q-2},k}(l_2,l_1)\right\}_{n,k=1}^{\infty} \\
&\cup \left\{m_{{\rm har},q}:\hat f^2_{|\alpha_q|,\alpha_q,k}(l_1,l_2)\right\}_{k=1}^{\infty} \cup \left\{ m_{{\rm har},q-1}:\hat
f^2_{|\alpha_{q-1}|,\alpha_{q-1},k}(l_1,l_2)\right\}_{k=1}^{\infty},
\end{align*}
where  the $\hat f_{\mu,c,k}(a,b)$ are the zeros of the function 
\[
\hat F_{\mu,c} (x;l_1,l_2)=
J_{\mu}(l_1 x) (c Y_{\mu}(l_2 x) + l_2 x Y'_{\mu}(l_2 x)) - Y_{\mu}(l_1 x) (c J_{\mu}(l_2 x) + l_2 x J'_{\mu}(l_2 x)) ,
\]
with real $c\not=0$, and $\alpha_q$ and $\mu_{q,n}$ as defined in Lemma \ref{l2}.
\end{lem}

\begin{rem} Application of the BC in Lemma \ref{l2} would give the zeros of the function
\[
J_{\mu}(l_1 x) (c J_{-\mu}(l_2 x) + l_2 x J'_{-\mu}(l_2 x)) - J_{-\mu}(l_1 x) (c J_{\mu}(l_2 x) + l_2 x J'_{\mu}(l_2 x))
=-\sin(\pi\mu)  \hat F_{\mu,c} (x;l_1,l_2),
\]
however, for the following analysis, it is much more convenient to work with the function $Y$ instead that with the function $J_-$.
\end{rem}

\begin{lem}\label{al1}
The positive part of the spectrum of the Hodge-Laplace operator  on $C_{[l_1,l_2]}(W)$, $l_1>0$, with absolute boundary conditions is ($0\leq q\leq m+1$):
\begin{align*}
\Sp_+ \Delta_{\rm abs}^{(q)}  &= \left\{m_{{\rm cex},q,n} : \hat v^{2}_{\mu_{q,n},\alpha_q,k}\right\}_{n,k=1}^{\infty}
\cup
\left\{m_{{\rm cex},q-1,n} : \hat v^{2}_{\mu_{q-1,n},\alpha_{q-1},k}\right\}_{n,k=1}^{\infty} \\
&\cup \left\{m_{{\rm cex},q-1,n} : v^{2}_{\mu_{q-1,n},k}\right\}_{n,k=1}^{\infty} \cup \left\{m _{{\rm cex},q-2,n} :
v^{2}_{\mu_{q-2,n},k}\right\}_{n,k=1}^{\infty} \\
&\cup \left\{m_{{\rm har},q}:\hat v^2_{|\alpha_q|,\alpha_q,k}\right\}_{k=1}^{\infty} \cup \left\{ m_{{\rm har},q-1}:\hat
v^2_{|\alpha_{q-1}|,\alpha_{q-1},k}\right\}_{k=1}^{\infty},
\end{align*}
where  the $v_{\mu,k}$ are the zeros of the function 
\[
\Upsilon_{\mu} (x)=
J_{\mu}(l_2x) Y_{\mu}(l_1 x) - Y_{\mu}(l_1 x) J_{\mu}(l_2 x),
\]
and   the $\hat v_{\mu,c,k}$ are the zeros of the function 
\[
\hat \Upsilon_{\mu,c} (x)=(c J_{\mu}(l_2 x) + l_2 x J'_{\mu}(l_2 x))
 (c Y_{\mu}(l_1 x) + l_1 x Y'_{\mu}(l_1 x)) -  (c Y_{\mu}(l_2 x) + l_2 x Y'_{\mu}(l_2 x)) (c J_{\mu}(l_1 x) + l_1 x J'_{\mu}(l_1 x)),
\]
with real $c\not=0$, and $\alpha_q$ and $\mu_{q,n}$ as defined in Lemma \ref{l2}.
\end{lem}

\section{Simplifying the torsion zeta function}
\label{s2}

We define the  {\it torsion zeta function} by
\[
t_{M^{(n)}}(s)=\frac{1}{2}\sum_{q=1}^{n} (-1)^q q \zeta(s,\Delta^{(q)}),
\]
and the  analytic torsion is:
\[
\log T(M^{(n)},g)=t_{M^{(n)}}'(0).
\]

In this section, we consider the torsion zeta function   for the cone with absolute BC and for the frustum with mixed and absolute BC respectively. We proceed to some simplifications of it. We present the proof in the case of the cone, the proof for the frustum is analogous.

\begin{lem}\label{l000} The torsion zeta function of the cone is:
\[
t_{\rm Cone}^{(m)}(s)=t_0^{(m)}(s)+t_1^{(m)}(s)+t^{(m)}_2(s)+t^{(m)}_3(s),
\]
with
\begin{align*}
t_0^{(m)}(s)&= \frac{l^{2s}}{2}\sum^{\left[\frac{m}{2}\right]-1}_{q=0} (-1)^q \left((Z_q(s)-\hat Z_{q,+}(s))
+(-1)^{m-1}(Z_q(s)-\hat Z_{q,-}(s))\right),\\
t_1^{(2p-1)}(s)&= (-1)^{p-1}\frac{l^{2s}}{2}  \left(  Z_{p-1}(s)-\hat Z_{p-1,0}(s) \right),&t_1^{(2p)}(s)&=0\\
t_2^{(m)}(s)&= \frac{l^{2s}}{2} \sum_{q=0}^{\left[\frac{m-1}{2}\right]} (-1)^{q+1}  m_{{\rm har},q}\left(z_{q-1,-}(s)+(-1)^m z_{q,-}(s)\right),\\
t_3^{(2p)}(s)&= (-1)^{p+1} m_{{\rm har},p}\frac{l^{2s}}{4}\left(z_{p,+}(s)+z_{p,-}(s)\right),&t_3^{(2p-1)}(s)&=0,
\end{align*}
where $\hat Z_{q,0}$ denotes $\hat Z_{q,\pm}$ with $\alpha_q=0$, and
\begin{align*}
Z_q(s)&=\sum^{\infty}_{n,k=1} m_{{\rm cex},q,n}j^{-2s}_{\mu_{q,n},k},&
\hat Z_{q,\pm}(s)&=\sum^{\infty}_{n,k=1} m_{{\rm cex},q,n}\hat j^{-2s}_{\mu_{q,n},\pm\alpha_q,k},
&z_{q,\pm}(s)&=\sum_{k=1}^{\infty} j^{-2s}_{\pm\alpha_{q},k}.
\end{align*}

\end{lem}
\begin{proof} Rearranging the sums and isolating the case $q=p$, $\alpha_p=\frac{1}{2}$ when $m=2p$, we obtain:
if $m=2p-1$,
\begin{align*}
t_{\rm Cone}^{(2p-1)}(s)=& \frac{l^{2s}}{2} \sum^{2p-1}_{q=0} (-1)^q \sum^{\infty}_{n,k=1} m_{{\rm cex},q,n}\left( j^{-2s}_{\mu_{q,n},k} - \hat j^{-2s}_{\mu_{q,n},\alpha_q,k}\right)+\frac{l^{2s}}{2}\sum^{2p-1}_{q=0} (-1)^{q+1} \sum_{k=1}^\infty m_{{\rm har},q} \hat  j_{|\alpha_q|,\alpha_q,k}^{-2s},
\end{align*}
if $m=2p$, 
\begin{align*}
t_{\rm Cone}^{(2p)}(s)=& \frac{l^{2s}}{2} \sum^{2p-1}_{q=0} (-1)^q \sum^{\infty}_{n,k=1} m_{{\rm cex},q,n}\left( j^{-2s}_{\mu_{q,n},k} - \hat j^{-2s}_{\mu_{q,n},\alpha_q,k}\right)+\frac{l^{2s}}{2}\sum^{2p}_{q=0,q\not = p} (-1)^{q+1} \sum_{k=1}^\infty m_{{\rm har},q} \hat j_{|\alpha_q|,\alpha_q,k}^{-2s}\\
&+ (-1)^{p+1} \frac{l^{2s}}{4}\sum_{k=1}^{\infty} m_{{\rm har},p}\left(j_{\frac{1}{2},k}^{-2s}+j_{-\frac{1}{2},k}^{-2s}\right).
\end{align*}

Using  Hodge duality 
on coexact $q$-forms on the section, we obtain the following identities:
\begin{align*}
m_{{\rm cex},q,n}&= m_{{\rm cex},m-1-q,n},\\
\lambda_{q,n} & = \lambda_{m-1-q,n},\\ 
\alpha_{q} & =-\alpha_{m-1-q},\\
\mu_{q,n} &=\mu_{m-1-q,n},
\end{align*}
and hence the first term in the previous equations reads:
\begin{align}
\nonumber\frac{l^{2s}}{2} \sum^{m-1}_{q=0} (-1)^q &\sum^{\infty}_{n,k=1} m_{{\rm cex},q,n}\left( j^{-2s}_{\mu_{q,n},k} - \hat j^{-2s}_{\mu_{q,n},\alpha_q,k}\right)\\
\label{main}=& \frac{l^{2s}}{2}\sum^{\left[\frac{m}{2}\right]-1}_{q=0} (-1)^q \sum^{\infty}_{n,k=1} m_{{\rm cex},q,n} 
\left(\left( j^{-2s}_{\mu_{q,n},k} - \hat j^{-2s}_{\mu_{q,n},\alpha_q,k}\right)
+(-1)^{m-1}\left( j^{-2s}_{\mu_{q,n},k} - \hat j^{-2s}_{\mu_{q,n},-\alpha_q,k}\right)\right)\\
\nonumber&+\left({\rm term~with~}q=\left[\frac{m}{2}\right]\right),
\end{align}
where the last term appears only if $\left[\frac{m}{2}\right]$ is an integer, namely if $m=2p-1$ is odd, and has the following form. Since when $q=p-1$, $m=2p-1$, $\alpha_{p-1}=0$, (and $j_{\nu,0,k}=\lambda_{p-1,n}=j'_{\nu,k}$), then
\[
\left({\rm term~with~}q=\left[\frac{m}{2}\right]\right)= (-1)^{p-1}\frac{l^{2s}}{2} \sum^{\infty}_{n,k=1} m_{{\rm cex},p-1,n} \left( j^{-2s}_{\mu_{p-1,n},k} -\hat j_{\mu_{p-1,n},0,k}^{-2s} \right).
\]

The sign in front to the second term in equation (\ref{main}) is the key difference between the even and the odd case.

Next consider the term involving the harmonics, i.e. the sums
\begin{align*}
k^{(2p-1)}(s)&=\frac{1}{2}\sum^{2p-1}_{q=0} (-1)^{q+1} \sum_{k=1}^\infty m_{{\rm har},q} \hat j_{|\alpha_q|,\alpha_q,k}^{-2s},&
k^{(2p)}(s)&=\frac{1}{2}\sum^{2p}_{q=0,q\not=p} (-1)^{q+1} \sum_{k=1}^\infty m_{{\rm har},q} \hat j_{|\alpha_q|,\alpha_q,k}^{-2s}.
\end{align*} 

Consider the function 
\[
\hat J_{|\alpha_q|,\alpha_q} (x)=
\alpha_q J_{|\alpha_q|}( x) + l x J'_{|\alpha_q|}(x) ,
\]

Since
\[
zZ'_{\mu}(z) = -z Z_{\mu+1}(z) + \mu Z_\mu(z),\;\;{\rm and}\;\;zZ'_{\mu}(z) = z Z_{\mu-1}(z) - \mu Z_\mu(z),
\] 
where $Z$ is either $J_+$ or $J_-$, it follows that 
\begin{align*}
\hat Z_{-\alpha_q,\alpha_q}(z) &:= \alpha_q Z_{-\alpha_q}(z) +z Z'_{-\alpha_q}(z)= -z Z_{-\alpha_q + 1}(z) =- zZ_{-\alpha_{q-1}}(z),\\
\hat Z_{\alpha_q,\alpha_q}(z) &:= \alpha_q Z_{\alpha_q}(z) +z Z'_{\alpha_q}(z)=zZ_{\alpha_q -1}(z)=zZ_{\alpha_{q -1}}(z).
\end{align*}

This permit to simplify  $\hat J_{|\alpha_q|,\alpha_q}$ as follows. If  $\alpha_q$ is negative,  then 
\begin{align*}
\hat J_{|\alpha_q|,\alpha_q} (x)
&=-x J_{-\alpha_{q-1}}(x),
\end{align*}
and this means that $\hat j_{|\alpha_{q}|,\alpha_q,k} =  j_{-\alpha_{q-1},k}$, for these $q$. If  is positive, then
\begin{align*}
\hat J_{|\alpha_q|,\alpha_q} (x)
&=x J_{\alpha_{q-1}}(x),
\end{align*}
and this means that $\hat j_{|\alpha_{q}|,\alpha_q,k} =  j_{\alpha_{q-1},k}$, for these $q$. When $\alpha_q=0$ (and this happens only if $m=2p-1$, $q=p-1$), we have $\hat J_{0,0} (x)=x J_{\pm 1}(x)$, and hence we can use either $j_{1,k}$ or $j_{-1,k}$. Next, if $m=2p-1$, $\alpha_q$ is negative for $0\leq q\leq p-2$, $\alpha_{p-1}=0$, and $\alpha_q$ is positive for $p\leq q\leq 2p-1$, whence
\begin{align*}
k^{(2p-1)}(s)=&\frac{1}{2} \sum_{q=0}^{p-2}
(-1)^{q+1}\sum_{k=1}^{\infty}\frac{m_{{\rm har},q}}{ j^{2s}_{-\alpha_{q-1},k}} + \frac{1}{2}(-1)^{p}
\sum_{k=1}^{\infty}\frac{m_{{\rm har},1}}{ j^{-2s}_{\alpha_{1,k}}} +\frac{1}{2}\sum_{q=p}^{2p-1}(-1)^{q+1}\sum_{k=1}^{\infty}\frac{m_{{\rm har},q}}{ j^{-2s}_{\alpha_{q-1},k} },
\end{align*}
if $m=2p$, $\alpha_q$ is negative for $0\leq q\leq p-1$,  and $\alpha_q$ is positive for $p\leq q\leq 2p-1$, whence
\begin{align*}
k^{(2p)}(s)=& \frac{1}{2}\sum_{q=0}^{p-1}
(-1)^{q+1}\sum_{k=1}^{\infty}\frac{m_{{\rm har},q}}{ j^{2s}_{-\alpha_{q-1},k}}  +\frac{1}{2}\sum_{q=p}^{2p}(-1)^{q+1}\sum_{k=1}^{\infty}\frac{m_{{\rm har},q}}{ j^{-2s}_{\alpha_{q-1},k} },
\end{align*}
and since by Poincarr\'e duality $m_{{\rm har},q}=m_{{\rm har},m-q}$, and $\alpha_{m-q}=-\alpha_{q-1}$, we have the thesis (note that when $m=2p$, $\alpha_p=\frac{1}{2}$). \end{proof}

For the frustum, we define  (for  $c>0$) the function 
\begin{align*}
F_{\nu} (x;l_1,l_2)&=  J_{\nu}(l_1 x) Y_{\nu-1}(l_2 x) - Y_{\nu}(l_1 x) J_{\nu-1}(l_2 x),
\end{align*}
and let $f_{\nu,k}(l_1,l_2)$ denote its zeros.  Then, we have the following result.

\begin{lem}\label{fit} The torsion zeta function of the frustum with mixed BC is
\[
t_{\rm Frustum, mix}^{(m)}(s)=w_0^{(m)}(s)+w_1^{(m)}(s)+w^{(m)}_2(s)+w^{(m)}_3(s),
\]
with
\begin{align*}
w_0^{(m)}(s)&= \frac{1}{2}\sum^{\left[\frac{m}{2}\right]-1}_{q=0} (-1)^q \left((\hat D_{q,-}(s;l_2,l_1)-\hat D_{q,+}(s;l_1,l_2))\right.\\
&\left.+(-1)^{m-1}(\hat D_{q,+}(s;l_2,l_1)-\hat D_{q,-}(s;l_1,l_2))\right),\\
w_1^{(2p-1)}(s)&= (-1)^{p-1}\frac{1}{2}  \left(  \hat D_{p-1,0}(s;l_2,l_1)-\hat D_{p-1,0}(s;l_1,l_2) \right),&w_1^{(2p)}(s)&=0,\\
w_2^{(m)}(s)&= \frac{1}{2} \sum_{q=0}^{\left[\frac{m-1}{2}\right]} (-1)^{q+1}  \left(d_{q}(s;l_2,l_1)+(-1)^m d_{q}(s;l_1,l_2)\right),\\
w_3^{(2p)}(s)=&(-1)^{p+1}\frac{1}{2}d_p(s;l_1,l_2),&w_3^{(2p-1)}(s)&=0, 
\end{align*}
where $\hat D_{p-1,0}$ is given either by $\hat D_{p-1,+}$ or $\hat D_{p-1,-}$, since $\alpha_{p-1}=0$, when $m=2p-1$, and
\begin{align*}
\hat D_{q,\pm}(s;l_1,l_2)&=\sum^{\infty}_{n,k=1} m_{{\rm cex},q,n} \hat f^{-2s}_{\mu_{q,n},\pm\alpha_q,k}(l_1,l_2),&d_q(s;l_1,l_2)&=m_{{\rm har},q}\sum_{k=1}^{\infty} f^{-2s}_{-\alpha_{q-1},k}(l_1,l_2).
\end{align*}

\end{lem}

\begin{lem}\label{al2} The torsion zeta function on the frustum with absolute BC reads
\[
t^{(m)}_{\rm Frustum, abs}(s)=y^{(m)}_0(s)+y^{(m)}_1(s)+y^{(m)}_2(s)+y^{(m)}_3(s),
\]
where:
\begin{align*}
y_0^{(m)}(s)&= \frac{1}{2}\sum^{\left[\frac{m}{2}\right]-1}_{q=0} (-1)^q \left((E_q(s)-\hat E_{q,+}(s))
+(-1)^{m-1}(E_q(s)-\hat E_{q,-}(s))\right),\\
y_1^{(2p-1)}(s)&= (-1)^{p-1}\frac{1}{2}  \left(  E_{p-1}(s)-\hat E_{p-1,0}(s) \right),&y_1^{(2p)}(s)&=0,\\
y_2^{(m)}(s)&= \frac{1}{2} \sum_{q=0}^{\left[\frac{m-1}{2}\right]} (-1)^{q+1}  m_{{\rm har},q}\left(e_{q-1,-}(s)+(-1)^m e_{q,-}(s)\right),\\
y_3^{(2p)}(s)&= (-1)^{p+1} m_{{\rm har},p}\frac{1}{2}e_{p,-}(s),&y_3^{(2p-1)}(s)&=0,
\end{align*}
where $\hat E_{q,0}$ denotes $\hat E_{q,\pm}$ with $\alpha_q=0$, and
\begin{align*}
E_q(s)&=\sum^{\infty}_{n,k=1} m_{{\rm cex},q,n}\upsilon^{-2s}_{\mu_{q,n},k},&
\hat E_{q,\pm}(s)&=\sum^{\infty}_{n,k=1} m_{{\rm cex},q,n}\hat \upsilon^{-2s}_{\mu_{q,n},\pm\alpha_q,k},
&e_{q,-}(s)&=\sum_{k=1}^{\infty} \upsilon^{-2s}_{-\alpha_{q},k}.
\end{align*}
\end{lem}

\section{Zeta determinants}
\label{s3}

We recall in this section the main points of the technique that we will use to compute the derivative at $s=0$ of the zeta functions appearing in the torsion zeta functions introduced  in Section \ref{s2}. This section is essentially contained in Section 4 of \cite{HS1}, to which we refer  for details, and based on \cite{Spr9}. Given a sequence $S=\{a_n\}_{n=1}^\infty$ of spectral type, we define the {\it zeta function}  by 
\[
\zeta(s,S)=\sum_{n=1}^\infty a_n^{-s},
\]
when $\Re(s)> \es(S)$, and by analytic continuation otherwise, and for all $\lambda\in\rho(S)=\C-S$, we define the {\it Gamma function}   by the canonical product,
\beq\label{gamma}
\frac{1}{\Gamma(-\lambda,S)}=\prod_{n=1}^\infty\left(1+\frac{-\lambda}{a_n}\right)\e^{\sum_{j=1}^{\gs(S)}\frac{(-1)^j}{j}\frac{(-\lambda)^j}{a_n^j}}.
\eeq

Given a double sequence $S=\{\lambda_{n,k}\}_{n,k=1}^\infty$  of non
vanishing complex numbers with unique accumulation point at the
infinity, finite exponent $s_0=\es(S)$ and genus $p=\gs(S)$, we use the notation $S_n$ ($S_k$) to denote the simple sequence with fixed $n$ ($k$), we call the exponents of $S_n$ and $S_k$ the {\it relative exponents} of $S$, and we use the notation $(s_0=\es(S),s_1=\es(S_k),s_2=\es(S_n))$; we define {\it relative genus} accordingly.

\begin{defi} Let $S=\{\lambda_{n,k}\}_{n,k=1}^\infty$ be a double
sequence with finite exponents $(s_0,s_1,s_2)$, genus
$(p_0,p_1,p_2)$, and positive spectral sector
$\Sigma_{\theta_0,c_0}$. Let $U=\{u_n\}_{n=1}^\infty$ be a totally
regular sequence of spectral type of infinite order with exponent
$r_0$, genus $q$, domain $D_{\phi,d}$. We say that $S$ is
spectrally decomposable over $U$ with power $\kappa$, length $\ell$ and
asymptotic domain $D_{\theta,c}$, with $c={\rm min}(c_0,d,c')$,
$\theta={\rm max}(\theta_0,\phi,\theta')$, if there exist positive
real numbers $\kappa$, $\ell$ (integer), $c'$, and $\theta'$, with
$0< \theta'<\pi$,   such that:
\begin{enumerate}
\item the sequence
$u_n^{-\kappa}S_n=\left\{\frac{\lambda_{n,k}}{u^\kappa_n}\right\}_{k=1}^\infty$ has
spectral sector $\Sigma_{\theta',c'}$, and is a totally regular
sequence of spectral type of infinite order for each $n$;
\item the logarithmic $\Gamma$-function associated to  $S_n/u_n^\kappa$ has an asymptotic expansion  for large
$n$ uniformly in $\lambda$ for $\lambda$ in
$D_{\theta,c}$, of the following form
\beq\label{exp}
\log\Gamma(-\lambda,u_n^{-\kappa} S_n)=\sum_{h=0}^{\ell}
\phi_{\sigma_h}(\lambda) u_n^{-\sigma_h}+\sum_{l=0}^{L}
P_{\rho_l}(\lambda) u_n^{-\rho_l}\log u_n+o(u_n^{-r_0}),
\eeq
where $\sigma_h$ and $\rho_l$ are real numbers with $\sigma_0<\dots <\sigma_\ell$, $\rho_0<\dots <\rho_L$, the
$P_{\rho_l}(\lambda)$ are polynomials in $\lambda$ satisfying the condition $P_{\rho_l}(0)=0$, $\ell$ and $L$ are the larger integers 
such that $\sigma_\ell\leq r_0$ and $\rho_L\leq r_0$.
\end{enumerate}
\label{spdec}
\end{defi}

Define the following functions, ($\Lambda_{\theta,c}=\left\{z\in \C~|~|\arg(z-c)|= \frac{\theta}{2}\right\}$, oriented counter clockwise):
\beq\label{fi1}
\Phi_{\sigma_h}(s)=\int_0^\infty t^{s-1}\frac{1}{2\pi i}\int_{\Lambda_{\theta,c}}\frac{\e^{-\lambda t}}{-\lambda} \phi_{\sigma_h}(\lambda) d\lambda dt.
\eeq

By Lemma 3.3 of \cite{Spr9}, for all $n$, we have the expansions:
\beq\label{form}\begin{aligned}
\log\Gamma(-\lambda,S_n/{u_n^\kappa})&\sim\sum_{j=0}^\infty a_{\alpha_j,0,n}
(-\lambda)^{\alpha_j}+\sum_{k=0}^{p_2} a_{k,1,n}(-\lambda)^k\log(-\lambda),\\
\phi_{\sigma_h}(\lambda)&\sim\sum_{j=0}^\infty b_{\sigma_h,\alpha_j,0}
(-\lambda)^{\alpha_j}+\sum_{k=0}^{p_2} b_{\sigma_h,k,1}(-\lambda)^k\log(-\lambda),
\end{aligned}
\eeq
for large $\lambda$ in $D_{\theta,c}$. We set (see Lemma 3.5 of \cite{Spr9})
\beq\label{fi2}
\begin{aligned}
A_{0,0}(s)&=\sum_{n=1}^\infty \left(a_{0, 0,n} -{\sum_{h=0}^\ell}{^{\displaystyle
'}}
b_{\sigma_h,0,0}u_n^{-\sigma_h}\right)u_n^{-\kappa s},\\
A_{j,1}(s)&=\sum_{n=1}^\infty \left(a_{j, 1,n} -{\sum_{h=0}^\ell}{^{\displaystyle
'}}
b_{\sigma_h,j,1}u_n^{-\sigma_h}\right)u_n^{-\kappa s},
~~~0\leq j\leq p_2,
\end{aligned}
\eeq
where the notation $\sum'$ means that only the terms such that $\zeta(s,U)$ has a pole at $s=\sigma_h$ appear in the sum.

\begin{theo} \label{t4} Let $S$ be spectrally decomposable over $U$ as in Definition \ref{spdec}. Assume that the functions $\Phi_{\sigma_h}(s)$ have at most simple poles for $s=0$. Then,
$\zeta(s,S)$ is regular at $s=0$, and
\begin{align*}
\zeta(0,S)=&-A_{0,1}(0)+\frac{1}{\kappa}{\sum_{h=0}^\ell} \Ru_{s=0}\Phi_{\sigma_h}(s)\Ru_{s=\sigma_h}\zeta(s,U),\\
\zeta'(0,S)=&-A_{0,0}(0)-A_{0,1}'(0)+\frac{\gamma}{\kappa}\sum_{h=0}^\ell\Ru_{s=0}\Phi_{\sigma_h}(s)\Ru_{s=\sigma_h}\zeta(s,U)\\
&+\frac{1}{\kappa}\sum_{h=0}^\ell\Rz_{s=0}\Phi_{\sigma_h}(s)\Ru_{s=\sigma_h}\zeta(s,U)+{\sum_{h=0}^\ell}{^{\displaystyle
'}}\Ru_{s=0}\Phi_{\sigma_h}(s)\Rz_{s=\sigma_h}\zeta(s,U),
\end{align*}
where the notation $\sum'$ means that only the terms such that $\zeta(s,U)$ has a pole at $s=\sigma_h$ appear in the sum.

\end{theo}


\begin{rem}\label{rqr} We call regular part of $\zeta(0,S)$ the first term appearing in the formula given in the theorem, and regular part of $\zeta'(0,S)$ the first two terms. The other terms  we call singular part.
\end{rem}

\begin{corol} \label{c} Let $S_{(j)}=\{\lambda_{(j),n,k}\}_{n,k=1}^\infty$, $j=1,...,J$, be a finite set of  double sequences that satisfy all the requirements of Definition \ref{spdec} of spectral decomposability over a common sequence $U$, with the same parameters $\kappa$, $\ell$, etc., except that the polynomials $P_{(j),\rho}(\lambda)$ appearing in condition (2) do not vanish for $\lambda=0$. Assume that some linear combination $\sum_{j=1}^J c_j P_{(j),\rho}(\lambda)$, with complex coefficients, of such polynomials does satisfy this condition, namely that $\sum_{j=1}^J c_j P_{(j),\rho}(\lambda)=0$. Then, the linear combination of the zeta function $\sum_{j=1}^J c_j \zeta(s,S_{(j)})$ is regular at $s=0$ and satisfies the linear combination of the formulas given in Theorem \ref{t4}.
\end{corol}

We conclude recalling some formulas for the zeta determinants of some simple sequences. The results are known to specialists, and can be found in different places. We will use the formulation of \cite{Spr1}. For positive real  $l$ and $q$, define the {\it non homogeneous quadratic Bessel zeta function} by
\[
z(s,\nu,q,l)=\sum_{k=1}^\infty \left(\frac{j_{\nu,k}^2}{l^2}+q^2\right)^{-s},
\]
for $\Re(s)>\frac{1}{2}$. Then, $z(s,\nu,q,l)$ extends analytically to a meromorphic function in the complex plane with simple poles at $s=\frac{1}{2}, -\frac{1}{2}, -\frac{3}{2}, \dots$. The point $s=0$ is a regular point and
\beq\label{p000}
\begin{aligned}
z(0,\nu,q,l)&=-\frac{1}{2}\left(\nu+\frac{1}{2}\right),&
z'(0,\nu,q,l)&=-\log\sqrt{2\pi l}\frac{I_\nu(lq)}{q^\nu}.
\end{aligned}
\eeq

In particular, taking the limit for $q\to 0$,
\beq\label{p00}
z'(0,\nu,0,l)=-\log\frac{\sqrt{\pi}l^{\nu+\frac{1}{2}}}{2^{\nu-\frac{1}{2}}\Gamma(\nu+1)}.
\eeq

\section{Decomposition of the torsion zeta function}
\label{s4}

Inspection of the formulas in the lemmas of Section \ref{s2} shows that the torsion zeta function is the finite sum of some simple and some double series:
\begin{align*}
t_{\rm Cone, abs}(s)&=t_{0}(s)+t_1(s) +t_2(s)+t_3(s),\\
t_{\rm Frustum, mixed}(s)&=w_0(s)+w_1(s)+w_2(s)+w_3(s),\\
t_{\rm Frustum, abs}(s)=&y_0(s)+y_1(s)+y_2(s)+y_3(s),
\end{align*}
where $t_0,t_1$,  $w_0,w_1$, and $y_0,y_1$ are double series, and the others simple series. Simple series can be treated by using the formulas appearing at the end of Section \ref{s3}. In order to deal with the double series,  applying the spectral decomposition Theorem  \ref{t4}, the derivative at zero of $t_0,t_1$,  $w_0,w_1$, and $y_0,y_1$ decomposes into two parts, called regular and singular contribution (see Remark \ref{rqr}), and this gives the following  decomposition of the analytic torsion (in the following we assume $l_1>0$), 
\begin{align*}
\log T_{\rm abs, ideal}(C_l (W))&=t_{0, {\rm reg}}'(0)+t_{1, {\rm reg}}'(0)+t_{0, {\rm sing}}'(0)+t_{1, {\rm sing}}'(0)+t_2'(0)+t_3'(0),\\
\log T_{\rm mixed}(C_{[l_1,l_2]} (W))&=w_{0, {\rm reg}}'(0)+w_{1, {\rm reg}}'(0)+w_{0, {\rm sing}}'(0)+w_{1, {\rm sing}}'(0)+w_2'(0)+w_3'(0),\\
\log T_{\rm abs}(C_{[l_1,l_2]} (W))&=y_{0,\rm reg}(s)+y_{1,\rm reg}(s)+y_{0,\rm sing}(s)+y_{1,\rm sing}(s)+y_2(s)+y_3(s).
\end{align*}

On the other side, for the frustum, that is a manifold, we also have the decomposition
\begin{align*}
\log T_{\rm mixed}(C_{[l_1,l_2]} (W))=&\log \tau(C_{[l_1,l_2]} (W), \{l_1\}\times W)\\
&+\frac{1}{4}\chi(\b C_{[l_1,l_2]} (W))\log 2+A_{\rm BM, mixed}(\b  C_{[l_1,l_2]} (W)),\\
\log T_{\rm abs}(C_{[l_1,l_2]} (W))=&\log \tau(C_{[l_1,l_2]} (W))\\
&+\frac{1}{4}\chi(\b C_{[l_1,l_2]} (W))\log 2+A_{\rm BM, abs}(\b  C_{[l_1,l_2]} (W)). 
\end{align*}

We will prove in Section \ref{regularpart} that
\beq\label{www}
\begin{aligned}
\log \tau(C_{[l_1,l_2]} (W), \{l_1\}\times W)+\frac{1}{2}\chi(W)\log 2&=w_{0, {\rm reg}}'(0)+w_{1, {\rm reg}}'(0)+w_2'(0)+w_3'(0),\\
A_{\rm BM, mixed}(\b  C_{[l_1,l_2]} (W))&=w_{0, {\rm sing}}'(0)+w_{1, {\rm sing}}'(0),
\end{aligned}
\eeq
\beq\label{www1}
\begin{aligned}
\log \tau(C_{[l_1,l_2]} (W))+\frac{1}{2}\chi(W)\log 2&=y_{0, {\rm reg}}'(0)+y_{1, {\rm reg}}'(0)+y_2'(0)+y_3'(0),\\
A_{\rm BM, abs}(\b  C_{[l_1,l_2]} (W))&=y_{0, {\rm sing}}'(0)+y_{1, {\rm sing}}'(0).
\end{aligned}
\eeq

This suggests to introduce a similar decomposition for the cone
\begin{align*}
\log T_{\rm abs, ideal}(C_l (W))&=\log T_{\rm global}(C_l (W))+\frac{1}{4}\chi(\b C_l (W))\log 2+A_{\rm BM, abs}(\b  C_l (W)),
\end{align*}
with
\beq \label{ww}
\begin{aligned}
\log T_{\rm global}(C_l (W))+\frac{1}{4}\chi(\b C_l (W))\log 2&=t_{0, {\rm reg}}'(0)+t_{1, {\rm reg}}'(0)+t_2'(0)+t_3'(0),\\
A_{\rm BM, abs}(\b  C_l (W))&=t_{0, {\rm sing}}'(0)+t_{1, {\rm sing}}'(0).
\end{aligned}
\eeq

We will call the first term, namely global  plus Euler, the {\it regular part} of the torsion, and the second term the {\it anomaly boundary term}. We will prove in Section \ref{nove} that the notation is justified by the fact that $t_{0, {\rm sing}}'(0)+t_{1, {\rm sing}}'(0)$ coincides with the term that we would obtain if we would apply the formula of Br\"{u}ning and Ma for the cone as if it were a regular manifold.

\section{Calculations I: application of Definition \ref{spdec}}
\label{calc1}

In both the cases of the cone and of the frustum, we prove that the double series are spectrally decomposable according to Definition \ref{spdec} (see detail in \cite{HS2}) on the same simple sequence, that we discuss now. The simple sequence is $U_q=\{m_{q,n}:\mu_{q,n}\}_{n=1}^\infty$. This is  a totally regular sequence of spectral type with infinite order, $\es(U_{q})=\gs(U_{q})=m=\dim W$, and the associated zeta functions is
\[
\zeta(s,U_{q})=\zeta_{\rm cex}\left(\frac{s}{2},\tilde\Delta^{(q)}+\alpha^2_q\right).
\]

The  possible poles of $\zeta(s,U_{q})$ are at $s=m-h$, $h=0,2,4,\dots$,  and the residues are completely determined by the residues of the function $\zeta_{\rm cex}(s,\tilde\Delta^{(q)})$ \cite[Lemma 5.2]{HS2}.

We give in this section complete calculations for the cone and for the frustum with mixed BC, we omit the calculations for the frustum with absolute BC that are very similar to ones for the frustum with mixed BC.

\subsection{The cone} We consider the double series
\[
\begin{aligned}
Z_q(s)&=\sum^{\infty}_{n,k=1} m_{{\rm cex},q,n}j^{-2s}_{\mu_{q,n},k},&
\hat Z_{q,\pm}(s)&=\sum^{\infty}_{n,k=1} m_{{\rm cex},q,n}\hat j^{-2s}_{\mu_{q,n},\pm\alpha_q,k},
\end{aligned}
\]
appearing in Lemma \ref{l000}. These are the zeta functions associated to the double sequences $S_q=\{m_{{\rm cex},q,n}:j^{2}_{\mu_{q,n},k}\}$ and $\hat S_{q,\pm}=\{m_{{\rm cex},q,n}:\hat j^{2}_{\mu_{q,n},\pm\alpha_q,k}\}$. It is easy to see that these sequences have power $\kappa=2$.  
By Definition \ref{spdec}, the relevant spectral function associated to $S_q$ is the function (compare with equation (\ref{gamma})):
\begin{align*}
\log \Gamma(-\lambda,S_{q})=&-\log\prod_{k=1}^\infty
\left(1+\frac{(-\lambda)}{j_{\mu_{q,n},k}^2}\right)\\
=&-\log I_{\mu_{q,n}}(\sqrt{-\lambda})+\mu_{q,n}\log\sqrt{-\lambda} -\mu_{q,n}\log 2-\log\Gamma(\mu_{q,n}+1),
\end{align*}
and the relevant spectral function associated to $\hat S_{q,\pm}$ is the function:
\begin{align*}
\log \Gamma(-\lambda,\hat S_{q,\pm})=&-\log\prod_{k=1}^\infty \left(1+\frac{(-\lambda)}{\hat j_{\mu_{q,n},\pm\alpha_q,k}^2}\right)\\
=&-\log \hat I_{\mu_{q,n},\pm\alpha_q}(\sqrt{-\lambda})+\mu_{q,n}\log\sqrt{-\lambda}-\log2^{\mu_{q,n}}\Gamma(\mu_{q,n})+\log\left(1\pm\frac{\alpha_q}{\mu_{q,n}}\right),
\end{align*}
where (for $-\pi < \arg(z)\leq \frac{\pi}{2}$)
\[
\hat I_{\nu,c}(z)=\e^{-\frac{\pi}{2}i\nu} \hat J_{\nu,c}(i z).
\]

First, we need uniform asymptotic expansions of the function $\log \Gamma$ for large $n$. Such expansions can be obtained from those of the Bessel function (see \cite[Section 5.1]{HS2} for details). For large $n$, uniformly in $\lambda$, (compare with equation (\ref{exp}))
\beq\label{p1}
\begin{aligned}
\log \Gamma(-\lambda,S_{q}/\mu_{q,n}^2)=&\mu_{q,n}\log\mu_{q,n}\sqrt{-\lambda} -\log 2^{\mu_{q,n}}\Gamma(\mu_{q,n}+1)-\mu_{q,n}\sqrt{1-\lambda}\\
&-\mu_{q,n}\log\sqrt{-\lambda} +\mu_{q,n}\log(1+\sqrt{1-\lambda})+ \frac{1}{2}\log2\pi\mu_{q,n} + \frac{1}{4}\log(1-\lambda) \\
&+ \sum_{j=1}^{m}\frac{\phi_{q,j}(\lambda)}{\mu_{q,n}^j} + O(\mu_{q,n}^{-2p}),
\end{aligned}
\eeq
where $\phi_{q,j}(\lambda)$ is the coefficient of $\mu_{q,n}^{-j}$ in the asymptotic expansion of (for the definition of the function $U$ and $W$ see Lemma \ref{tec2} below)
\[
\log\left(1+\sum_{k=1}^\infty U_{k}(\lambda) \mu_{q,n}^{-k}\right),
\]
for large $\mu_{q,n}$, i.e.
\beq\label{ff11}
\begin{aligned}
\log\left(1+\sum_{k=1}^\infty U_{k}(\lambda) \mu_{q,n}^{-k}\right)=\sum_{j=1}^\infty \phi_{q,j}(\lambda) \mu_{q,n}^{-j}.
\end{aligned}
\eeq
and
\beq\label{p2}
\begin{aligned}
\log \Gamma(-\lambda,\hat S_{q,\pm}/\mu_{q,n}^2)=&\mu_{q,n}\log\mu_{q,n}\sqrt{-\lambda}-\log2^{\mu_{q,n}}\Gamma(\mu_{q,n}+1)+\log\left(1\pm\frac{\alpha_q}{\mu_{q,n}}\right)\\
&-\mu_{q,n}\sqrt{1-\lambda}-\mu_{q,n}\log\sqrt{-\lambda}+\mu_{q,n}\log(1+\sqrt{1-\lambda}) \\
&+ \frac{1}{2}\log2\pi\mu_{q,n} -\frac{1}{4}\log(1-\lambda)+  \sum_{j=0}^{m}\frac{\hat\psi_{q,j,\pm}(\lambda)}{\mu_{q,n}^j} + O(\mu_{q,n}^{-2p}),
\end{aligned}
\eeq
where $\hat\psi_{q,j,\pm}(\lambda)$ is the coefficient of $\mu_{q,n}^{-j}$ in the asymptotic expansion of 
\[
\log\left(1+\sum_{k=1}^\infty W_{\pm\alpha_q,k}(\lambda) \mu_{q,n}^{-k}\right),
\]
for large $\mu_{q,n}$, i.e.
\[
\log\left(1+\sum_{k=1}^\infty W_{\pm\alpha_q,k}(\lambda) \mu_{q,n}^{-k}\right)=\sum_{j=1}^\infty \hat\psi_{q,j,\pm}(\lambda) \mu_{q,n}^{-j},
\]
therefore the complete coefficient of $\mu_{q,n}^{-j}$ in $\log \Gamma(-\lambda,\hat S_{q,\pm}/\mu_{q,n}^2)$ is
\beq\label{ff22}
\hat\phi_{q,j,\pm}(\lambda)=\hat\psi_{q,j,\pm}(\lambda)+\frac{(-1)^{j+1}}{j} (\pm\alpha_q)^j.
\eeq


Second, we need  asymptotic expansions of the function $\log \Gamma$ and of the functions $\phi_{q,j}$ and $\hat \phi_{q,j,\pm}$ for large $\lambda$. The expansion of $\log \Gamma$ can be obtained from those of the Bessel function  \cite[pg.s 641, 642]{HS2} (compare with equation (\ref{form}))
\beq
\label{loo1}
\begin{aligned}
\log\Gamma(-\lambda, S_q/\mu_{q,n}^2) =& \frac{1}{2}\log
2\pi+\left(\mu_{q,n}+\frac{1}{2}\right)\log\mu_{q,n}-\mu_{q,n}\log2\\
&-\log\Gamma(\mu_{q,n}+1)+\frac{1}{2}\left(\mu_{q,n}+\frac{1}{2}\right)\log(-\lambda) +
O(\e^{-\mu_{q,n}\sqrt{-\lambda}}),
\end{aligned}
\eeq
\beq
\label{loo2}
\begin{aligned}
\log\Gamma(-\lambda, \hat S_{q,\pm}/\mu_{q,n}^2) =& \mu_{q,n}\sqrt{-\lambda}+\frac{1}{2}\log
2\pi+\left(\mu_{q,n}-\frac{1}{2}\right)\log\mu_{q,n}-\log2^{\mu_{q,n}}\Gamma(\mu_{q,n})\\
&+\frac{1}{2}\left(\mu_{q,n}-\frac{1}{2}\right)\log(-\lambda) +
\log\left(1\pm\frac{\alpha_q}{\mu_{q,n}}\right) +O(\e^{-\mu_{q,n}\sqrt{-\lambda}}).
\end{aligned}
\eeq

About the functions $\phi_{q,j}(\lambda)$ and $\hat \phi_{q,j,\pm}(\lambda)$ we have the following facts.

\begin{lem}\label{tec2} For $j=0$, 
\begin{align*}
\phi_{q,0}(\lambda)&=\frac{1}{2}\log 2\pi +\frac{1}{4}\log (1-\lambda),\\
\hat\phi_{q,0,\pm}(\lambda)&=\frac{1}{2}\log 2\pi -\frac{1}{4}\log (1-\lambda).
\end{align*}

For $j>0$, the functions $\phi_{q,j}(\lambda)$ and  $\hat\phi_{q,j,\pm}(\lambda)$ are polynomial of order $3j$ in $\frac{1}{\sqrt{1-\lambda}}$, whose parity only depends on the index $j$,  and with minimal monomial of degree $j$, and satisfy the following formulas:
\begin{align*}
\left.\left(2\phi_{q,2k-1}(\lambda)-\hat\phi_{q,2k-1,+}(\lambda)-\hat\phi_{q,2k-1,-}(\lambda)\right)\right|_{\lambda=0}&=0,\\
\left.\left(\hat\phi_{q,2k,-}(\lambda)-\hat\phi_{q,2k,+}(\lambda)\right)\right|_{\lambda=0}&=0.
\end{align*}
\end{lem}
\begin{proof} The case $j=0$ follows by inspection of the formulas in equations (\ref{p1}) and (\ref{p2}). Next, assume $j>0$. Recall the definition of the  functions $U_j(\lambda)$ and $V_j(\lambda)$  \cite[(7.10) and Ex. 7.2]{Olv}. Let
\begin{align*}
u_0(w)&=1,&u_{j+1}(w)&=\frac{1}{2}w^2 (1-w^2)u_j'(w)+\frac{1}{8}\int_0^w (1-5y^2) u_j(y) dy,\\
v_0(w)&=1,&v_{j+1}(w)&=u_{j+1}(w)-\frac{1}{2}w (1-w^2)u_j(w)-w^2(1-w^2)u'(w),
\end{align*}
then 
\begin{align*}
U_j(\lambda)&=u_j\left(\frac{1}{\sqrt{1-\lambda}}\right),&V_j(\lambda)&=v_j\left(\frac{1}{\sqrt{1-\lambda}}\right).
\end{align*}

By these formulas, it is clear that the functions $U_j$ and $V_j$ are polynomials in $\frac{1}{\sqrt{1-\lambda}}$ with the stated properties. 
Since, \cite[pg. 639]{HS2})
\[
W_{\pm\alpha_q, j}(\lambda)=V_j(\lambda)\pm \frac{\alpha_q}{\sqrt{1-\lambda}}U_{j-1}(\lambda),
\]
the same is true for $W_j(\lambda)$. Whence the first part of the statement follows by the very definition of the functions $\phi_{q,j}(\lambda)$ and  $\hat\phi_{q,j,\pm}(\lambda)$, since in the functions $\phi_{q,2k-1}(\lambda)$ and $\hat\phi_{q,2k-1,\pm}(\lambda)$ only $U_{2k-1}$ and $V_{2k-1}$ appear, and in the $\phi_{q,2k}(\lambda)$ and $\hat\phi_{q,2k,\pm}(\lambda)$  only $U_{2k}$ and $V_{2k}$ appear. 
Next consider the situation in more details:  by definition
\begin{align*}
\log\left(1+\sum_{k=1}^\infty U_{k}(\lambda) \mu_{q,n}^{-k}\right)&=\sum_{j=1}^\infty \phi_{q,j}(\lambda) \mu_{q,n}^{-j}, &
\hat\phi_{q,j,\pm}(\lambda)&=\hat\psi_{q,j,\pm}(\lambda)+\frac{(-1)^{j+1}}{j} (\pm\alpha_q)^j,
\end{align*}
where
\[
\log\left(1+\sum_{k=1}^\infty \left(V_k(\lambda)\pm \frac{\alpha_q}{\sqrt{1-\lambda}}U_{k-1}(\lambda)\right) \mu_{q,n}^{-k}\right)=\sum_{j=1}^\infty \hat\psi_{q,j,\pm}(\lambda) \mu_{q,n}^{-j}.
\]

Since, again by the very definition, $V_j(0)=U_j(0)$ for all $j$, by comparing the two formulas above, we realize that, if $\lambda=0$, in each odd term of the expansion of $2\phi_{q,2k-1}(\lambda)-\hat\phi_{q,2k-1,+}(\lambda)-\hat\phi_{q,2k-1,-}(\lambda)$ in $\mu_{q,n}$, i.e. in the coefficient of each $\mu_{q,n}^{j=2k-1}$, the two terms (coming from $\hat\phi_{q,j,\pm}(\lambda)$)
\[
\pm \frac{\alpha_q}{\sqrt{1-\lambda}}U_{j-1}(\lambda)+\frac{(-1)^{j+1}}{j} (\pm\alpha_q)^j,
\]
cancel each other, while the term (coming from $\phi_{q,j}(\lambda)$)
\[
2U_{k}(\lambda),
\]
will cancel out by the two terms (coming from $\hat\phi_{q,j,\pm}(\lambda)$)
\[
-2V_{k}(\lambda).
\]

This prove the first formula in the last statement. The proof of the second one is similar. 
\end{proof}

\begin{lem}\label{tec2b} For $j>0$, the following functions of $\Phi_{q,j}$ and $\hat\Phi_{q,j,\pm}$ are regular at $s=0$:
\begin{align*}
\Ru_{s=0}\left(2\Phi_{q,2k-1}(s)-\hat\Phi_{q,2k-1,+}(s)-\hat\Phi_{q,2k-1,-}(s)\right)&=0,\\
\Ru_{s=0}\left(\hat\Phi_{q,2k,-}(s)-\hat\Phi_{q,2k,+}(s)\right)&=0.
\end{align*}
\end{lem}
\begin{proof} The argument is the same for the functions $\hat \Phi$ and the function $\Phi$, so just consider the last ones. By Lemma \ref{tec2}
\[
\phi_{q,j}(\lambda)=\sum_{k=j}^{3j} c_k \frac{1}{(1-\lambda)^\frac{k}{2}},
\]
and since (see \cite{Spr3}) 
\[
\int_0^\infty t^{s-1}\frac{1}{2\pi i}\int_{\Lambda_{\theta,c}}\frac{\e^{-\lambda t}}{-\lambda} \frac{1}{(1-\lambda)^\frac{k}{2}}d\lambda dt=\left\{\begin{array}{cl}0,&k=0,\\ \frac{\Gamma(s+k)}{\Gamma(k)s}, &k\not=0,\end{array}\right.
\]
we have that
\[
\Phi_{q,j}(s)=\sum_{k=1}^{3j} c_k \frac{\Gamma(s+k)}{\Gamma(k)s},
\]
and this means that
\begin{align*}
\Ru_{s=0}\Phi_{q,2h-1}(s)&=\sum_{k=j}^{3(2h-1)} c_k=\phi_{q,2h-1}(0),\\
\Ru_{s=0}\Phi_{q,2h}(s)&=\sum_{k=j}^{6h} c_k=\phi_{q,2h}(0).
\end{align*}
\end{proof}

Applying the definition in equation (\ref{fi2}), we see that all relevant $b_{\sigma_h,0,0/1}$ and $\hat b_{\sigma_h,0,0/1}$ vanish. For $j=\sigma_h=m-h$, with $h=0,2,4,\dots$, $h\not=m$; and  therefore when $j=0$, $b_{0,0,0/1}$ and $\hat b_{0,0,0/1}$ do not appear in the sums, while when $j>0$ there are neither logarithmic terms nor constant terms in the expansion for large $\lambda$ of the functions $\phi_{q,j}(\lambda)$ and $\hat \phi_{q,j,\pm}(\lambda)$ by the previous lemmas. Thus,

\beq\label{aa}
\begin{aligned}
A_{0,0,q}(s)&=\sum_{n=1}^\infty m_{{\rm cex},q,n}a_{0, 0,n,q}    \mu_{q,n}^{-2s},&
A_{0,1,q}(s)&=\sum_{n=1}^\infty m_{{\rm cex},q,n} a_{0, 1,n,q}   \mu_{q,n}^{-2s},\\
\hat A_{0,0,q,\pm}(s)&=\sum_{n=1}^\infty m_{{\rm cex},q,n} \hat a_{0, 0,n,q,\pm}  \mu_{q,n}^{-2s},&
\hat A_{0,1,q,\pm}(s)&=\sum_{n=1}^\infty m_{{\rm cex},q,n} \hat a_{0, 1,n,q,\pm}  \mu_{q,n}^{-2s}.
\end{aligned}
\eeq
where $a_{0, 0,n,q}$  and $a_{0, 1,n,q}$ are the constant and the logarithmic term in the expansion in equation (\ref{loo1}), and $\hat a_{0, 0,n,q,\pm}$  and $\hat a_{0, 1,n,q,\pm}$ are the constant and the logarithmic term in the expansion in equation (\ref{loo2}).  

Third, concerning the singular term,  the functions defined in equation (\ref{fi1}), are
\beq\label{vvv}
\begin{aligned}
\Phi_{q,j}(s)=&\int_0^\infty t^{s-1}\frac{1}{2\pi i}\int_{\Lambda_{\theta,c}}\frac{\e^{-\lambda t}}{-\lambda} \phi_{q,j}(\lambda) d\lambda dt,\\
\hat \Phi_{q,j, \pm}(s)=&\int_0^\infty t^{s-1}\frac{1}{2\pi i}\int_{\Lambda_{\theta,c}}\frac{\e^{-\lambda t}}{-\lambda} \hat \phi_{q,j,\pm}(\lambda) d\lambda dt,
\end{aligned}
\eeq
where the $\phi$ and the $\hat \phi$ are given in equations (\ref{ff11}) and (\ref{ff22}), respectively. 


\subsection{The frustum with mixed BC} Consider the double series
\begin{align*}
\hat D_{q,\pm}(s;l_1,l_2)&=\sum^{\infty}_{n,k=1} m_{{\rm cex},q,n} \hat f^{-2s}_{\mu_{q,n},\pm\alpha_q,k}(l_1,l_2),
\end{align*}
associated to the double sequence $\hat \Theta_{q,\pm}(l_1,l_2)=\{m_{{\rm cex},q,n}:\hat f^{-2s}_{\mu_{q,n},\pm\alpha_q,k}(l_1,l_2)\}$. It is easy to see that this sequence has power $\kappa=2$.  
By Definition \ref{spdec}, the relevant spectral function associated to $\hat \Theta_{q,\pm}(l_1,l_2)$ is the function (compare with equation (\ref{gamma})):
\begin{align*}
\log(-\lambda;\hat \Theta_{q,\pm}(l_1,l_2))=&-\log \prod_{k=1}^\infty \left(1+\frac{-\lambda}{\hat f^2_{\mu_{q,n},\pm\alpha_q,k}(l_1,l_2)}\right)\\
=&-\log \hat G_{\mu_{n,q},\pm\alpha_q}(\sqrt{-\lambda};l_1,l_2) \\
& + \log\frac{1}{ \pi}+ \log \left(\left(\frac{l_2^{\mu_{q,n}}}{l_1^{\mu_{q,n}}}+\frac{l_1^{\mu_{q,n}}}{l_2^{\mu_{q,n}}}\right)\pm\frac{\alpha_q}{\mu_{q,n}}\left(\frac{l_2^{\mu_{q,n}}}{l_1^{\mu_{q,n}}} - \frac{l_1^{\mu_{q,n}}}{l_2^{\mu_{q,n}}}\right)\right),
\end{align*}
where (for $-\pi < \arg(z)\leq \frac{\pi}{2}$)
\[
\hat G_{\mu,c}(z;l_1,l_2) =  \hat F_{\mu,c}(iz;l_1,l_2).
\]

We need (uniform) asymptotic expansions of the function $\log \Gamma$ for large $n$ and for large $\lambda$. Such expansions can be obtained from those of the Bessel function. A long calculation gives the following results. For large $n$, uniformly in $\lambda$, (compare with equation (\ref{exp}))
\begin{align*}
\log(-\lambda;\hat \Theta_{q,\pm}(l_1,l_2)/\mu_{q,n}^2) =& -\mu_{q,n}\left(\sqrt{1-l_2^2 \lambda} -\sqrt{1-l_1^2 \lambda} \right) - \mu_{q,n}\log\frac{l_2(1+\sqrt{1 - l_1^2\lambda})}{l_1(1+\sqrt{1-l_2^2\lambda})} \\ 
&- \frac{1}{4} \log\frac{(1-l_1^2 \lambda)}{(1-l_2^2 \lambda)} - \sum_{j=1}^{m} \frac{\hat \psi_{q,j,\pm}(\lambda;l_1,l_2)}{\mu_{q,n}^j}\\
& + \log \left(\left(\frac{l_2^{\mu_{q,n}}}{l_1^{\mu_{q,n}}}+\frac{l_1^{\mu_{q,n}}}{l_2^{\mu_{q,n}}}\right)\pm\frac{\alpha_q}{\mu_{q,n}}\left(\frac{l_2^{\mu_{q,n}}}{l_1^{\mu_{q,n}}} - \frac{l_1^{\mu_{q,n}}}{l_2^{\mu_{q,n}}}\right)\right)+ O(\mu_{q,n}^{1-m}),
\end{align*} 
where $\hat \psi_{q,j,\pm}(\lambda;l_1,l_2)$ is the coefficient of $\mu_{q,n}^{-j}$ in the asymptotic expansion for large $\mu_{q,n}$, i.e. of 
\[
\log\left(1+\sum_{k=1}^\infty\Psi_{q,k,\pm}(l_1,l_2)\mu_{q,n}^{-k}\right),
\]
where
\begin{align*}
\Psi_{q,k,\pm}(l_1,l_2)=& \sgn(l_1-l_2)^kU_k(l_1\sqrt{-\lambda})+ \sgn(l_2-l_1)^kW_{\pm\sgn(l_2-l_1)\alpha_q,k}(l_2\sqrt{-\lambda})\\
&+\sum_{h=1}^{k-1} \sgn(l_1-l_2)^{h}\sgn(l_2-l_1)^{k-h}U_h(l_1\sqrt{-\lambda})W_{\pm\sgn(l_2-l_1)\alpha_q,k-h}(l_2\sqrt{-\lambda}).
\end{align*}


In other words,
\[
\log\left(1+\sum_{k=1}^\infty\Psi_{q,k,\pm}(l_1,l_2) \mu_{q,n}^{-k}\right)=\sum_{j=1}^\infty \hat\psi_{q,j,\pm}(\lambda;l_1,l_2) \mu_{q,n}^{-j},
\]
therefore the complete coefficient of $\mu_{q,n}^{-j}$ in $\log \Gamma(-\lambda,\hat \Theta_{q,\pm}(l_1,l_2)/\mu_{q,n}^2)$ is
\beq\label{ppmix}
\hat\phi_{q,j,\pm}(\lambda;l_1,l_2)=\hat\psi_{q,j,\pm}(\lambda;l_1,l_2)
+\frac{(-1)^{j+1}}{j}(\sgn(l_2-l_1)\alpha_q)^j.
\eeq


For large $\lambda$, (compare with equation (\ref{form}))
\beq
\label{loomix}\begin{aligned}
\log \Gamma(-\lambda,\hat S_{\mu_{q,n},\pm\alpha_q,k}(l_1,l_2)/\mu_{q,n}^2) =&-\mu_{q,n} (l_2 - l_1)\sqrt{-\lambda} - \frac{1}{2}\log\frac{l_2}{l_1}\\
& + \log \left(\left(\frac{l_2^{\mu_{q,n}}}{l_1^{\mu_{q,n}}}+\frac{l_1^{\mu_{q,n}}}{l_2^{\mu_{q,n}}}\right)\pm\frac{\alpha_q}{\mu_{q,n}}\left(\frac{l_2^{\mu_{q,n}}}{l_1^{\mu_{q,n}}} - \frac{l_1^{\mu_{q,n}}}{l_2^{\mu_{q,n}}}\right)\right) \\
&+ O\left(\frac{1}{\sqrt{-\lambda}}\right)
\end{aligned}
\eeq

Next consider the  functions $\hat \phi_{q,j,\pm}(\lambda;l_1,l_2)$. By the following lemma we see that all the results given in Lemmas \ref{tec2} and \ref{tec2b} hold for these functions.

\begin{lem}\label{tec1} The following relation holds:
\[
\hat\phi_{q,j,\pm}(\lambda;l_1,l_2)=\sgn(l_1-l_2)^j\phi_{q,j}(l_1^2\lambda)+ \sgn(l_2-l_1)^j\hat\phi_{q,j,\pm \sgn(l_2-l_1)}(l_2^2\lambda).
\]

\end{lem}
\begin{proof} By definition in equations (\ref{ff11}) and (\ref{ff22})
\[
\begin{aligned}
\log\left(1+\sum_{k=1}^\infty U_{k}(\sqrt{-\lambda}) \mu_{q,n}^{-k}\right)=\sum_{j=1}^\infty \phi_{q,j}(\lambda) \mu_{q,n}^{-j}.
\end{aligned}
\]
and
\[
\hat\phi_{q,j,\pm}(\lambda)=\hat\psi_{q,j,\pm}(\lambda)\pm\frac{(-1)^{j+1}}{j} \alpha_q,
\]
where
\[
\log\left(1+\sum_{k=1}^\infty W_{\pm\alpha_q,k}(\sqrt{-\lambda}) \mu_{q,n}^{-k}\right)=\sum_{j=1}^\infty \hat\psi_{q,j,\pm}(\lambda) \mu_{q,n}^{-j},
\]
while by equation (\ref{ppmix})
\[
\hat\phi_{q,j,\pm}(\lambda;l_1,l_2)=\hat\psi_{q,j,\pm}(\lambda;l_1,l_2)+\frac{(-1)^{j+1}}{j}(\sgn(l_2-l_1)\alpha_q)^j,
\]
where
\[
\log\left(1+\sum_{k=1}^\infty\Psi_{q,k,\pm}(l_1,l_2)\mu_{q,n}^{-k}\right)=\sum_{j=1}^\infty \hat\psi_{q,j,\pm}(\lambda;l_1,l_2)\mu_{q,n}^{-j},
\]
with
\begin{align*}
\Psi_{q,k,\pm}(l_1,l_2)=&\sgn(l_1-l_2)^kU_k(l_1\sqrt{-\lambda})+ \sgn(l_2-l_1)^kW_{\pm\sgn(l_2-l_1)\alpha_q,k}(l_2\sqrt{-\lambda})\\
&+\sum_{h=1}^{k-1} \sgn(l_1-l_2)^{h}\sgn(l_2-l_1)^{k-h}U_h(l_1\sqrt{-\lambda})W_{\pm\sgn(l_2-l_1)\alpha_q,k-h}(l_2\sqrt{-\lambda}).
\end{align*}

Now, 
\begin{align*}
\log\left(1+\sum_{k=1}^\infty\Psi_{q,k,\pm}(l_1,l_2)\mu_{q,n}^{-k})\right)=&\log\left(1+\sum_{k=1}^\infty \sgn(l_1-l_2)^kU_{k}(l_1\sqrt{-\lambda}) \mu_{q,n}^{-k}\right)\\
&+\log\left(1+\sum_{k=1}^\infty \sgn(l_2-l_1)^kW_{\pm\sgn(l_2-l_1)\alpha_q,k}(l_2\sqrt{-\lambda}) \mu_{q,n}^{-k}\right),
\end{align*}
and hence the thesis follows by the very definition.

\end{proof}

This means that the $b_{{\sigma_h},j,k}$ in equation (\ref{fi2}) are all disappearing, and hence 
\[
\begin{aligned}
A_{0,0,\pm}(s;l_1,l_2)&=\sum_{n=1}^\infty  m_{{\rm cex},q,n}a_{0, 0,n,\pm}(l_1,l_2)    \mu_{q,n}^{-2s},\\
A_{0,1,\pm}(s;l_1,l_2)&=\sum_{n=1}^\infty m_{{\rm cex},q,n} a_{0, 1,n,\pm}(l_1,l_2)  \mu_{q,n}^{-2s},
\end{aligned}
\]
where $a_{0, 0,n,\pm}(l_1,l_2)$  and $a_{0, 1,n,\pm}(l_1,l_2)$ are the constant and the logarithmic term in the expansion in equation (\ref{loomix}). 

Eventually,  concerning the singular term,  the functions defined in equation (\ref{fi1}) are
\beq\label{phiphi}
\hat \Phi_{q,j,\pm}(s;l_1,l_2)=\int_0^\infty t^{s-1}\frac{1}{2\pi i}\int_{\Lambda_{\theta,c}}\frac{\e^{-\lambda t}}{-\lambda} \hat \phi_{q,j,\pm}(\lambda;l_1,l_2) d\lambda dt,
\eeq
where the $\hat \phi$ are given in equation (\ref{ppmix}).

\section{Calculation II: application of Theorem \ref{t4}}
\label{calc2}

We now apply Theorem \ref{t4} and its corollary to compute the derivative at $s=0$ of the functions $Z$, $D$, and $E$ appearing in Lemmas \ref{l000}, \ref{fit}, and \ref{al2}. According to Remark \ref{rqr} we split the result in the regular and singular parts, denoted by the obvious subscript. In order to improve readability we will use in this section the simplified notation $A$ and $B$ for the two terms $A_{0,0}$ and $A_{0,1}$ defined in equation (\ref{fi2}). We observe that, in all cases, when computing the regular part,  all the terms that are equal in $B_q(0)$ and $\hat B_{q,\pm}(0)$ and in $A_q(0)$ and $\hat A_{q,\pm}(0)$ respectively, cancel in the final formula for $t_{0,{\rm reg}}^{(m)}(0)$, and similarly for the others. We thus introduce some regularized terms, denoted by a slashed letter, where only the relevant parts appear. 

\subsection{The cone} We split the calculation into three parts: the regular contribution, the singular contribution and the contribution of the harmonics. We give all details for the cone.

\subsubsection{The contribution of the regular part} 
\label{contregcone} 

By Theorem \ref{t4}, we have
\begin{align*}
Z_{q,{\rm  reg}}(0)&=-B_q(0), &
Z_{q,{\rm  reg}}'(0)&=-A_q(0)-B'_q(0),\\
\hat Z_{q,\pm, {\rm reg}}(0)&=-\hat B_{q,\pm}(0),&
\hat Z_{q,\pm,{\rm reg}}'(0)&=-\hat A_{q,\pm}(0)-\hat B'_{q,\pm}(0),\\
\end{align*}
where, by equation (\ref{aa}),
\begin{align*}
A_q(s)&=\sum_{n=1}^\infty m_{{\rm cex},q,n}\left(\frac{1}{2}\log 2\pi+\left(\mu_{q,n}-\frac{1}{2}\right)\log\mu_{q,n}-\log 2^{\mu_{q,n}}\Gamma(\mu_{q,n})\right) \mu_{q,n}^{-2s},\\
B_q(s)&=\frac{1}{2}\sum_{n=1}^\infty m_{{\rm cex},q,n}\left(\mu_{q,n}+\frac{1}{2}\right)\mu_{q,n}^{-2s}, \\
B_q'(s)&=-\sum_{n=1}^\infty m_{{\rm cex},q,n}\left(\mu_{q,n}+\frac{1}{2}\right)\mu_{q,n}^{-2s}\log \mu_{q,n}, \\
\hat A_{q,\pm}(s)&=\sum_{n=1}^\infty m_{{\rm cex},q,n} \left(\frac{1}{2}\log 2\pi+\left(\mu_{q,n}-\frac{1}{2}\right)\log\mu_{q,n}-\log
2^{\mu_{q,n}}\Gamma(\mu_{q,n})+\log\left(1\pm\frac{\alpha_q}{\mu_{q,n}}\right)\right)\mu_{q,n}^{-2s},\\
\hat B_{q,\pm}(s)&=\frac{1}{2}\sum_{n=1}^\infty m_{{\rm cex},q,n}\left(\mu_{q,n}-\frac{1}{2}\right)\mu_{q,n}^{-2s},\\
\hat B_{q,\pm}'(s)&=-\sum_{n=1}^\infty m_{{\rm cex},q,n}\left(\mu_{q,n}-\frac{1}{2}\right)\mu_{q,n}^{-2s}\log\mu_{q,n}.
\end{align*}

This gives (in the following formulas we are taking the finite part)
\begin{align*}
A_q(0)&=\sum_{n=1}^\infty m_{{\rm cex},q,n}\left(\frac{1}{2}\log 2\pi+\left(\mu_{q,n}-\frac{1}{2}\right)\log\mu_{q,n}-\log 2^{\mu_{q,n}}\Gamma(\mu_{q,n})\right),\\
B_q(0)&=\frac{1}{2}\sum_{n=1}^\infty m_{{\rm cex},q,n}\left(\mu_{q,n}+\frac{1}{2}\right), \\
B_q'(0)&=-\sum_{n=1}^\infty m_{{\rm cex},q,n}\left(\mu_{q,n}+\frac{1}{2}\right)\log \mu_{q,n}, \\
\hat A_{q,\pm}(0)&=\sum_{n=1}^\infty m_{{\rm cex},q,n}\left(\frac{1}{2}\log 2\pi+\left(\mu_{q,n}-\frac{1}{2}\right)\log\mu_{q,n}-\log
2^{\mu_{q,n}}\Gamma(\mu_{q,n})+\log\left(1\pm\frac{\alpha_q}{\mu_{q,n}}\right)\right),\\
\hat B_{q,\pm}(0)&=\frac{1}{2}\sum_{n=1}^\infty m_{{\rm cex},q,n}\left(\mu_{q,n}-\frac{1}{2}\right),\\
\hat B_{q,\pm}'(0)&=-\sum_{n=1}^\infty m_{{\rm cex},q,n}\left(\mu_{q,n}-\frac{1}{2}\right)\log\mu_{q,n}.
\end{align*}

Thus, considering first $t_{0,{\rm reg}}^{(m)}$, by Lemma \ref{l000}, we have
\[
\begin{aligned}
{t_{0,{\rm reg}}^{(m)}}'(0)=& \sum^{\left[\frac{m}{2}\right]-1}_{q=0} (-1)^q \left((Z_{q,{\rm reg}}(0)-\hat Z_{q,+,{\rm reg}}(0))
+(-1)^{m-1}(Z_{q,{\rm reg}}(0)-\hat Z_{q,-,{\rm reg}}(0))\right)\log l,\\
&+\frac{1}{2}\sum^{\left[\frac{m}{2}\right]-1}_{q=0} (-1)^q \left((Z'_{q,{\rm reg}}(0)-\hat Z'_{q,+,{\rm reg}}(0))
+(-1)^{m-1}(Z'_{q,{\rm reg}}(0)-\hat Z'_{q,-,{\rm reg}}(0))\right).
\end{aligned}
\]
and hence
\begin{align*}
{t_{0,{\rm reg}}^{(m)}}'(0)=&-\sum^{\left[\frac{m}{2}\right]-1}_{q=0} (-1)^q \left((B_{q}(0)-\hat B_{q,+}(0))
+(-1)^{m-1}(B_{q}(0)-\hat B_{q,-}(0))\right)\log l\\
&-\frac{1}{2}\sum^{\left[\frac{m}{2}\right]-1}_{q=0} (-1)^q \left(A_{q}(0)+B_q'(0)-\hat A_{q,+}(0)-\hat B_{q,+}'(0)\right)\\
&-(-1)^{m-1}\frac{1}{2}\sum^{\left[\frac{m}{2}\right]-1}_{q=0} (-1)^q\left(A_{q}(0)+B_q'(0)-\hat A_{q,-}(0)-\hat B_{q,-}'(0))\right).
\end{align*}

We observe that all the terms that are equal in $B_q(0)$ and $\hat B_{q,\pm}(0)$ and in $A_q(0)$ and $\hat A_{q,\pm}(0)$ respectively, cancel in the final formula for $t_{0,{\rm reg}}^{(m)}(0)$. This is because the final formulas are 
\begin{align*}
(Z'_{q,{\rm reg}}(0)-\hat Z_{q,+,{\rm reg}}(0))+(Z'_{q,{\rm reg}}(0)-\hat Z_{q,-,{\rm reg}}(0))&=2Z'_{q,{\rm reg}}(0)-\hat Z_{q,+,{\rm reg}}(0)-\hat Z_{q,-,{\rm reg}}(0),\\
(Z_{q,{\rm reg}}(0)-\hat Z_{q,+,{\rm reg}}(0))+(Z_{q,{\rm reg}}(0)-\hat Z_{q,-,{\rm reg}}(0))&=2Z_{q,{\rm reg}}(0)-\hat Z_{q,+,{\rm reg}}(0)-\hat Z_{q,-,{\rm reg}}(0),
\end{align*}
if $m=2p-1$ is odd, and
\begin{align*}
(Z'_{q,{\rm reg}}(0)-\hat Z_{q,+,{\rm reg}}(0))-(Z'_{q,{\rm reg}}(0)-\hat Z'_{q,-,{\rm reg}}(0))&=-\hat Z'_{q,+,{\rm reg}}(0)+\hat Z'_{q,-,{\rm reg}}(0),\\
(Z_{q,{\rm reg}}(0)-\hat Z_{q,+,{\rm reg}}(0))-(Z_{q,{\rm reg}}(0)-\hat Z_{q,-,{\rm reg}}(0))&=-\hat Z_{q,+,{\rm reg}}(0)+\hat Z_{q,-,{\rm reg}}(0),
\end{align*}
if $m=2p$ is even. 
Therefore, we can rewrite
\begin{align*}
{t_{0,{\rm reg}}^{(m)}}'(0)=&-\sum^{\left[\frac{m}{2}\right]-1}_{q=0} (-1)^q \left((\B_{q}(0)-\hat \B_{q,+}(0))
+(-1)^{m-1}(\B_{q}(0)-\hat \B_{q,-}(0))\right)\log l\\
&-\frac{1}{2}\sum^{\left[\frac{m}{2}\right]-1}_{q=0} (-1)^q \left(\A_{q}(0)+\B_q'(0)-\hat \A_{q,+}(0)-\hat \B_{q,+}'(0)\right)\\
&-(-1)^{m-1}\frac{1}{2}\sum^{\left[\frac{m}{2}\right]-1}_{q=0} (-1)^q\left(\A_{q}(0)+\B_q'(0)-\hat \A_{q,-}(0)-\hat \B_{q,-}'(0))\right).
\end{align*}
where
\begin{align*}
\A_q(0)&=0,\\
\B_q(0)&=\frac{1}{4}\sum_{n=1}^\infty m_{{\rm cex},q,n}=\frac{1}{4}\zeta_{\rm cex}(0,\tilde\Delta^{(q)}+\alpha_q^2)=\frac{1}{4}\zeta(0,\tilde\Delta^{(q)}), \\
\B_q'(0)&=-\frac{1}{2}\sum_{n=1}^\infty m_{{\rm cex},q,n}\log \mu_{q,n}, \\
\hat \A_{q,\pm}(0)&=\sum_{n=1}^\infty m_{{\rm cex},q,n}\log\left(1\pm\frac{\alpha_q}{\mu_{q,n}}\right),\\
\hat \B_{q,\pm}(0)&=-\frac{1}{4}\sum_{n=1}^\infty m_{{\rm cex},q,n}=-\frac{1}{4}\zeta_{\rm cex}(0,\tilde\Delta^{(q)}+\alpha_q^2)=-\frac{1}{4}\zeta_{\rm cex}(0,\tilde\Delta^{(q)}),\\
\hat \B_{q,\pm}'(0)&=\frac{1}{2}\sum_{n=1}^\infty m_{{\rm cex},q,n}\log\mu_{q,n},
\end{align*}
and 
\begin{align*}
\sum_{n=1}^\infty m_{{\rm cex},q,n}&=\zeta_{\rm cex}(0,\tilde\Delta^{(q)}+\alpha_q^2),&
-2\sum_{n=1}^\infty m_{{\rm cex},q,n}\log \mu_{q,n}&=\zeta_{\rm cex}'(0,\tilde\Delta^{(q)}+\alpha_q^2).
\end{align*}

Thus,
\beq\label{t2.1}\begin{aligned}
{t_{0,{\rm reg}}^{(m)}}'(0)=&-\frac{1}{2}\sum^{\left[\frac{m}{2}\right]-1}_{q=0} (-1)^q \left(\zeta_{\rm cex}(0,\tilde\Delta^{(q)})
+(-1)^{m-1}\zeta_{\rm cex}(0,\tilde\Delta^{(q)})\right)\log l\\
&+\frac{1}{2}\sum^{\left[\frac{m}{2}\right]-1}_{q=0} (-1)^q \left(\sum_{n=1}^\infty m_{{\rm cex},q,n}\log\left(1+\frac{\alpha_q}{\mu_{q,n}}\right)+\sum_{n=1}^\infty m_{{\rm cex},q,n}\log \mu_{q,n}\right)\\
&+(-1)^{m-1}\frac{1}{2}\sum^{\left[\frac{m}{2}\right]-1}_{q=0} (-1)^q\left(\sum_{n=1}^\infty m_{{\rm cex},q,n}\log\left(1-\frac{\alpha_q}{\mu_{q,n}}\right)+\sum_{n=1}^\infty m_{{\rm cex},q,n}\log \mu_{q,n}\right).
\end{aligned}
\eeq
Next, consider $t_{1,{\rm reg}}^{(2p-1)}$. By Lemma \ref{l000}, we have
\[
{t_{1,{\rm reg}}^{(2p-1)}}'(0)= -(-1)^{p-1} \left(B_{p-1}(0)-\hat B_{p-1,0}(0)\right)\log l\\
- \frac{1}{2} \left(A_{p-1}(0)+B'_{p-1}(0)-\hat A_{p-1,0}(0)-\hat B'_{p-1,0}(0)\right).
\]

As above,  all the terms that are equal in $B_q(0)$ and $\hat B_{q,0}(0)$ and in $A_q(0)$ and $\hat A_{q,0}(0)$ respectively, cancel in the final formula for $t_{1,{\rm reg}}^{(2p-1)}(0)$. Therefore, we can rewrite
\begin{align*}
{t_{1,{\rm reg}}^{(2p-1)}}'(0)=& -(-1)^{p-1} \left(\B_{p-1}(0)-\hat \B_{p-1,0}(0)\right)\log l
- \frac{1}{2} \left(\A_{p-1}(0)+\B'_{p-1}(0)-\hat \A_{p-1,0}(0)-\hat \B'_{p-1,0}(0)\right).
\end{align*}
where
\begin{align*}
\hat\A_{p-1,0}(0)&=0,\\
\hat\B_{p-1,0}(0)&=-\frac{1}{4}\sum_{n=1}^\infty m_{{\rm cex},p-1,n}=-\frac{1}{4}\zeta_{\rm cex}(0,\tilde\Delta^{(p-1)}+\alpha_{p-1}^2)=-\frac{1}{4}\zeta_{\rm cex}(0,\tilde\Delta^{(p-1)}), \\
\hat\B_{p-1,0}'(0)&=\frac{1}{2}\sum_{n=1}^\infty m_{{\rm cex},p-1,n}\log \mu_{p-1,n}.
\end{align*}

Observing  that $\alpha_{p-1}=\frac{1}{2}(1+2p-2-2p+1)=0$,  $\mu_{p-1,n}=\lambda_{p-1,n}$, and hence
\[
\zeta_{\rm cex}(s,\tilde\Delta^{(p-1)}+\alpha_{p-1}^2)=\zeta_{\rm cex}(s,\tilde\Delta^{(p-1)})=\sum_{n=1}^\infty m_{{\rm cex},p-1,n}\lambda_{p-1,n}^{-2s},
\]
and
\[
\hat\B_{p-1,0}'(0)=\frac{1}{2}\sum_{n=1}^\infty m_{{\rm cex},p-1,n}\log \lambda_{p-1,n}=-\frac{1}{4}\zeta_{\rm cex}'(0,\tilde\Delta^{(p-1)}). \\
\]

Thus,
\beq\label{t2.2}\begin{aligned}
{t_{1,{\rm reg}}^{(2p-1)}}'(0)=& -(-1)^{p-1} \frac{1}{2}\zeta_{\rm cex}(0,\tilde\Delta^{(p-1)})\log l
-\frac{1}{4}(-1)^{p-1}\zeta_{\rm cex}'(0,\tilde\Delta^{(p-1)}).
\end{aligned}
\eeq

\subsubsection{The contribution of the singular part}
\label{singconesing}

By  Lemma \ref{tec2b}, the functions $\Phi_{q,2j-1}$ and $\Phi_{q,2j-1, \pm}$ are regular at $s=0$ when $m$ is odd, and the  functions $\Phi_{q,2j}$ and $\Phi_{q,2j, \pm}$ are regular at $s=0$ when $m$ is even. By the location of the poles of $\zeta(s,U)$ given at the beginning of Section \ref{calc1}, it follows that the  formula for the singular part in Theorem \ref{t4} reduces to the single sums (recall $\kappa=2$ and the zeta of $U$ is regular at $s=0$)
\begin{align*}
Z_{q,{\rm sing}}(0)&=0,& \hat Z_{q,\pm,{\rm sing}}(0)&=0,\\
Z_{q,{\rm sing}}'(0)&=\frac{1}{2}\sum_{j=1}^m \Phi_{q,j}(0)\Ru_{s=j}\zeta(s,U),& \hat Z'_{q,\pm,{\rm sing}}(0)&=\frac{1}{2}\sum_{j=1}^m \hat\Phi_{q,j,\pm}(0)\Ru_{s=j}\zeta(s,U).
\end{align*}
where the $\Phi$ are defined in equation (\ref{vvv}). 
Thus, considering first $t_{0,{\rm sing}}^{(m)}$, by Lemma \ref{l000}, we have
\[
\begin{aligned}
{t_{0,{\rm sing}}^{(m)}}'(0)&= \frac{1}{2}\sum^{\left[\frac{m}{2}\right]-1}_{q=0} (-1)^q \left((Z'_{q,{\rm sing}}(0)-\hat Z'_{q,+,{\rm sing}}(0))
+(-1)^{m-1}(Z'_{q,{\rm sing}}(0)-\hat Z'_{q,-,{\rm sing}}(0))\right),
\end{aligned}
\]
and hence
\[
\begin{aligned}
{t_{0,{\rm sing}}^{(m)}}'(0)&= \frac{1}{4}\sum^{\left[\frac{m}{2}\right]-1}_{q=0} (-1)^q \sum_{j=1}^m \left((\Phi_{q,j}(0)-\hat \Phi_{q,j,+}(0))
+(-1)^{m-1}(\Phi_{q,j}(0)-\hat\Phi_{q,j,-}(0))\right)\Ru_{s=j}\zeta(s,U),
\end{aligned}
\]

Second, consider ${t_{1,{\rm sing}}^{(2p-1)}}'(0)$ as given in Lemma \ref{l000}. We have
\[
{t_{1,{\rm sing}}^{(2p-1)}}'(0)= (-1)^{p-1}\frac{1}{2}  \left(  Z'_{p-1,{\rm sing}}(0)-\hat Z'_{p-1,0,{\rm sing}}(0) \right),
\]
and hence
\[
{t_{1,{\rm sing}}^{(2p-1)}}'(0)= (-1)^{p-1}\frac{1}{4}  \sum_{j=1}^m \left(  \Phi_{p-1,j}(0)-\hat \Phi_{p-1,j,0}(0) \right)\Ru_{s=j}\zeta(s,U).
\]

\subsubsection{The contribution of the harmonics}\label{har11}  The contribution of the harmonics can be computed using the formula in equation (\ref{p00}) at the end of Section \ref{s3}. By the definition in Lemma \ref{l000},
since
\[
z_{q,\pm}(s)=z(s,\pm\alpha_q,0,1),
\]
we have
\begin{align*}
t_2^{(m)}(s)&= \frac{l^{2s}}{2} \sum_{q=0}^{\left[\frac{m-1}{2}\right]} (-1)^{q+1}  m_{{\rm har},q}\left( z(s,-\alpha_{q-1},0,1)+(-1)^m z(s,-\alpha_{q},0,1)\right);
\end{align*}
by equations (\ref{p00}) and (\ref{p000}), 
\begin{align}
\nonumber{t_2^{(m)}}'(0)=& \sum_{q=0}^{\left[\frac{m-1}{2}\right]} (-1)^{q+1}  m_{{\rm har},q}\left( z(0,-\alpha_{q-1},0,1)+(-1)^m z(0,-\alpha_{q},0,1)\right)\log l\\
\nonumber&+\frac{1}{2} \sum_{q=0}^{\left[\frac{m-1}{2}\right]} (-1)^{q+1}  m_{{\rm har},q}\left( z'(0,-\alpha_{q-1},0,1)+(-1)^m z'(0,-\alpha_{q},0,1)\right)\\
\label{t1.1}=& -\frac{1}{2}\sum_{q=0}^{\left[\frac{m-1}{2}\right]} (-1)^{q+1}  m_{{\rm har},q}\left(\left(-\alpha_{q-1}+\frac{1}{2}\right) +(-1)^m \left(-\alpha_{q}+\frac{1}{2}\right)\right)\log l\\
\nonumber&+\frac{1}{2} \sum_{q=0}^{\left[\frac{m-1}{2}\right]} (-1)^{q+1}  m_{{\rm har},q}\left(\log \frac{2^{-\alpha_{q-1}-\frac{1}{2}}\Gamma(-\alpha_{q-1}+1)}{\sqrt{\pi}}+(-1)^m \log \frac{2^{-\alpha_{q}-\frac{1}{2}}\Gamma(-\alpha_{q}+1)}{\sqrt{\pi}}\right).
\end{align}

On the other side, 
\begin{align*}
t_3^{(2p)}(s)&= (-1)^{p+1} \frac{l^{2s}}{4} m_{{\rm har},p}\left(z(s,\alpha_p,0,1)+z(s,-\alpha_p,0,1)\right).
\end{align*}

Whence
\begin{align*}
{t_3^{(2p)}}'(0)=& (-1)^{p+1}  m_{{\rm har},p}\frac{1}{2}\left(z(0,\alpha_p,0,1)+z(0,-\alpha_p,0,1)\right)\log l\\
&+(-1)^{p+1} m_{{\rm har},p} \frac{1}{4}\left(z'(s,\alpha_p,0,1)+z'(s,-\alpha_p,0,1)\right)\\
=& -(-1)^{p+1} \frac{1}{4} m_{{\rm har},p}\left(\left(\alpha_{p}+\frac{1}{2}\right)
+ \left(-\alpha_{p}+\frac{1}{2}\right)\right)\log l\\
&+(-1)^{p+1} \frac{1}{4} m_{{\rm har},p}\left(\log \frac{2^{\alpha_{p}-\frac{1}{2}}\Gamma(\alpha_{p}+1)}{\sqrt{\pi}}+\log \frac{2^{-\alpha_{p}-\frac{1}{2}}\Gamma(-\alpha_{p}+1)}{\sqrt{\pi}}\right)\\
=&(-1)^{p} \frac{1}{4}m_{{\rm har},p}\log l+(-1)^{p+1} \frac{1}{4} m_{{\rm har},p}\log\frac{\alpha_p}{2\sin (\pi\alpha_p)}.
\end{align*}

Since $\alpha_p=\frac{1}{2}(1+2p-2p)=\frac{1}{2}$, this gives
\begin{align}
\label{t1.2}{t_3^{(2p)}}'(0)=&(-1)^{p} \frac{1}{4}m_{{\rm har},p}\log l+(-1)^{p} \frac{1}{2} m_{{\rm har},p}\log 2.
\end{align}

\subsection{The frustum with mixed BC} We omit the details, since the calculations are similar to the one performed for the cone. As above, we split into three parts.

\subsubsection{ The contribution of the regular part.} According to Theorem \ref{t4}. we have
\[
\hat D'_{q,\pm,{\rm reg}}(0;l_1,l_2)=-\hat A_{q,\pm}(0;l_1,l_2)-\hat B'_{q,\pm}(0;l_1,l_2),
\]
where, according to equation (\ref{loomix}) 
\begin{align*}
\hat A_{\pm}(s;l_1,l_2))&=\sum_{n=1}^\infty m_{q,{\rm cex}}\left(- \frac{1}{2}\log\frac{l_2}{l_1}+ \log \left(\left(\frac{l_2^{\mu_{q,n}}}{l_1^{\mu_{q,n}}}+\frac{l_1^{\mu_{q,n}}}{l_2^{\mu_{q,n}}}\right)\pm\frac{\alpha_q}{\mu_{q,n}}\left(\frac{l_2^{\mu_{q,n}}}{l_1^{\mu_{q,n}}} - \frac{l_1^{\mu_{q,n}}}{l_2^{\mu_{q,n}}}\right)\right)\right)\mu_{q,n}^{-2s}, \\
\hat B_{\pm}(s;l_1,l_2))&=0,
\end{align*}
and hence
\begin{align*}
\hat A_{\pm}(0;l_1,l_2))&=\sum_{n=1}^\infty m_{q,{\rm cex}}\left(- \frac{1}{2}\log\frac{l_2}{l_1}+ \log \left(\left(\frac{l_2^{\mu_{q,n}}}{l_1^{\mu_{q,n}}}+\frac{l_1^{\mu_{q,n}}}{l_2^{\mu_{q,n}}}\right)\pm\frac{\alpha_q}{\mu_{q,n}}\left(\frac{l_2^{\mu_{q,n}}}{l_1^{\mu_{q,n}}} - \frac{l_1^{\mu_{q,n}}}{l_2^{\mu_{q,n}}}\right)\right)\right), \\
\hat B_{\pm}(0;l_1,l_2))&=0,
\end{align*}

Now, note that
\[
\hat A_{-}(0;l_2,l_1))-\hat A_{+}(0;l_1,l_2))=\sum_{n=1}^\infty m_{q,{\rm cex}}\log\frac{l_2}{l_1}=\zeta_{\rm cex}(0,\tilde \Delta^{(q)})\log\frac{l_2}{l_1},
\]
and hence, by Lemma \ref{fit},
\beq
\label{w0}
\begin{aligned}
{w_{0,{\rm reg}}^{(m)}}'(0)&=\frac{1}{2}\sum^{\left[\frac{m}{2}\right]-1}_{q=0} (-1)^q \left((\hat D'_{q,-,{\rm reg}}(0;l_2,l_1)-\hat D'_{q,+,{\rm reg}}(0;l_1,l_2))\right.\\
&\left.+(-1)^{m-1}(\hat D'_{q,+,{\rm reg}}(0;l_2,l_1)-\hat D'_{q,-,{\rm reg}}(0;l_1,l_2))\right)\\
&=-\frac{1}{2}\sum^{\left[\frac{m}{2}\right]-1}_{q=0} (-1)^q \left((\hat A_{q,-}(0;l_2,l_1)-\hat A_{q,+}(0;l_1,l_2))\right.\\
&\left.+(-1)^{m-1}(\hat A_{q,+}(0;l_2,l_1)-\hat A_{q,-}(0;l_1,l_2))\right)\\
&=- \frac{1}{2}\left(\log\frac{l_2}{l_1}\right)\sum^{\left[\frac{m}{2}\right]-1}_{q=0} (-1)^q (\zeta_{\rm cex}(0,\tilde \Delta^{(q)})+(-1)^{m+1}\zeta_{\rm cex}(0,\tilde \Delta^{(q)})),\\
\end{aligned}
\eeq
and
\beq\label{w1}
\begin{aligned}
{w_{1,{\rm reg}}^{(2p-1)}}'(0)&= (-1)^{p-1}\frac{1}{2}  \left(  \hat D'_{p-1,0,{\rm reg}}(0;l_2,l_1)-\hat D'_{p-1,0,{\rm reg}}(0;l_1,l_2) \right)\\
&= (-1)^{p}\frac{1}{2}  \log\frac{l_2}{l_1} \zeta_{\rm cex}(0,\tilde \Delta^{(p-1)}).
\end{aligned}
\eeq

\subsubsection{The contribution of the singular part.} Since the functions  $\Phi_{q,2j-1}$ and $\Phi_{q,2j-1, \pm}$ are regular at $s=0$, and by the location of the poles of $\zeta(s,U)$ given at the beginning of Section \ref{calc1}, it follows that the  formula for the singular part in Theorem \ref{t4} reduces to the single sums (recall $\kappa=2$ and the zeta of $U$ is regular at $s=0$)
\begin{align*}
\hat D_{q,\pm,{\rm sing}}(0;l_1,l_2)&=0,& \hat D'_{q,\pm,{\rm sing}}(0;l_1,l_2)&=\frac{1}{2}\sum_{j=1}^m \hat\Phi_{q,j,\pm}(0;l_1,l_2)\Ru_{s=j}\zeta(s,U).
\end{align*}
where the $\Phi$ are defined in equation (\ref{phiphi}). 
Thus, considering first $w_{0,{\rm sing}}^{(m)}$, by  Lemma \ref{fit}, we have
\[
\begin{aligned}
{w_{0,{\rm sing}}^{(m)}}'(0)=& \frac{1}{2}\sum^{\left[\frac{m}{2}\right]-1}_{q=0} (-1)^q \left((\hat D'_{q,-,{\rm sing}}(0;l_2,l_1)-\hat D'_{q,+,{\rm sing}}(0;l_1,l_2))\right.\\
&\left.+(-1)^{m-1}(\hat D'_{q,+,{\rm sing}}(0;l_2,l_1)-\hat D'_{q,-,{\rm sing}}(0;l_1,l_2))\right),
\end{aligned}
\]
and hence
\beq\label{wsin0}
\begin{aligned}
{w_{0,{\rm sing}}^{(m)}}'(0)=& \frac{1}{4}\sum^{\left[\frac{m}{2}\right]-1}_{q=0} (-1)^q \sum_{j=1}^m \left((\hat\Phi_{q,j,-}(0;l_2,l_1)-\hat \Phi_{q,j,+}(0;l_1,l_2))\right.\\
&\left.+(-1)^{m-1}(\hat \Phi_{q,j,+}(0;l_2,l_1)-\hat\Phi_{q,j,-}(0;l_1,l_2))\right)\Ru_{s=j}\zeta(s,U),
\end{aligned}
\eeq

Second, for ${w_{1,{\rm sing}}^{(2p-1)}}'(0)$, we have
\[
{w_{1,{\rm sing}}^{(2p-1)}}'(0)= (-1)^{p-1}\frac{1}{2}  \left(  \hat D'_{p-1,0,{\rm sing}}(0;l_2,l_1)-\hat D'_{p-1,0,{\rm sing}}(0;l_1,l_2) \right),
\]
and hence
\beq\label{wsin1}
{w_{1,{\rm sing}}^{(2p-1)}}'(0)= (-1)^{p-1}\frac{1}{4}  \sum_{j=1}^{2p-1} \left( \hat \Phi_{p-1,j,0}(0;l_2,l_1)-\hat \Phi_{p-1,j,0}(0;l_1,l_2) \right)\Ru_{s=j}\zeta(s,U).
\eeq


\subsubsection{ The contribution of the harmonics.} 
In order to determinate the contribution of the harmonics, we use the technique described in Section 2 of \cite{Spr9} to deal with (simple) sequences of spectral type (in fact this was already in \cite{Spr4}). Recall that, by Lemma \ref{fit}, we need $d'_{q}(0;l_1,l_2)$, where by definition 
\[
d_q(s;l_1,l_2)=m_{{\rm har},q}\sum_{k=1}^{\infty} f^{-2s}_{-\alpha_{q-1},k}(l_1,l_2),
\]
and the  $f^{-2s}_{\nu,k}(l_1,l_2)$ are the positive zeros of the function
\[
F_{\nu} (z;l_1,l_2)=  J_{\nu}(l_1 z) Y_{\nu-1}(l_2 z) - Y_{\nu}(l_1 z) J_{\nu-1}(l_2 z).
\]

Recalling the series definition of the Bessel functions, we obtain that near $z=0$,
\[
F_\nu(z) = \frac{2}{\pi z} \frac{l_2^{\nu-1}}{l_1^\nu}.
\]

This means that the function $F_\nu(z)$ is an even function of $z$.  By the Hadamard factorization Theorem, we have the product expansion
\[
F_\nu(z) = \frac{2}{\pi z} \frac{l_2^{\nu-1}}{l_1^\nu}{\prod_{k=-\infty}^{+\infty}}'\left(1-\frac{z}{f_{\nu,k}(l_1,l_2)}\right),
\]
and therefore
\[
F_\nu(z) = \frac{2}{\pi z} \frac{l_2^{\nu-1}}{l_1^\nu}{\prod_{k=1}^{\infty}}\left(1-\frac{z^2}{f^2_{\nu,k}(l_1,l_2)}\right),
\]

Defining, for $-\pi<\arg(z)<\frac{\pi}{2}$,  
\[  
G_\nu(z)=- F_\nu(i z), 
\]
we have
\begin{align*}
 G_\nu(z)&=\frac{2}{\pi z} \frac{l_2^{\nu-1}}{l_1^\nu}{\prod_{k=1}^{\infty}}\left(1+\frac{z^2}{f^2_{\nu,k}(l_1,l_2)}\right).
\end{align*}

By equation (\ref{gamma}), the logarithmic Gamma function associated to the simple sequence $\Theta_\nu(l_1,l_2)=\{f^2_{\nu,k}(l_1,l_2)\}$ is
\begin{align*}
\log \Gamma(-\lambda,\Theta_{\nu}(l_1,l_2))&=-\log\prod_{k=1}^\infty \left(1+\frac{(-\lambda)}{ f_{\nu,k}^2}\right)\\
&=-\log  G_{\nu}(\sqrt{-\lambda})  -\log \frac{\pi}{2}-  \log l_2\sqrt{\lambda}+\nu \log \frac{l_2}{l_1}.
\end{align*}

Combining the asymptotic expansions of the Bessel functions, we obtain for large $\lambda$:
\[
\log \Gamma(-\lambda,\Theta_{\nu}(l_1,l_2))=\left(\frac{1}{2}-\nu\right)\log \frac{l_1}{l_2}+\log 2-(l_2-l_1)\sqrt{\lambda}+O\left(\frac{1}{\sqrt{\lambda}}\right).
\]

Therefore, by Theorem 2.11 of \cite{Spr9}, 
\[
d'_q(0;l_1,l_2)=-m_{{\rm har},q}\left(\left(\frac{1}{2}+\alpha_{q-1}\right)\log \frac{l_1}{l_2}+\log 2\right),
\]
and hence, by  Lemma \ref{fit},
\beq\label{w2}
\begin{aligned}
{w_2^{(m)}}'(0)=& \frac{1}{2} \sum_{q=0}^{\left[\frac{m}{2}\right]} (-1)^{q+1}  \left(d'_{q}(0;l_1,l_2)+(-1)^m d_{q}(0;l_2,l_1)\right)\\
=& \frac{1}{2} \sum_{q=0}^{\left[\frac{m}{2}\right]} (-1)^{q}m_{{\rm  har}, q}  \left(\left(\frac{1}{2}+\alpha_{q-1}\right)\log \frac{l_1}{l_2}+\log 2\right.\\
&\left.+(-1)^m \left(\left(\frac{1}{2}+\alpha_{q-1}\right)\log \frac{l_2}{l_1}+\log 2\right)\right).
\end{aligned}
\eeq

\beq\label{w3}
\begin{aligned}
{w_3^{(2p)}}'(0)&=(-1)^{p+1} \frac{1}{2}d'_p(0;l_1,l_2)=(-1)^p\frac{1}{2} m_{{\rm  har}, p}\left(\left(\frac{1}{2}+\alpha_{p-1}\right)\log \frac{l_1}{l_2}+\log 2\right)\\
&=(-1)^p\frac{1}{2} m_{{\rm  har}, p}\log 2.
\end{aligned}
\eeq

\subsection{The furstum with absolute BC}

\subsubsection{The contribution of the regular part}

Proceeding as in Section \ref{calc2}, by Theorem \ref{t4}, we have
\begin{align*}
E_{q,{\rm  reg}}'(0)&=-A_q(0)-B'_q(0), &
\hat E_{q,\pm,{\rm reg}}'(0)&=-\hat A_{q,\pm}(0)-\hat B'_{q,\pm}(0),
\end{align*}
where, 
\begin{align*}
A_q(s)&=\sum_{n=1}^\infty m_{{\rm cex},q,n}\left(\log \left( \frac{l_2^{\mu_{q,n}}}{l_1^{\mu_{q,n}}}-\frac{l_1^{\mu_{q,n}}}{l_2^{\mu_{q,n}}}\right)+\frac{1}{2}\log l_1 l_2\right) \mu_{q,n}^{-2s},\\
B_q(s)&=\frac{1}{2}\sum_{n=1}^\infty m_{{\rm cex},q,n}\mu_{q,n}^{-2s}, \\
B_q'(s)&=-\sum_{n=1}^\infty m_{{\rm cex},q,n}\mu_{q,n}^{-2s}\log \mu_{q,n}, \\
\hat A_{q,\pm}(s)&=\sum_{n=1}^\infty m_{{\rm cex},q,n} \left(-\frac{1}{2}\log l_1l_2+\log\left(\frac{l_2^{\mu_{q,n}}}{l_1^{\mu_{q,n}}}-\frac{l_1^{\mu_{q,n}}}{l_2^{\mu_{q,n}}}\right)+\log \left(1-\frac{\alpha_q^2}{\mu^2_{q,n}}\right)\right)\mu_{q,n}^{-2s},\\
\hat B_{q,\pm}(s)&=-\frac{1}{2}\sum_{n=1}^\infty m_{{\rm cex},q,n}\mu_{q,n}^{-2s},\\
\hat B_{q,\pm}'(s)&=\sum_{n=1}^\infty m_{{\rm cex},q,n}\mu_{q,n}^{-2s}\log\mu_{q,n}.
\end{align*}

This gives (in the following formulas we are taking the finite part)
\begin{align*}
A_q(0)&=\sum_{n=1}^\infty m_{{\rm cex},q,n}\left(\log \left( \frac{l_2^{\mu_{q,n}}}{l_1^{\mu_{q,n}}}-\frac{l_1^{\mu_{q,n}}}{l_2^{\mu_{q,n}}}\right)+\frac{1}{2}\log l_1 l_2\right) ,\\
B_q(0)&=\frac{1}{2}\sum_{n=1}^\infty m_{{\rm cex},q,n}, \\
B_q'(0)&=-\sum_{n=1}^\infty m_{{\rm cex},q,n}\log \mu_{q,n}, \\
\hat A_{q,\pm}(0)&=\sum_{n=1}^\infty m_{{\rm cex},q,n} \left(-\frac{1}{2}\log l_1l_2+\log\left(\frac{l_2^{\mu_{q,n}}}{l_1^{\mu_{q,n}}}-\frac{l_1^{\mu_{q,n}}}{l_2^{\mu_{q,n}}}\right)+\log \left(1-\frac{\alpha_q^2}{\mu^2_{q,n}}\right)\right),\\
\hat B_{q,\pm}(0)&=-\frac{1}{2}\sum_{n=1}^\infty m_{{\rm cex},q,n},\\
\hat B_{q,\pm}'(0)&=\sum_{n=1}^\infty m_{{\rm cex},q,n}\log\mu_{q,n}.
\end{align*}

Thus, considering first $y_{0,{\rm reg}}^{(m)}$, we have
\[
\begin{aligned}
{y_{0,{\rm reg}}^{(m)}}'(0)=&\frac{1}{2}\sum^{\left[\frac{m}{2}\right]-1}_{q=0} (-1)^q \left((E'_{q,{\rm reg}}(0)-\hat E'_{q,+,{\rm reg}}(0))
+(-1)^{m-1}(E'_{q,{\rm reg}}(0)-\hat E'_{q,-,{\rm reg}}(0))\right).
\end{aligned}
\]
and hence
\begin{align*}
{y_{0,{\rm reg}}^{(m)}}'(0)=&-\frac{1}{2}\sum^{\left[\frac{m}{2}\right]-1}_{q=0} (-1)^q \left(A_{q}(0)+B_q'(0)-\hat A_{q,+}(0)-\hat B_{q,+}'(0)\right)\\
&-(-1)^{m-1}\frac{1}{2}\sum^{\left[\frac{m}{2}\right]-1}_{q=0} (-1)^q\left(A_{q}(0)+B_q'(0)-\hat A_{q,-}(0)-\hat B_{q,-}'(0))\right).
\end{align*}

We observe that all the terms that are equal in $B_q(0)$ and $\hat B_{q,\pm}(0)$ and in $A_q(0)$ and $\hat A_{q,\pm}(0)$ respectively cancel in the final formula for $t_{0,{\rm reg}}^{(m)}(0)$. This is because the final formulas are 
\begin{align*}
(E'_{q,{\rm reg}}(0)-\hat E_{q,+,{\rm reg}}(0))+(E'_{q,{\rm reg}}(0)-\hat E_{q,-,{\rm reg}}(0))&=2E'_{q,{\rm reg}}(0)-\hat E_{q,+,{\rm reg}}(0)-\hat E_{q,-,{\rm reg}}(0),
\end{align*}
if $m=2p-1$ is odd, and
\begin{align*}
(E'_{q,{\rm reg}}(0)-\hat E_{q,+,{\rm reg}}(0))-(E'_{q,{\rm reg}}(0)-\hat E'_{q,-,{\rm reg}}(0))&=-\hat E'_{q,+,{\rm reg}}(0)+\hat E'_{q,-,{\rm reg}}(0),
\end{align*}
if $m=2p$ is even. 
Therefore, we can rewrite
\begin{align*}
{y_{0,{\rm reg}}^{(m)}}'(0)=&-\frac{1}{2}\sum^{\left[\frac{m}{2}\right]-1}_{q=0} (-1)^q \left(\A_{q}(0)+\B_q'(0)-\hat \A_{q,+}(0)-\hat \B_{q,+}'(0)\right)\\
&-(-1)^{m-1}\frac{1}{2}\sum^{\left[\frac{m}{2}\right]-1}_{q=0} (-1)^q\left(\A_{q}(0)+\B_q'(0)-\hat \A_{q,-}(0)-\hat \B_{q,-}'(0))\right).
\end{align*}
where
\begin{align*}
\A_q(0)&=\frac{1}{2}\log l_1 l_2 \sum_{n=1}^\infty m_{{\rm cex},q,n} = \frac{1}{2}\zeta(0,\tilde\Delta^{(q)})\log l_1 l_2 ,\\
\B_q'(0)&=-\sum_{n=1}^\infty m_{{\rm cex},q,n}\log \mu_{q,n}, \\
\hat \A_{q,\pm}(0)&=\sum_{n=1}^\infty m_{{\rm cex},q,n}\left(-\frac{1}{2}\log l_1l_2 + \log\left(1-\frac{\alpha^2_q}{\mu^2_{q,n}}\right)\right)\\
&=-\frac{1}{2}\zeta(0,\tilde\Delta^{(q)})\log l_1 l_2  + \sum_{n=1}^\infty m_{{\rm cex},q,n}\log\left(1-\frac{\alpha^2_q}{\mu^2_{q,n}}\right),\\
\hat \B_{q,\pm}'(0)&=\sum_{n=1}^\infty m_{{\rm cex},q,n}\log\mu_{q,n},
\end{align*}
and 
\begin{align*}
\sum_{n=1}^\infty m_{{\rm cex},q,n}&=\zeta_{\rm cex}(0,\tilde\Delta^{(q)}+\alpha_q^2),&
-2\sum_{n=1}^\infty m_{{\rm cex},q,n}\log \mu_{q,n}&=\zeta_{\rm cex}'(0,\tilde\Delta^{(q)}+\alpha_q^2).
\end{align*}

Thus,
\beq\label{ly2}\begin{aligned}
{y_{0,{\rm reg}}^{(m)}}'(0)=&-\frac{1}{2} \sum^{\left[\frac{m}{2}\right]-1}_{q=0} (-1)^q\left(\zeta(0,\tilde\Delta^{(q)})+
(-1)^{m-1}\zeta(0,\tilde\Delta^{(q)})\right) \log l_1 l_2\\
&+\frac{1}{2}\sum^{\left[\frac{m}{2}\right]-1}_{q=0} (-1)^q \left(\sum_{n=1}^\infty m_{{\rm cex},q,n}\log\left(-\frac{\alpha_q^2}{\mu_{q,n}^2}\right)  + 2\sum_{n=1}^\infty m_{{\rm cex},q,n}\log\mu_{q,n}\right)\\
&+(-1)^{m-1}\frac{1}{2}\sum^{\left[\frac{m}{2}\right]-1}_{q=0} (-1)^q\left(\sum_{n=1}^\infty m_{{\rm cex},q,n}\log\left(1-\frac{\alpha_q^2}{\mu_{q,n}^2}\right)  + 2\sum_{n=1}^\infty m_{{\rm cex},q,n}\log\mu_{q,n}\right).
\end{aligned}
\eeq

Next, consider $y_{1,{\rm reg}}^{(2p-1)}$. By Lemma \ref{l000}, we have
\[
{y_{1,{\rm reg}}^{(2p-1)}}'(0)= - \frac{1}{2} \left(A_{p-1}(0)+B'_{p-1}(0)-\hat A_{p-1,0}(0)-\hat B'_{p-1,0}(0)\right).
\]

As above,  all the terms that are equal in $B_q(0)$ and $\hat B_{q,0}(0)$ and in $A_q(0)$ and $\hat A_{q,0}(0)$ respectively cancel in the final formula for $t_{1,{\rm reg}}^{(2p-1)}(0)$. Therefore, we can write
\begin{align*}
{y_{1,{\rm reg}}^{(2p-1)}}'(0)=& - \frac{1}{2} \left(\A_{p-1}(0)+\B'_{p-1}(0)-\hat \A_{p-1,0}(0)-\hat \B'_{p-1,0}(0)\right),
\end{align*}
where
\begin{align*}
\A_{p-1}(0)&= \frac{1}{2}\zeta(0,\tilde\Delta^{(p-1)})\log l_1 l_2 , &
\B_{p-1}'(0)&=-\sum_{n=1}^\infty m_{{\rm cex},q,n}\log \mu_{q,n}, \\
\hat \A_{q,\pm}(0)&=-\frac{1}{2}\log l_1 l_2 \zeta(0,\tilde\Delta^{(q)}) ,&
\hat \B_{q,\pm}'(0)&=\sum_{n=1}^\infty m_{{\rm cex},q,n}\log\mu_{q,n}.
\end{align*}

Observing  that $\alpha_{p-1}=\frac{1}{2}(1+2p-2-2p+1)=0$,  and $\mu_{p-1,n}=\lambda_{p-1,n}$, we have
\[
\zeta_{\rm cex}(s,\tilde\Delta^{(p-1)}+\alpha_{p-1}^2)=\zeta_{\rm cex}(s,\tilde\Delta^{(p-1)})=\sum_{n=1}^\infty m_{{\rm cex},p-1,n}\lambda_{p-1,n}^{-2s},
\]
and
\[
\hat\B_{p-1,0}'(0)=\sum_{n=1}^\infty m_{{\rm cex},p-1,n}\log \lambda_{p-1,n}=-\frac{1}{2}\zeta_{\rm cex}'(0,\tilde\Delta^{(p-1)}). \\
\]

Thus,
\beq\label{ly3}\begin{aligned}
{y_{1,{\rm reg}}^{(2p-1)}}'(0)=& -(-1)^{p-1} \frac{1}{2}\zeta_{\rm cex}(0,\tilde\Delta^{(p-1)})\log l_1 l_2
-\frac{1}{2}(-1)^{p-1}\zeta_{\rm cex}'(0,\tilde\Delta^{(p-1)}).
\end{aligned}
\eeq

\subsubsection{The contribution of the harmonics}  The contribution of the harmonics can be computed directly, proceeding as in Section \ref{har11}. We obtain (this formula holds for all cases in which $\alpha_q\not=0$, the particular case only appears when $m$ is odd and is described below)
\[
e'_{q,-}(0)= \log(-\alpha_q) - \log \left(\frac{l_2^{-\alpha_q}}{l_1^{-\alpha_q}}-\frac{l_1^{-\alpha_q}}{l_2^{-\alpha_q}}\right) - \frac{1}{2}\log l_1 l_2,
\]
and hence
\begin{align*}
\nonumber{y_2^{(m)}}'(0)=& \frac{1}{2} \sum_{q=0}^{\left[\frac{m-1}{2}\right]} (-1)^{q+1}  m_{{\rm har},q}\left( e'_{q-1,-}(0)+(-1)^m e'_{q,-}(0)\right)\\
\nonumber=&\frac{1}{2} \sum_{q=0}^{\left[\frac{m-1}{2}\right]} (-1)^{q+1}  m_{{\rm har},q}\log\left(\frac{m-2q+1}{2} \right)\left(\frac{m-2q-1}{2}\right)^{(-1)^m} \\
&+\frac{1}{2} \sum_{q=0}^{\left[\frac{m-1}{2}\right]} (-1)^{q}  m_{{\rm har},q}\log\left(\frac{l_2^{m-2q+1}-l_1^{m-2q+1}}{(l_1 l_2)^{\frac{m-2q+1}{2}}}  \right)\left(\frac{l_2^{m-2q-1}-l_1^{m-2q-1}}{(l_1 l_2)^{\frac{m-2q-1}{2}}}\right)^{(-1)^m} \\
\nonumber&+\frac{1}{2} \sum_{q=0}^{\left[\frac{m-1}{2}\right]} (-1)^{q}m_{{\rm har},q} \frac{1+(-1)^m}{2} \log l_1 l_2.
\end{align*}

It is convenient to distinguish odd and even cases, as in Section \ref{s8.1}. When, $m=2p-1$, we need to isolate the case $q=p-1$, when the correct formula is
\[
e'_{p-1,-}(0) = -\log\log\frac{l_2}{l_1} - \log 2 - \log \sqrt{l_1 l_2}.
\]

After some calculations, we obtain
\begin{align*}
\nonumber{y_2^{(2p-1)}}'(0)=& \frac{1}{2} \sum_{q=0}^{p-1} (-1)^{q}  m_{{\rm
har},q}  \log \frac{l_2^{2p-2q}-l_1^{2p-2q}}{2p-2q}-\frac{1}{2} \sum_{q=0}^{p-2} (-1)^{q}m_{{\rm
har},q}\log \frac{l_2^{2p-2q-2}-l_1^{2p-2q-2}}{2p-2q-2} \\
&+\frac{1}{2} \sum_{q=0}^{p-1} (-1)^{q+1}m_{{\rm har},q}\log l_1l_2 +\frac{1}{2}(-1)^p m_{{\rm
har},p}\log\log \frac{l_2}{l_1}.
\end{align*}

By duality on the section: 
\[
\sum_{q=0}^{p-2} (-1)^{q+1} r_{q} \log \frac{l_2^{2p-2-2q}- l_1^{2p-2-2q}}{2p-2-2q } = \sum_{q=p+1}^{2p-1} (-1)^{q}  m_{{\rm har},q}\log\frac{l_2^{2p-2q}-l_1^{2p-2q}}{2p-2q} - \sum_{q=0}^{p-2} (-1)^q r_q \log (l_1 l_2)^{2p-2-2q},
\]
and hence
\beq\label{h1}
\begin{aligned}
{y_2^{(2p-1)}}'(0)=& \frac{1}{2} \sum_{q=0,q\not= p}^{2p-1} (-1)^{q+1} 
m_{{\rm har},q} \log\frac{l_2^{2p-2q}-l_1^{2p-2q}}{2p-2q}\\
+&\frac{1}{2}\sum_{q=0}^{p-1} (-1)^{q+1}(2p-1-2q)m_{{\rm har},q} \log l_1 l_2 + \frac{1}{2}(-1)^p m_{{\rm har},p}\log\log \frac{l_2}{l_1}.
\end{aligned}
\eeq

When $m=2p$, we obtain
\begin{align*}
\nonumber{y_2^{(2p)}}'(0)
=&\frac{1}{2} \sum_{q=0}^{p-1} (-1)^{q+1}  m_{{\rm har},q}\left(\log\left(2p-2q+1\right) +\log\left(2p-2q-1\right) -2\log2\right)
\\
&+\frac{1}{2} \sum_{q=0}^{p-1} (-1)^{q}  \left(m_{{\rm har},q}\log\left(l_2^{2p-2q+1}-l_1^{2p-2q+1}\right)+ \log\left(l_2^{2p-2q-1}-l_1^{2p-2q-1}\right) - (2p-2q)\log l_1l_2\right)\\
&+\frac{1}{2} \sum_{q=0}^{p-1} (-1)^{q}m_{{\rm har},q} \log l_1 l_2.
\end{align*}

The last contribution is
\begin{align*}
{y_3^{(2p)}}'(0)&= (-1)^{p+1} \frac{1}{2} m_{{\rm har},p} e'_{p,-}(0)\\
&= (-1)^{p+1} \frac{1}{2} m_{{\rm har},p}\left(- \log 2 - \log \left(\frac{l_2^{\frac{1}{2}}}{l_1^{\frac{1}{2}}}-\frac{l_1^{\frac{1}{2}}}{l_2^{\frac{1}{2}}}\right) - \frac{1}{2}\log l_1 l_2
 \right)\\
 &=(-1)^{p}\frac{1}{2}m_{{\rm har},p}\left(\log (l_2 - l_1) + \log 2\right).
\end{align*}


Collecting
\begin{align*}
{y_2^{(2p)}}'(0)+{y_3^{(2p)}}'(0)=& \sum_{q=0}^{p-1} (-1)^{q}  \log2 + \frac{1}{2}(-1)^{p}m_{{\rm har},p}\log 2
\\
&+\frac{1}{2} \sum_{q=0}^{p} (-1)^{q}  m_{{\rm har},q}\log\frac{l_2^{2p-2q+1}-l_1^{2p-2q+1}}{2p-2q+1}+\frac{1}{2} \sum_{q=0}^{p-1} (-1)^q  m_{{\rm har},q}\log\frac{l_2^{2p-2q-1}-l_1^{2p-2q-1}}{2p-2q-1} \\
&+\frac{1}{2}\log l_1 l_2 \sum_{q=0}^{p-1} (-1)^{q+1}  m_{{\rm har},q}(2p-2q-1).
\end{align*} 

Using duality on the section
\begin{align*}
\sum_{q=0}^{p-1} (-1)^q  m_{{\rm har},q}\log\frac{l_2^{2p-2q-1}-l_1^{2p-2q-1}}{2p-2q-1}
=&\sum_{q=p+1}^{2p} (-1)^q r_{q} \log \frac{-1}{2p-2q+1}\left(\frac{1}{l_2^{2p-2q+1}} -\frac{1}{l_1^{2p-2q+1}}\right)\\
=&\sum_{q=p+1}^{2p} (-1)^q r_{q} \log \frac{l_2^{2p-2q+1} - l_1^{2p-2q+1}}{2p-2q+1}\\
&+ \sum_{q=0}^{p-1} (-1)^{q+1} r_q (2p-2q-1)\log l_1l_2,
\end{align*} 
and hence
\beq\label{h2}
{y_2^{(2p)}}'(0)+{y_3^{(2p)}}'(0)=\frac{1}{2} \chi(W)\log 2+ \frac{1}{2} \sum_{q=0}^{2p} (-1)^{q}  m_{{\rm har},q}\log\frac{l_2^{2p-2q+1}-l_1^{2p-2q+1}}{2p-2q+1}.
\eeq

\subsubsection{The contribution of the singular part} The calculations show that the
functions $\Phi$ appearing in the singular part are the same appearing in the
singular part for the cone and the frustum with mixed BC, as described in Sections
\ref{calc1} and \ref{calc2}. We obtain the following result
\begin{align*}
E_{q,{\rm sing}}(0)&=0,&E_{q,{\rm sing}}'(0)&=\frac{1}{2}\sum_{j=1}^m 
\Phi_{q,j}^{\rm Frustum, abs}(0)\Ru_{s=j}\zeta(s,U),\\
\hat E_{q,\pm,{\rm sing}}(0)&=0,&\hat E_{q,\pm,{\rm
sing}}'(0)&=\frac{1}{2}\sum_{j=1}^m  \hat \Phi_{q,j,\pm}^{\rm Frustum,
abs}(0)\Ru_{s=j}\zeta(s,U).
\end{align*}

Since, 
\[
\begin{aligned}
{y_{0,{\rm sing}}^{(m)}}'(0)=& \frac{1}{2}\sum^{\left[\frac{m}{2}\right]-1}_{q=0}
(-1)^q \left(( E'_{q,{\rm sing}}(0)-\hat E'_{q,+,{\rm sing}}(0))\right.\left.+(-1)^{m-1}( E'_{q,{\rm sing}}(0)-\hat E'_{q,-,{\rm sing}}(0))\right),
\end{aligned}
\]
we obtain
\beq\label{ysin0}
\begin{aligned}
{y_{0,{\rm sing}}^{(m)}}'(0)=& \frac{1}{4}\sum^{\left[\frac{m}{2}\right]-1}_{q=0}
(-1)^q \sum_{j=1}^m \left((\Phi^{\rm Frustum, abs}_{q,j}(0)-\hat \Phi^{\rm Frustum,
abs}_{q,j,+}(0))\right.\\
&\left.+(-1)^{m-1}( \Phi^{\rm Frustum, abs}_{q,j}(0)-\hat\Phi^{\rm Frustum,
abs}_{q,j,-}(0))\right)\Ru_{s=j}\zeta(s,U),
\end{aligned}
\eeq

If $m=2p-1$:
\begin{align*}
{y_{0,{\rm sing}}^{(2p-1)}}'(0)=& \frac{1}{4}\sum^{p-2}_{q=0}
(-1)^q \sum_{j=1}^{2p-1} \left(2\Phi^{\rm Frustum, abs}_{q,j}(0)-\hat \Phi^{\rm Frustum,abs}_{q,j,+}(0))-\hat\Phi^{\rm Frustum,
abs}_{q,j,-}(0)\right)\Ru_{s=j}\zeta(s,U),
\end{align*}
if $m=2p$:
\begin{align*}
{y_{0,{\rm sing}}^{(2p)}}'(0)=&-  \frac{1}{4}\sum^{p-1}_{q=0}
(-1)^q \sum_{j=1}^{2p} \left(\hat \Phi^{\rm Frustum,abs}_{q,j,+}(0))-\hat\Phi^{\rm Frustum,
abs}_{q,j,-}(0))\right)\Ru_{s=j}\zeta(s,U),
\end{align*}

For ${y_{1,{\rm sing}}^{(2p-1)}}'(0)$, we have
\[
{y_{1,{\rm sing}}^{(2p-1)}}'(0)= (-1)^{p-1}\frac{1}{2}  \left(  E'_{p-1,{\rm
sing}}(0)-\hat E'_{p-1,0,{\rm sing}}(0) \right),
\]
and hence
\beq\label{ysin1}
{y_{1,{\rm sing}}^{(2p-1)}}'(0)= (-1)^{p-1}\frac{1}{4}  \sum_{j=1}^m \left( 
\Phi^{\rm Frustum, abs}_{p-1,j}(0)-\hat \Phi^{\rm Frustum, abs}_{p-1,j,0}(0)
\right)\Ru_{s=j}\zeta(s,U).
\eeq

In particular, we have a result analogous to Lemma \ref{tec1}, that leads to the
following formula
\begin{align*}
&\Phi^{\rm Frustum, abs}_{q,j}(s)-\hat \Phi^{\rm Frustum, abs}_{q,j,+}(s)
+(-1)^{m-1}( \Phi^{\rm Frustum, abs}_{q,j}(s)-\hat\Phi^{\rm Frustum,
abs}_{q,j,-}(s)) =\\
&=\left(l_2^{2s}+(-1)^j l_1^{2s}\right)
\left(\Phi_{q,j}(s)-\hat\Phi_{q,j,+}(s)+(-1)^{m-1}\left(\Phi_{q,j}(s)-\hat\Phi_{q,j,-}(s)\right)\right),
\end{align*}
and therefore, if $m=2p-1$,
\begin{align*}
2\Phi^{\rm Frustum, abs}_{q,j}(0)-\hat \Phi^{\rm Frustum, abs}_{q,j,+}(0)
-\hat\Phi^{\rm Frustum,abs}_{q,j,-}(0)=\left(1+(-1)^j \right)
\left(2\Phi_{q,j}(0)-\hat\Phi_{q,j,+}(0)-\hat\Phi_{q,j,-}(0)\right),
\end{align*}
if $m=2p$,
\begin{align*}
\hat \Phi^{\rm Frustum, abs}_{q,j,+}(0)
-\hat\Phi^{\rm Frustum,abs}_{q,j,-}(0)) =\left(1+(-1)^j \right)
\left(\hat\Phi_{q,j,+}(0)-\hat\Phi_{q,j,-}(0)\right).
\end{align*}

Since in the odd case the relevant terms are those with odd index $j$, we have that
\[
{y_{0,{\rm sing}}^{(2p-1)}}'(0)+{y_{1,{\rm sing}}^{(2p-1)}}'(0)=0.
\]

In the even case, we obtain
\[
{y_{0,{\rm sing}}^{(2p)}}'(0)=2{t_{0,{\rm sing}}^{(2p-1)}}'(0).
\]

\section{The regular part of the torsion}
\label{regularpart}

We decompose the torsion in two parts,  the regular part and the anomaly boundary term, according to the formulas (\ref{www}), (\ref{www1}), and (\ref{ww}) of Section \ref{s4}. In this section we give the formulas for the regular part, for the cone and for the frustum. As a consequence, we have the proof of the formulas in equations (\ref{www}), and (\ref{www1}) namely that, in the case of the frustum,  the regular part of the torsion coincides with the Reidemeister torsion plus the Euler part of the boundary contribution, and the singular part is precisely the anomaly boundary contribution.

\subsection{The regular part of the torsion for the cone}\label{s8.1} We distinguish odd and even cases.

\subsubsection{Odd case: $m=2p-1$} $\left[\frac{m}{2}\right]-1=p-2$, and hence, using equation (\ref{t2.1}),
\begin{align*}
{t_{0,{\rm reg}}^{(2p-1)}}'(0)
=&-\sum^{p-2}_{q=0} (-1)^q \zeta_{\rm cex}(0,\tilde\Delta^{(q)})\log l\\
&+\frac{1}{2}\sum^{p-2}_{q=0} (-1)^q \left(\sum_{n=1}^\infty m_{{\rm cex},q,n}\log\left(1-\frac{\alpha_q^2}{\mu_{q,n}^2}\right)+\sum_{n=1}^\infty m_{{\rm cex},q,n}\log \mu_{q,n}^2\right).
\end{align*}

Since $\mu_{q,n}^2-\alpha_q^2=\lambda_{q,n}^2$,  we obtain
\begin{align*}
{t_{0,{\rm reg}}^{(2p-1)}}'(0)
=&-\sum^{p-2}_{q=0} (-1)^q \zeta_{\rm cex}(0,\tilde\Delta^{(q)})\log l+\frac{1}{2}\sum^{p-2}_{q=0} (-1)^q \left(\sum_{n=1}^\infty m_{{\rm cex},q,n}\log \lambda_{q,n}\right)\\
=&-\sum^{p-2}_{q=0} (-1)^q \zeta_{\rm cex}(0,\tilde\Delta^{(q)})\log l-\frac{1}{2}\sum^{p-2}_{q=0} (-1)^q \zeta_{\rm cex}'(0,\tilde\Delta^{(q)}).
\end{align*}

Next, from equation (\ref{t2.2})
\[
{t_{1,{\rm reg}}^{(2p-1)}}'(0)= -(-1)^{p-1} \frac{1}{2}\zeta_{\rm cex}(0,\tilde\Delta^{(p-1)})\log l
-(-1)^{p-1}\frac{1}{4}\zeta_{\rm cex}'(0,\tilde\Delta^{(p-1)}).
\]

Eventually, the contribution of the harmonics is given in equation (\ref{t1.1})
\begin{align*}
{t_2^{(2p-1)}}'(0)=& \frac{1}{2}\sum_{q=0}^{p-1} (-1)^{q+1}  m_{{\rm har},q}\left(\alpha_{q-1}-\alpha_{q}\right)\log l\\
&+\frac{1}{2} \sum_{q=0}^{p-1} (-1)^{q+1}  m_{{\rm har},q}\left(\log \frac{2^{-\alpha_{q-1}}\Gamma(-\alpha_{q-1}+1)}{2^{-\alpha_{q}}\Gamma(-\alpha_{q}+1)}\right).
\end{align*}

Since, by definition $\alpha_q=\frac{1}{2}(1+2q-2p+1)$ and $\alpha_{q-1}=\frac{1}{2}(1+2q-2-2p+1)$, 
we obtain
\begin{align*}
{t_2^{(2p-1)}}'(0)=& \frac{1}{2}\sum_{q=0}^{p-1} (-1)^{q}  m_{{\rm har},q}\log l
+\frac{1}{2} \sum_{q=0}^{p-1} (-1)^{q+1}  m_{{\rm har},q}\log(2(p-q)).
\end{align*}

Summing up, as in equation (\ref{ww}), we obtain
\begin{align*}
\log T_{\rm global}(C_l (W))&={t_{0,{\rm reg}}^{(2p-1)}}'(0)+{t_{1,{\rm reg}}^{(2p-1)}}'(0)+{t_2^{(2p-1)}}'(0)+{t_3^{(2p-1)}}'(0)\\
=&\left(\frac{1}{2}\sum_{q=0}^{p-1} (-1)^{q}  m_{{\rm har},q}-\sum^{p-2}_{q=0} (-1)^q \zeta_{\rm cex}(0,\tilde\Delta^{(q)})-(-1)^{p-1}\frac{1}{2}\zeta_{\rm cex}(0,\tilde\Delta^{(p-1)})\right)\log l\\
&+\frac{1}{2}\sum^{p-2}_{q=0} (-1)^{q+1} \zeta_{\rm cex}'(0,\tilde\Delta^{(q)})+\frac{1}{4}(-1)^{p}\zeta_{\rm cex}'(0,\tilde\Delta^{(p-1)})\\
&-\frac{1}{2} \sum_{q=0}^{p-1} (-1)^{q}  m_{{\rm har},q}\log(2(p-q)).
\end{align*}

Using equations (\ref{forfor}) and (\ref{Todd}), this formula can be rewritten in terms of harmonics:
\begin{align*}
\log T_{\rm global}(C_l (W))
=&\frac{1}{2}\sum_{q=0}^{p-1} (-1)^{q}  m_{{\rm har},q} \log \frac{l^{2p-2q}}{2p-2q} +\frac{1}{2}\log T(W,g).
\end{align*}


\subsubsection{Even case: $m=2p$}  $\left[\frac{m}{2}\right]-1=p-1$, and hence, using equation (\ref{t2.1}),
\begin{align*}
{t_{0,{\rm reg}}^{(2p)}}'(0)=&\frac{1}{2}\sum^{p-1}_{q=0} (-1)^q \sum_{n=1}^\infty m_{{\rm cex},q,n}\left(\log\left(1+\frac{\alpha_q}{\mu_{q,n}}\right)-\log\left(1-\frac{\alpha_q}{\mu_{q,n}}\right)\right).
\end{align*}

Next, the contribution of the harmonics is given in equation (\ref{t1.1})
\begin{align*}
{t_2^{(2p)}}'(0)=& -\frac{1}{2}\sum_{q=0}^{p-1} (-1)^{q+1}  m_{{\rm har},q}\left(-\alpha_{q-1}-\alpha_{q}+1\right)\log l\\
\nonumber&+\frac{1}{2} \sum_{q=0}^{p-1} (-1)^{q+1}  m_{{\rm har},q}\log \frac{2^{-\alpha_{q-1}-\alpha_q-1}\Gamma(-\alpha_{q-1}+1)\Gamma(-\alpha_{q}+1)}{\pi}.
\end{align*}

Since, by definition $\alpha_q=\frac{1}{2}(1+2q-2p)$,  and $\alpha_{q-1}=\frac{1}{2}(1+2q-2-2p)$, we obtain
\begin{align*}
{t_2^{(2p)}}'(0)
=& \frac{1}{2}\sum_{q=0}^{p-1} (-1)^{q}  m_{{\rm har},q}(2p-2q+1)\log l\\
&+\frac{1}{2} \sum_{q=0}^{p-1} (-1)^{q+1}  m_{{\rm har},q}\log (2^{-2}((2p-2q-1)!!)^2(2p-2q+1)).
\end{align*}

Eventually, from equation (\ref{t1.2})
\begin{align*}
{t_3^{(2p)}}'(0)=&(-1)^{p} \frac{1}{4}m_{{\rm har},p}\log l+(-1)^{p} \frac{1}{2} m_{{\rm har},p}\log 2.
\end{align*}

Summing up, as in equation (\ref{ww}), we obtain 
\begin{align*}
\log T_{\rm global}(C_l (W))+\frac{1}{4}\chi( \b C_l(W))\log 2=&{t_{0,{\rm reg}}^{(2p)}}'(0)+{t_{1,{\rm reg}}^{(2p)}}'(0)+{t_2^{(2p)}}'(0)+{t_3^{(2p)}}'(0)\\
=&\left(\frac{1}{2}\sum_{q=0}^{p-1} (-1)^{q}  m_{{\rm har},q}(2p-2q+1)+(-1)^{p} \frac{1}{4}m_{{\rm har},p}\right)\log l\\
&+\frac{1}{2} \sum_{q=0}^{p-1} (-1)^{q+1}  m_{{\rm har},q}\log (2p-2q+1)((2p-2q-1)!!)^2\\
&+\frac{1}{2}\chi(W)\log 2+\frac{1}{2}\sum^{p-1}_{q=0} (-1)^q \sum_{n=1}^\infty m_{{\rm cex},q,n}\log\frac{1+\frac{\alpha_q}{\mu_{q,n}}}{1-\frac{\alpha_q}{\mu_{q,n}}}.
\end{align*}


\subsection{The regular part of the torsion for the frustum with mixed BC}
\label{8.2} By equation (\ref{w0}), 
\[
{w_{0,{\rm reg}}^{(m)}}'(0)=- \frac{1}{2}\left(\log\frac{l_2}{l_1}\right)\sum^{\left[\frac{m}{2}\right]-1}_{q=0} (-1)^q (\zeta_{\rm cex}(0,\tilde \Delta^{(q)})+(-1)^{m+1}\zeta_{\rm cex}(0,\tilde \Delta^{(q)})),
\]
whence
\begin{align*}
{w_{0,{\rm reg}}^{(2p-1)}}'(0)&=- \log\left(\frac{l_2}{l_1}\right)\sum^{p-2}_{q=0} (-1)^q \zeta_{\rm cex}(0,\tilde \Delta^{(q)}),&
{w_{0,{\rm reg}}^{(2p)}}'(0)&=0;
\end{align*}
by equation (\ref{w1}),
\[
{w_{1,{\rm reg}}^{(2p-1)}}'(0)=  (-1)^{p}\frac{1}{2}  \log\left(\frac{l_2}{l_1}\right)\zeta_{\rm cex}(0,\tilde \Delta^{(p-1)});
\]
by equation (\ref{w2}) 
\[
{w_2^{(m)}}'(0)=\frac{1}{2} \sum_{q=0}^{\left[\frac{m}{2}\right]} (-1)^{q}m_{{\rm  har}, q}  \left(\left(\frac{1}{2}+\alpha_{q-1}\right)\log \frac{l_1}{l_2}+\log 2+(-1)^m \left(\left(\frac{1}{2}+\alpha_{q-1}\right)\log \frac{l_2}{l_1}+\log 2\right)\right),
\]
and hence
\begin{align*}
{w_2^{(2p-1)}}'(0)&=  \sum_{q=0}^{p-1} (-1)^{q+1} m_{{\rm har},q}\left(p-q-\frac{1}{2}\right)\log \frac{l_2}{l_1},&
{w_2^{(2p)}}'(0)&=\sum_{q=0}^p (-1)^q m_{{\rm har}, q} \log 2;
\end{align*}

By equation (\ref{w3}),
\[
{w_3^{(2p)}}'(0)=(-1)^p\frac{1}{2} m_{{\rm  har}, p}\log 2.
\]

Thus, when $m=2p$, 
\[
{w_{0,{\rm reg}}^{(2p)}}'(0)+{w_2^{(2p)}}'(0)+{w_3^{(2p)}}'(0)=\sum_{q=0}^p (-1)^q m_{{\rm har}, q} \log 2+(-1)^p \frac{1}{2}m_{{\rm  har}, p}\log 2=\frac{1}{2}\chi(W)\log 2.
\]

When $m=2p-1$, recalling equations (\ref{peppo}) and (\ref{zetazero}),
\[
\zeta_{\rm ccl}(0,\tilde\Delta^{(q)})=(-1)^q\sum_{k=0}^q (-1)^k \zeta(0,\tilde\Delta^{(k)})=-(-1)^q\sum_{k=0}^q (-1)^k m_{{\rm har},k},
\]
since the dimension of $W$ is odd; whence, after some calculation we obtain
\begin{align*}
{w_{0,{\rm reg}}^{(2p-1)}}'(0)+{w_2^{(2p-1)}}'(0)&=\frac{1}{2}\log\left(\frac{l_2}{l_1}\right)\sum_{q=0}^{p-1} (-1)^{q+1}m_{{\rm har},q},\\
{w_{1,{\rm reg}}^{(2p-1)}}'(0)&=  \frac{1}{2}  \log\left(\frac{l_2}{l_1}\right)\sum_{k=0}^{p-1} (-1)^k m_{{\rm har},k},
\end{align*}
and hence
\[
{w_{0,{\rm reg}}^{(2p-1)}}'(0)+{w_{1,{\rm reg}}^{(2p-1)}}'(0)+{w_2^{(m)}}'(0)=0.
\]

Therefore, for any parity of $m$, the regular part of the torsion is
\[
{w_{0,{\rm reg}}^{(m)}}'(0)+{w_{1,{\rm reg}}^{(m)}}'(0)+{w_2^{(m)}}'(0)+{w_3^{(m)}}'(0)=\frac{1}{2}\chi(W)\log 2,
\]
and this is equal to the Reidemeister torsion plus the Euler part of the boundary contribution, i.e.
\[
\log \tau(C_{[l_1,l_2]} (W), \{l_1\}\times W)+\frac{1}{4}\chi(\b C_{[l_1,l_2]} (W))\log 2={w_{0,{\rm reg}}^{(m)}}'(0)+{w_{1,{\rm reg}}^{(m)}}'(0)+{w_2^{(m)}}'(0)+{w_3^{(m)}}'(0),
\]
and hence 
\[
A_{\rm BM, mix}(\b  C_{[l_1,l_2]} (W))={w_{0,{\rm sing}}^{(m)}}'(0)+{w_{1,{\rm sing}}^{(m)}}'(0),
\]
and this concludes the proof of the formulas in equation (\ref{www}).

\subsection{The regular part of the torsion for the frustum with absolute BC}
\label{8.3}
By equations (\ref{ly2}) and (\ref{ly3}),
\begin{align*}
{y_{0,{\rm reg}}^{(2p)}}'(0)=&0,\\
{y_{0,{\rm reg}}^{(2p-1)}}'(0)=&- \sum^{p-2}_{q=0} (-1)^q\zeta(0,\tilde\Delta^{(q)}) \log l_1 l_2
-\sum^{p-2}_{q=0} (-1)^q \zeta_{\rm cex}'(0,\tilde \Delta^{(q)}),\\
{y_{1,{\rm reg}}^{(2p-1)}}'(0)=&-(-1)^{p-1} \frac{1}{2}\zeta_{\rm cex}(0,\tilde\Delta^{(p-1)})\log l_1 l_2
-\frac{1}{2}(-1)^{p-1}\zeta_{\rm cex}'(0,\tilde\Delta^{(p-1)}),
\end{align*}
and, using equation (\ref{forfor}),
\[
{y_{0,{\rm reg}}^{(2p-1)}}'(0)+{y_{1,{\rm reg}}^{(2p-1)}}'(0)= \sum^{p-1}_{q=0} (-1)^q m_{{\rm har}, q} (2p-1-2q)\log l_1 l_2+\log T(W,g).
\]

By equation (\ref{h1})
\begin{align*}
{y_{0,{\rm reg}}^{(2p-1)}}'(0)+{y_{1,{\rm reg}}^{(2p-1)}}'(0)+{y_{2}^{(2p-1)}}'(0)=&\log T(W,g)+\frac{1}{2} \sum_{q=0,q\not= p}^{2p-1} (-1)^{q+1} 
m_{{\rm har},q} \log\frac{l_2^{2p-2q}-l_1^{2p-2q}}{2p-2q}\\
&+ \frac{1}{2}(-1)^p m_{{\rm har},p}\log\log \frac{l_2}{l_1},
\end{align*}
and by equation (\ref{h2})
\begin{align*}
{y_{0,{\rm reg}}^{(2p)}}'(0)+{y_{2}^{(2p)}}'(0)+{y_{3}^{(2p)}}'(0)=&\frac{1}{2} \chi(W)\log 2+ \frac{1}{2} \sum_{q=0}^{2p} (-1)^{q}  m_{{\rm har},q}\log\frac{l_2^{2p-2q+1}-l_1^{2p-2q+1}}{2p-2q+1}.
\end{align*}

Comparison with Proposition 3.3 of \cite{HS3}, proves the following formulas:
\[
{y_{0,{\rm reg}}^{(m)}}'(0)+{y_{1,{\rm reg}}^{(m)}}'(0)+{y_{2}^{(m)}}'(0)+{y_{3}^{(m)}}'(0)=\log\tau_{\rm R}(C_{[l_1,l_2]}(W^{(m)}))+\frac{1}{2}\chi(W^{(m)})\log 2,
\]
and 
\[
{y_{0,{\rm sing}}^{(m)}}'(0)+{y_{1,{\rm sing}}^{(m)}}'(0)=A_{\rm BM,abs}(\b C_{[l_1,l_2]}(W),
\]
and this concludes the proof of the formulas in equation (\ref{www1}).

\section{The anomaly boundary term}
\label{nove}

We show in this section that the anomaly boundary term of the frustum with mixed BC is twice the one of the cone. This concludes the proof of Theorem \ref{t0}. 

By equations (\ref{wsin0}) and (\ref{wsin1}), 
\[
\begin{aligned}
{w_{0,{\rm sing}}^{(m)}}'(0)=& \frac{1}{4}\sum^{\left[\frac{m}{2}\right]-1}_{q=0} (-1)^q \sum_{j=1}^m \left((\hat\Phi_{q,j,-}(0;l_2,l_1)-\hat \Phi_{q,j,+}(0;l_1,l_2))\right.\\
&\left.+(-1)^{m+1}(\hat \Phi_{q,j,+}(0;l_2,l_1)-\hat\Phi_{q,j,-}(0;l_1,l_2))\right)\Ru_{s=j}\zeta(s,U),\\
{w_{1,{\rm sing}}^{(2p-1)}}'(0)=& (-1)^{p-1}\frac{1}{4}  \sum_{j=1}^{2p-1} \left( \hat \Phi_{p-1,j,0}(0;l_2,l_1)-\hat \Phi_{p-1,j,0}(0;l_1,l_2) \right)\Ru_{s=j}\zeta(s,U),
\end{aligned}
\]
where (see equation (\ref{phiphi}))
\[
\hat \Phi_{q,j,\pm}(s;l_1,l_2)=\int_0^\infty t^{s-1}\frac{1}{2\pi i}\int_{\Lambda_{\theta,c}}\frac{\e^{-\lambda t}}{-\lambda} \hat \phi_{q,j,\pm}(\lambda;l_1,l_2) d\lambda dt,
\]
and  the $\hat \phi$ are given in equation (\ref{ppmix}). Let introduce the linear operator
\[
\T(f)=\int_0^\infty t^{s-1}\frac{1}{2\pi i}\int_{\Lambda_{\theta,c}}\frac{\e^{-\lambda t}}{-\lambda} f (\lambda)d\lambda dt.
\]

Note that
\[
\T(f(a\_))(\_)=a^s \T(f(\_))(\_),
\]
and hence by Lemma \ref{tec1}
\[
\hat\Phi_{q,j,\pm}(s;l_1,l_2)=\sgn(l_1-l_2)^j l_1^{2s}\Phi_{q,j}(s)+ \sgn(l_2-l_1)^j l_2^{2s}\hat\Phi_{q,j,\pm \sgn(l_2-l_1)}(s).
\]

Therefore, 
\begin{align*}
(\hat\Phi_{q,j,-}(0;l_2,l_1)-\hat \Phi_{q,j,+}(0;l_1,l_2))&+(-1)^{m+1}(\hat \Phi_{q,j,+}(0;l_2,l_1)-\hat\Phi_{q,j,-}(0;l_1,l_2))\\
=&(1-(-1)^j)\Phi_{q,j}(0)-\hat\Phi_{q,j,+}(0)+(-1)^j\hat\Phi_{q,j,-}(0)\\
&+(-1)^{m+1}\left((1-(-1)^j)\Phi_{q,j}(0)-\hat\Phi_{q,j,-}(0)+(-1)^j\hat\Phi_{q,j,+}(0)\right);
\end{align*}
and 
\begin{align*}
\hat\Phi_{p-1,j,0}(0;l_2,l_1)-\hat \Phi_{p-1,j,0}(0;l_1,l_2))
=&(1-(-1)^j)\Phi_{p-1,j}(0)-(1-(-1)^j) \hat \Phi_{p-1,j,0}(0).
\end{align*}


It is  now convenient to distinguish odd and even cases. If $m=2p-1$ is odd, then the relevant values of the index $j$ are the odd ones, since the zeta function $\zeta(s,U)$ only has poles at odd integers $s$, while if $m=2p$ is even, the relevant values for $j$ are the even ones, for similar reason. Thus, if $m=2p-1$, 
\begin{align*}
(\hat\Phi_{q,2k-1,-}(0;l_2,l_1)&-\hat \Phi_{q,2k-1,+}(0;l_1,l_2))+(\hat \Phi_{q,2k-1,+}(0;l_2,l_1)-\hat\Phi_{q,2k-1,-}(0;l_1,l_2))\\
=&4\Phi_{q,2k-1}(0)-2\hat\Phi_{q,2k-1,+}(0)-2\hat\Phi_{q,2k-1,-}(0),
\end{align*}
and
\begin{align*}
\hat\Phi_{p-1,2k-1,0}(0;l_2,l_1)-\hat \Phi_{p-1,2k-1,0}(0;l_1,l_2))
=&2\Phi_{p-1,j}(0)-2 \hat \Phi_{p-1,j,0}(0),
\end{align*}
if $m=2p$ is even, 
\begin{align*}
(\hat\Phi_{q,2k,-}(0;l_2,l_1)&-\hat \Phi_{q,2k,+}(0;l_1,l_2))+(\hat \Phi_{q,2k,+}(0;l_2,l_1)-\hat\Phi_{q,2k,-}(0;l_1,l_2))\\
=&2\hat\Phi_{q,2k-1,-}(0)-2\hat\Phi_{q,2k-1,+}(0).
\end{align*}

Now, recalling the singular contribution for the cone given in Section \ref{singconesing}
\[
\begin{aligned}
{t_{0,{\rm sing}}^{(2p-1)}}'(0)&= \frac{1}{4}\sum^{p-2}_{q=0} (-1)^q \sum_{j=1}^{p-2} (2\Phi_{q,2j-1}(0)-\hat \Phi_{q,2j-1,+}(0)-\hat\Phi_{q,2j-1,-}(0))\Ru_{s=2j-1}\zeta(s,U),\\
{t_{0,{\rm sing}}^{(2p)}}'(0)&= \frac{1}{4}\sum^{p-1}_{q=0} (-1)^q \sum_{j=1}^{p-1} (\hat \Phi_{q,2j,-}(0))
-\hat\Phi_{q,2j,+}(0))\Ru_{s=2j}\zeta(s,U),
\end{aligned}
\]
and
\[
{t_{1,{\rm sing}}^{(2p-1)}}'(0)= (-1)^{p-1}\frac{1}{4}  \sum_{j=1}^m \left(  \Phi_{p-1,j}(0)-\hat \Phi_{p-1,j,0}(0) \right)\Ru_{s=j}\zeta(s,U),
\]
and hence we have proved that
\begin{align*}
{w_{0,{\rm sing}}^{(m)}}'(0)=2{t_{0,{\rm sing}}^{(m)}}'(0),\\
{w_{1,{\rm sing}}^{(2p-1)}}'(0)=2{t_{1,{\rm sing}}^{(2p-1)}}'(0).
\end{align*}

Now, by the final formula in the previous section, and by Lemma 4.1 of \cite{HS3}
\[
{w_{0,{\rm sing}}^{(m)}}'(0)+{w_{1,{\rm sing}}^{(m)}}'(0)=A_{\rm BM, mix}(\b  C_{[l_1,l_2]} (W))=\int_W B,
\]
where $B$ is defined in Section 4.3 of \cite{HS3} (based on \cite{BM1}). Therefore,
\[
{t_{0,{\rm sing}}^{(m)}}'(0)+{t_{1,{\rm sing}}^{(m)}}'(0)=\frac{1}{2}\int_W B,
\]
and this is exactly the term that we would obtain applying the formula of \cite{BM1} \cite{BM2} on the cone, as if it were a smooth manifold. This concludes the proof of Theorem \ref{t0}.

\section{The limiting case}

We address the following question: is there any relationship between the analytic torsion of the frustum and that of the cone? In \cite{HS3} we proved that, in the odd case $m=2p-1$,  regularising the analytic torsion of the frustum with absolute BC (taking the quotient by the suitable factor) and taking the limit for $l_1\to 0^+$ we obtained the torsion of the cone. The results of the present paper permit a more  explicit unified analysis for all dimensions, and a possible interpretation for the regularising factor.

The idea is to consider the set of the formal eigenfunctions of the cone that are not square integrable. If we proceed formally, applying the boundary conditions, these eigenfunctions gives   a new set of eigenvalues for the formal operator, that for simplicity we call the negative part of the spectrum. We verify that the negative part of the spectrum can be treated by the same method used for the positive part of the spectrum, up to some technical points, that we describe in details. As a result, we obtain a new term, that we call the negative part of the analytic torsion, and that we denote by $\log  T_-(C_{l}(W))$. We give a formula for this term in Proposition \ref{neg1}, that shows clearly the analogies and the differences with the regular torsion. Next, in Proposition \ref{neg2} we give the expansion for the logarithm of the ratio of the analytic torsion of the cone to the negative negative torsion of the cone (with $l=l_1$), and eventually we show in Theorem \ref{t2} that the finite part of this ratio coincides with the analytic torsion of the cone (with $l=l_2$) up to a classical boundary term (see Remark \ref{euler}).  

\subsection{The ``negative" analytic torsion of the cone}

We proceed with the notation of Section \ref{s1}. The set $\SS^{(q)}$ of the eigenvalues of the equation $\Delta^{(q)} u=\lambda^2 u$, with absolute BC and $\lambda\not =0$, is $\SS^{(q)}=\SS^{(q)}_+\cup \SS^{(q)}_-$, where $\SS^{(q)}_+=\Sp_+\Delta^{(q)}_{\rm abs}$, and $\SS^{(q)}_-$ is the same set with $\mu_{x,n}$ replaced by $-\mu_{x,n}$, $x=q,q-1,q-2$, namely:
if $m=\dim W=2p-1$, $p\geq 1$:
\begin{align*}
\SS_-^{(q)} &= \left\{m_{{\rm cex},q,n} : \hat j^{2}_{-\mu_{q,n},\alpha_q,k}/l^{2}\right\}_{n,k=1}^{\infty}
\cup
\left\{m_{{\rm cex},q-1,n} : \hat j^{2}_{-\mu_{q-1,n},\alpha_{q-1},k}/l^{2}\right\}_{n,k=1}^{\infty} \\
&\cup \left\{m_{{\rm cex},q-1,n} : j^{2}_{-\mu_{q-1,n},k}/l^{2}\right\}_{n,k=1}^{\infty} \cup \left\{m _{{\rm cex},q-2,n} :
j^{2}_{-\mu_{q-2,n},k}/l^{2}\right\}_{n,k=1}^{\infty} \\
&\cup \left\{m_{{\rm har},q}:\hat j^2_{-|\alpha_q|,\alpha_q,k}/l^{2}\right\}_{k=1}^{\infty} \cup \left\{ m_{{\rm har},q-1}:\hat
j^2_{-|\alpha_{q-1}|,\alpha_q,k}/l^{2}\right\}_{k=1}^{\infty};
\end{align*}
if $m=\dim W=2p$, $p\geq 1$:
\begin{align*}
\SS_-^{(q\not= p, p+1)} &= \left\{m_{{\rm cex},q,n} : \hat j^{2}_{-\mu_{q,n},\alpha_q,k}/l^{2}\right\}_{n,k=1}^{\infty}
\cup
\left\{m_{{\rm cex},q-1,n} : \hat j^{2}_{-\mu_{q-1,n},\alpha_{q-1},k}/l^{2}\right\}_{n,k=1}^{\infty} \\
&\cup \left\{m_{{\rm cex},q-1,n} : j^{2}_{-\mu_{q-1,n},k}/l^{2}\right\}_{n,k=1}^{\infty} \cup \left\{m _{{\rm cex},q-2,n} :
j^{2}_{-\mu_{q-2,n},k}/l^{2}\right\}_{n,k=1}^{\infty} \\
&\cup \left\{m_{{\rm har},q}:\hat j^2_{-|\alpha_q|,\alpha_q,k}/l^{2}\right\}_{k=1}^{\infty} \cup \left\{ m_{{\rm har},q-1}:\hat j^2_{-|\alpha_{q-1}|,\alpha_{q-1},k}/l^{2}\right\}_{k=1}^{\infty},
\end{align*}
\begin{align*}
\SS_-^{(p)} &= \left\{m_{{\rm cex},p,n} : \hat j^{2}_{-\mu_{p,n},\alpha_p,k}/l^{2}\right\}_{n,k=1}^{\infty}
\cup
\left\{m_{{\rm cex},p-1,n} : \hat j^{2}_{-\mu_{p-1,n},\alpha_{p-1},k}/l^{2}\right\}_{n,k=1}^{\infty} \\
&\cup \left\{m_{{\rm cex},p-1,n} : j^{2}_{-\mu_{p-1,n},k}/l^{2}\right\}_{n,k=1}^{\infty} \cup \left\{m _{{\rm cex},p-2,n} :
j^{2}_{-\mu_{p-2,n},k}/l^{2}\right\}_{n,k=1}^{\infty} \\
&\cup \left\{\frac{1}{2}m_{{\rm har},p}: j^2_{\frac{1}{2}}/l^{2}\right\}_{k=1}^{\infty} \cup \left\{\frac{1}{2}m_{{\rm har},p}: j^2_{-\frac{1}{2}}/l^{2}\right\}_{k=1}^{\infty}\cup \left\{ m_{{\rm har},p-1}:\hat j^2_{-|\alpha_{p-1}|,\alpha_{p-1},k}/l^{2}\right\}_{k=1}^{\infty},
\end{align*}
\begin{align*}
\SS_-^{(p+1)} &= \left\{m_{{\rm cex},p+1,n} : \hat j^{2}_{-\mu_{p+1,n},\alpha_{p+1},k}/l^{2}\right\}_{n,k=1}^{\infty}
\cup
\left\{m_{{\rm cex},p,n} : \hat j^{2}_{-\mu_{p,n},\alpha_{p},k}/l^{2}\right\}_{n,k=1}^{\infty} \\
&\cup \left\{m_{{\rm cex},p,n} : j^{2}_{-\mu_{p,n},k}/l^{2}\right\}_{n,k=1}^{\infty} \cup \left\{m _{{\rm cex},p-1,n} :
j^{2}_{-\mu_{p-1,n},k}/l^{2}\right\}_{n,k=1}^{\infty} \\
&\cup \left\{ m_{{\rm har},p+1}:\hat j^2_{-|\alpha_{p+1}|,\alpha_{p+1},k}/l^{2}\right\}_{k=1}^{\infty} \cup \left\{\frac{1}{2}m_{{\rm har},p}: 
j^2_{-\frac{1}{2}}/l^{2}\right\}_{k=1}^{\infty} \cup \left\{\frac{1}{2}m_{{\rm har},p}: j^2_{\frac{1}{2}}/l^{2}\right\}_{k=1}^{\infty},
\end{align*}
where  the $j_{\mu,k}$ are the positive zeros of the Bessel function $J_{\mu}(x)$,  the $\hat j_{\mu,c,k}$ are the positive zeros of
the function $\hat J_{\mu,c}(x) = c J_\mu (x) + x J'_\mu(x)$,  $c\in \R$, $\alpha_q$ and $\mu_{q,n}$ are defined in Lemma \ref{l2}.

\begin{rem}\label{rip1} The above description of the negative spectrum is always valid except that for the eigenvalues associated to the harmonics of the section  in the odd case $m=2p-1$. In such a case, the eigenfunctions associated to these eigenvalues are not the Bessel function themselves, as observed at the end of Lemma \ref{l2}. We will take care of this difference explicitly when we treat the term of the analytic torsion associated to these eigenvalues, in Subsection \ref{harmanom} below.
\end{rem} 

Note that the set $\SS^{(q)}$ satisfy all the same properties satisfied by the set $\SS^{(q)}_+$ and used in the previous sections in order to define and analyse the associated spectral functions. Following this idea, and proceeding as in Section \ref{s4}, we consider  the functions
\begin{align*}
Z_{q}^-(s)&=\sum^{\infty}_{n,k=1} m_{{\rm cex},q,n}j^{-2s}_{-\mu_{q,n},k},&
\hat Z_{q,\pm}^-(s)&=\sum^{\infty}_{n,k=1} m_{{\rm cex},q,n}\hat j^{-2s}_{-\mu_{q,n},\pm\alpha_q,k},
&z_{q,\pm}(s)&=\sum_{k=1}^{\infty} j^{-2s}_{\pm\alpha_{q},k},
\end{align*}
and 
\[
t^{(m)}_{\rm Cone,-}(s)=t_{0,-}^{(m)}(s)+t_{1,-}^{(m)}(s)+t^{(m)}_{2,-}(s)+t^{(m)}_{3,-}(s),
\]
where 
\begin{align*}
t_{0,-}^{(m)}(s)&= \frac{l^{2s}}{2}\sum^{\left[\frac{m}{2}\right]-1}_{q=0} (-1)^q \left((Z_q^-(s)-\hat Z_{q,+}^-(s))
+(-1)^{m-1}(Z_q^-(s)-\hat Z_{q,-}^-(s))\right),\\
t_{1,-}^{(2p-1)}(s)&= (-1)^{p-1}\frac{l^{2s}}{2}  \left(  Z^-_{p-1}(s)-\hat Z^-_{p-1,0}(s) \right),&t_{1,-}^{(2p)}(s)&=0\\
t_{2,-}^{(m)}(s)&= \frac{l^{2s}}{2} \sum_{q=0}^{\left[\frac{m-1}{2}\right]} (-1)^{q+1}  m_{{\rm har},q}\left(z_{q-1,+}(s)+(-1)^m z_{q,+}(s)\right),\\
t_{3,-}^{(2p)}(s)&= (-1)^{p+1} m_{{\rm har},p}\frac{l^{2s}}{4}\left(z_{p,+}(s)+z_{p,-}(s)\right),&t_{3,-}^{(2p-1)}(s)&=0,
\end{align*}

The aim of this section is to determine the quantity
\[
\log T_{-}(C_l(W^{(m)})=t^{(m)}_{\rm Cone,-}(0).
\]

\subsubsection{The contribution of the regular part}  The functions $t_{0,-}$ and $t_{1,-}$ are double series as  $t_{0}$ and $t_{1}$, and, up to solving some  technical problems,  can be analysed by the same method used in Section \ref{calc1}. The relevant sequences are now 
$S_{q}^-=\{m_{{\rm cex},q,n} : j_{-\mu_{q,n}, k}^2\}$, and 
$\hat S_{q,\pm}^-=\{m_{{\rm cex},q,n} : j_{-\mu_{q,n}, \pm \alpha_q, k}^2\}$.  We obtain the following representation  for associated logarithmic Gamma functions:
\begin{align*}
\log \Gamma(-\lambda,S_{q}^-)=&-\log\prod_{k=1}^\infty \left(1+\frac{(-\lambda)}{ j_{-\mu_{q,n},k}^2}\right)\\
=&-\log  I_{-\mu_{q,n}}(\sqrt{-\lambda})-\mu_{q,n}\log\sqrt{-\lambda}+\mu_{q,n}\log 2-\log\Gamma(1-\mu_{q,n}).
\end{align*}
\begin{align*}
\log \Gamma(-\lambda,\hat S_{q,\pm}^-)=&-\log\prod_{k=1}^\infty \left(1+\frac{(-\lambda)}{\hat j_{-\mu_{q,n},\pm\alpha_q,k}^2}\right)\\
=&-\log \hat I_{-\mu_{q,n},\pm\alpha_q}(\sqrt{-\lambda})-\mu_{q,n}\log\sqrt{-\lambda}+\mu_{q,n}\log 2\\
&-\log\Gamma(-\mu_{q,n})+\log\left(1\mp\frac{\alpha_q}{\mu_{q,n}}\right).
\end{align*}

Recalling
\[
I_{-\nu} (z)=\frac{2}{\pi} \sin\nu\pi K_\nu (z) + I_\nu (z),
\] 
substitution in the representations above of the logarithmic Gamma functions, and using the known asymptotic expansions for the Bessel functions, shows that the asymptotic expansions for the negative case may be deduced from the ones computed for the positive case. We give here the relevant results, using the same notation as in Section \ref{calc2}, with an  added minus index.

\begin{align*}
{t_{0,{\rm reg},-}^{(m)}}'(0)=&-\sum^{\left[\frac{m}{2}\right]-1}_{q=0} (-1)^q \left((\B_{q,-}(0)-\hat \B_{q,+,-}(0))
+(-1)^{m-1}(\B_{q,-}(0)-\hat \B_{q,-,-}(0))\right)\log l\\
&-\frac{1}{2}\sum^{\left[\frac{m}{2}\right]-1}_{q=0} (-1)^q \left(\A_{q,-}(0)+\B_{q,-}'(0)-\hat \A_{q,+,-}(0)-\hat \B_{q,+,-}'(0)\right)\\
&-(-1)^{m-1}\frac{1}{2}\sum^{\left[\frac{m}{2}\right]-1}_{q=0} (-1)^q\left(\A_{q,-}(0)+\B_{q,-}'(0)-\hat \A_{q,-,-}(0)-\hat \B_{q,-,-}'(0))\right),
\end{align*}
where
\begin{align*}
\A_{q,-}(0)&=0,\\
\B_{q,-}(0)&=\frac{1}{4}\sum_{n=1}^\infty m_{{\rm cex},q,n}=\frac{1}{4}\zeta_{\rm cex}(0,\tilde\Delta^{(q)}+\alpha_q^2)=\frac{1}{4}\zeta(0,\tilde\Delta^{(q)}), \\
\B_{q,-}'(0)&=-\frac{1}{2}\sum_{n=1}^\infty m_{{\rm cex},q,n}\log \mu_{q,n}, \\
\hat \A_{q,\pm,-}(0)&=\sum_{n=1}^\infty m_{{\rm cex},q,n}\log\left(1\pm\frac{\alpha_q}{\mu_{q,n}}\right),\\
\hat \B_{q,\pm,-}(0)&=-\frac{1}{4}\sum_{n=1}^\infty m_{{\rm cex},q,n}=-\frac{1}{4}\zeta_{\rm cex}(0,\tilde\Delta^{(q)}+\alpha_q^2)=-\frac{1}{4}\zeta_{\rm cex}(0,\tilde\Delta^{(q)}),\\
\hat \B_{q,\pm,-}'(0)&=\frac{1}{2}\sum_{n=1}^\infty m_{{\rm cex},q,n}\log\mu_{q,n},
\end{align*}
and 
\begin{align*}
\sum_{n=1}^\infty m_{{\rm cex},q,n}&=\zeta_{\rm cex}(0,\tilde\Delta^{(q)}+\alpha_q^2),&
-2\sum_{n=1}^\infty m_{{\rm cex},q,n}\log \mu_{q,n}&=\zeta_{\rm cex}'(0,\tilde\Delta^{(q)}+\alpha_q^2).
\end{align*}

Thus,
\begin{align}
\nonumber{t_{0,{\rm reg},-}^{(m)}}'(0)=&-\frac{1}{2}\sum^{\left[\frac{m}{2}\right]-1}_{q=0} (-1)^q \left(\zeta_{\rm cex}(0,\tilde\Delta^{(q)})
+(-1)^{m-1}\zeta_{\rm cex}(0,\tilde\Delta^{(q)})\right)\log l\\
\label{t2.1-}&+\frac{1}{2}\sum^{\left[\frac{m}{2}\right]-1}_{q=0} (-1)^q \left(\sum_{n=1}^\infty m_{{\rm cex},q,n}\log\left(1-\frac{\alpha_q}{\mu_{q,n}}\right)+\sum_{n=1}^\infty m_{{\rm cex},q,n}\log \mu_{q,n}\right)\\
\nonumber&+(-1)^{m-1}\frac{1}{2}\sum^{\left[\frac{m}{2}\right]-1}_{q=0} (-1)^q\left(\sum_{n=1}^\infty m_{{\rm cex},q,n}\log\left(1+\frac{\alpha_q}{\mu_{q,n}}\right)+\sum_{n=1}^\infty m_{{\rm cex},q,n}\log \mu_{q,n}\right)\\
\nonumber=&(-1)^{m+1}{t_{0,{\rm reg},-}^{(m)}}'(0)
\end{align}

Similar analysis gives
\beq\label{t2.2-}\begin{aligned}
{t_{1,{\rm reg},-}^{(2p-1)}}'(0)=& -(-1)^{p-1} \frac{1}{2}\zeta_{\rm cex}(0,\tilde\Delta^{(p-1)})\log l
-\frac{1}{4}(-1)^{p-1}\zeta_{\rm cex}'(0,\tilde\Delta^{(p-1)})
=&{t_{1,{\rm reg}}^{(2p-1)}}'(0).
\end{aligned}
\eeq

\subsubsection{The contribution of the singular part} It is easy to realise that the singular part coincides exactly with the singular part in the positive case, namely:
\beq\label{tsing-}\begin{aligned}
{t_{0,{\rm sing},-}^{(m)}}'(0)=&{t_{0,{\rm sing},-}^{(m)}}'(0),\\
{t_{1,{\rm sing},-}^{(2p-1)}}'(0)=&{t_{1,{\rm sing},-}^{(2p-1)}}'(0) .
\end{aligned}
\eeq

\subsubsection{The contribution of the harmonics}\label{harmanom} This is the contribution coming from the simple series $z_{q,+}(s)$. Respect with the harmonics for the positive torsion, where we studied the functions $z_{q,-}(s)$,  there is now a technical problem, since the values of $\alpha_q$ appearing in the $z_{q,-}(s)$ are never negative integers, while the values of the $\alpha_q$ appearing in the $z_{q,+}(s)$ maybe negative integers. This problem appears only when $m=2p-1$ is odd, that is treated below in details. In  the even case  $m=2p$, we can use the formulas in equations (\ref{p000}) and (\ref{p00}) as in Section \ref{har11}. For $0\leq q<p$, $\alpha_q = \frac{1}{2}+q-p$, and we obtain:
\begin{align*}
z_{q,+}(0)&=-\frac{1}{2}\left(\alpha_q+\frac{1}{2}\right)=-\frac{1}{2}\left(1+q-p\right)=-z_{q,-}(0)-\frac{1}{2},\\
z_{q,+}'(0)&=\log \frac{\sqrt{\pi}}{2^{\alpha_q-\frac{1}{2}}\Gamma(\alpha_q+1)}=-\log2-\log(2(p-q-1)-1)!!\\
&=-z'_{q,-}(0)+\log (2p-2q-1)-2\log 2.
\end{align*}
 
This gives:
\beq\label{t3-}
\begin{aligned}
{t_{2,-}^{(2p)}}'(0)=&-{t_{2}^{(2p)}}'(0)-\sum_{q=0}^{p-1}(-1)^{q+1}r_q\log l\\
&+\frac{1}{2}\sum_{q=0}^{p-1}(-1)^{q+1}r_q\log (2p-2q-1)(2p-2q+1)-2\sum_{q=0}^{p-1}(-1)^{q+1}r_q\log 2,\\
{t_{3,-}^{(2p)}}'(0)=&{t_{3}^{(2p)}}'(0).
\end{aligned}
\eeq
 
In the odd case, $m=2p-1$, since $\alpha_q =  1+q-p<0$ for $0\leq q\leq p-1$, as observed in Remark \ref{rip1}, the eigenfunctions of type $E$ and $O$ are not the Bessel function $J_{\alpha_q}$ of negative index, but the functions $Y_{\alpha_q}$. Whence,  function under study is the function is
\[
z_{q,+}(s)=\sum_{k=1}^\infty y_{\alpha_q,k}^{-2s},
\]
where $S=\{y_{\alpha_q,k}\}$ is the sequence of the zeros of the Bessel function $Y_{-\alpha_q}$. We proceed as follows. Consider the series representation of the Bessel function $Y_n$ ($n=-\alpha_q$),
\begin{align*}
Y_n(z) = \frac{2}{\pi} &J_n(z) \left(\log\frac{z}{2} + C\right) +\sum_{k=0}^{n-1} \frac{(n-k-1)!}{\pi k!} \left(\frac{z}{2}\right)^{2k-n}-\left(\frac{z}{2}\right)^n \frac{1}{n!}\sum_{k=1}^n \frac{1}{\pi k} - z^n\sum_{k=1}^\infty a_{n,k}.
\end{align*}  

Since
\begin{align*}
\lim_{z\to 0} z^{n}Y_n(z) =  \frac{1}{\pi} \frac{(n-1)! }{2^{-n}} = \frac{2^n (n-1)!}{\pi},
\end{align*} 
we have the  product expansion 
\[
G_n(z)= z^{n}Y_n(z) = \frac{2^n (n-1)!}{\pi}\prod_{k=-\infty, k\not=0}^{+\infty} \left(1-\frac{z}{y_{n,k}}\right).
\] 

Define
\begin{align*}
 \Upsilon_n(z) 
 = (niz)^{n} \left(J_{n}(iz) - H^{(1)}_{n}(iz)\right)
 =z^{n} \left( I_n(z) - \frac{2}{\pi i} K_n(z)\right),
\end{align*} 
then
\[
\Upsilon_n(z) =\frac{2^n (n-1)!}{\pi}\prod_{k=1}^{+\infty} \left( 1 + \frac{z^2}{y^2_{n,k}}\right),
\]
and
\begin{align*}
\log\left(-\lambda, S\right) &= - \log \prod_{k=1}^{\infty}  \left( 1 - \frac{\lambda}{y^2_{n,k}}\right)=-\log \Upsilon_n(\sqrt{-\lambda}) + n\log 2 + \log(n-1)!  - \log \pi.
\end{align*}

Using the classical expansion  for large $z$ of $I_n(z) $ and $K_n(z)$,
\[
\log\left( I_n(z) - \frac{2}{\pi i} K_n(z)\right) \sim \log I_{n}(z) + O(e^{-z}).
\] 

This implies that
\begin{align*}
\log \Upsilon_n (\sqrt{-\lambda}) &=n\log \sqrt{-\lambda} + \sqrt{-\lambda} -\frac{1}{2}\log 2\pi - \frac{1}{2}\log\sqrt{-\lambda}\\
&=\frac{2n-1}{2} \log \sqrt{-\lambda} + \sqrt{-\lambda} -\frac{1}{2}\log 2\pi,
\end{align*} 

and
\begin{align*}
\log\left(-\lambda, S\right) &=-\frac{2n-1}{2} \log \sqrt{-\lambda} - \sqrt{-\lambda} + \frac{2n+1}{2}\log 2 + \log(n-1)!  - \frac{1}{2}\log \pi.
\end{align*}

Whence, for $q\not=p-1$, 
\begin{align*}
z_{q,+}(0)&=\frac{1}{2}\left( p-q-1-\frac{1}{2}\right) = \frac{1}{2}\left(-\alpha_q +\frac{1}{2}\right) -\frac{1}{2}= - z_{q,-}(0)-\frac{1}{2},\\
z'_{q,+}(0)&=- \log \frac{2^{\left(-\frac{1}{2}-q+p\right)}(p-q-2)!}{\pi^{\frac{1}{2}} }= -\log\frac{2^{-\alpha_q +\frac{1}{2}} \Gamma(-\alpha_q)}{\pi^{\frac{1}{2}}} = -z'_{q,-}(0) - \log 2 + \log(- \alpha_q)\\
z_{p-1,+}(0) &= z_{p-1,-}(0)\\
z'_{p-1,+}(0) &= - z_{p-1,-}(0) - \log 2,
\end{align*}
and
\begin{align}\label{t2-}
{t_{2,-}^{(2p-1)}}'(0)=&-{t_2^{(2p-1)}}'(0)+\sum_{q=0}^{p-2}(-1)^{q+1}r_q\log\frac{p-q}{p-q-1}.
\end{align}

\subsection{Formula for the ``negative'' torsion and the limiting case $l_1=0$}

Collecting the results of the previous subsections we have:

\begin{align*}
\log T_-(C_l(W^{(2p-1)}))=&{t_{0,{\rm reg},-}^{(2p-1)}}'(0)+{t_{2,-}^{(2p-1)}}'(0)+{t_{3,-}^{(2p-1)}}'(0)+{t_{0,{\rm sing},-}^{(2p-1)}}'(0)\\
=&-\frac{1}{2} \sum_{q=0}^{p-1}(-1)^q r_q \log \frac{l^{2p-2q}}{2p-2q} +\frac{1}{2} \log T(W,g) + {t_{0,{\rm sing}}^{(2p-1)}}'(0)\\
&+\sum_{q=0}^{p-2}(-1)^{q+1}r_q\log\frac{p-q}{p-q-1}.
\end{align*}

\begin{align*}
\log  T_-(C_l(W^{(2p)}))=&{t_{0,{\rm reg},-}^{(2p)}}'(0)+{t_{2,-}^{(2p)}}'(0)+{t_{3,-}^{(2p)}}'(0)+{t_{0,{\rm sing},-}^{(2p)}}'(0)\\
=&-{t_{0,{\rm reg}}^{(2p)}}'(0)-{t_{2}^{(2p)}}'(0)+{t_{3}^{(2p)}}'(0)+{t_{0,{\rm sing}}^{(2p)}}'(0)\\
&+\sum_{q=0}^{p-1}(-1)^{q}r_q\log l-\frac{1}{2}\sum_{q=0}^{p-1}(-1)^{q}r_q\log (2p-2q-1)(2p-2q+1)\\
&+2\sum_{q=0}^{p-1}(-1)^{q}r_q\log 2\\
%
=&- \frac{1}{2}\sum_{q=0}^{p-1} (-1)^{q}  m_{{\rm har},q}\log \frac{l^{2p-2q+1}}{2p-2q+1}\\
&-\frac{1}{2}\sum_{q=0}^{p-1}(-1)^{q}r_q\log (2p-2q-1)(2p-2q+1)+(-1)^{p} \frac{1}{4}m_{{\rm har},p}\log l\\
&+\sum_{q=0}^{p-1}(-1)^{q}r_q\log l+\frac{1}{2} \sum_{q=0}^{p-1} (-1)^{q}  m_{{\rm har},q}\log((2p-2q-1)!!)^2
\end{align*}
\begin{align*}
&-\frac{1}{2}\sum^{p-1}_{q=0} (-1)^q \sum_{n=1}^\infty m_{{\rm cex},q,n}\left(\log\left(1+\frac{\alpha_q}{\mu_{q,n}}\right)-\log\left(1-\frac{\alpha_q}{\mu_{q,n}}\right)\right)\\
&+\sum_{q=0}^{p-1}(-1)^{q}r_q\log 2+(-1)^{p} \frac{1}{2} m_{{\rm har},p}\log 2+{t_{0,{\rm sing}}^{(2p)}}'(0).
\end{align*}

The previous formulas in the notation of Theorem \ref{t1.2} read:

\begin{prop}\label{neg1} 
\begin{align*}
\log T_-(C_l(W^{(2p-1)}))=&\frac{1}{2}\log T(W,g)-\log \frac{{\rm Det}^{p-1}_{0} \ddot\alphas_{C_l}}{{\rm Det}^{p-1}_{0,} \alphas}+A_{\rm BS,abs}(\b C_{l}(W))\\
&-\sum_{q=0}^{p-2}(-1)^{q}r_q\log\frac{p-q}{p-q-1};\\
\log T_-(C_l(W^{(2p)}))=&-\log \frac{{\rm Det}^{p-1}_0 \ddot\alphas_{C_l}}{{\rm Det}^{p-1}_0 \alphas}+\frac{1}{4}\chi(W)\log 2+A_{\rm BS,abs}(\b C_l(W))\\
&-B^{(2p)}_1(C_l(W))+B^{(2p)}_2(C_l(W))\\
&+\sum_{q=0}^{p-1}(-1)^{q}r_q\log l-\frac{1}{2}\sum_{q=0}^{p-1}(-1)^{q}r_q\log (2p-2q-1)(2p-2q+1).
\end{align*}
\end{prop}

We are now able to study the constant part of the limit for $l_1\to 0^+$ of the difference between the logarithms of the torsion of the frustum and the one of the negative torsion with $l=l_1$. For we need a few lemmas and a remark.

\begin{lem} For $l_1\to 0^+$, we have the expansions: 
\begin{align*}
\log \frac{{\rm Det} \ddot\alphas_F}{{\rm Det} \alphas}=&  \frac{(-1)^p}{2} r_p \log\log\frac{1}{l_1} +\frac{1}{2}\sum_{q=0}^{p-2} (-1)^q r_q \log l_1^{2p-2-2q}\\
&+ \log \frac{{\rm Det}^{p-1}_0 \ddot\alphas_{C_{l_2}}}{{\rm Det}^{p-1}_0 \alphas}
+\frac{1}{2}\sum_{q=0}^{p-2} (-1)^{q} r_{q} \log (2p-2-2q)+O(l_1),
\end{align*}
if $m=2p-1$, and
\begin{align*} 
\log \frac{{\rm Det} \ddot\alphas_F}{{\rm Det} \alphas}
=&-\frac{1}{2}\sum_{q=0}^{p-1} (-1)^q  r_q\log l_1^{2p-2q-1}\\
&+ \log \frac{{\rm Det}^{p-1}_0 \ddot\alphas_{C_{l_2}}}{{\rm Det}^{p-1}_0 \alphas}-\frac{1}{2}\sum_{q=0}^{p-1} (-1)^{q}  r_q\log(2p-2q-1)+\frac{1}{4}(-1)^p r_p \log l_2+O(l_1),
\end{align*}
if $m=2p$.
\end{lem}
\begin{proof} When $m=2p-1$,
\begin{align*}
\sum_{q=p+1}^{2p-1} (-1)^{q}  r_q\log\frac{l_2^{2p-2q}-l_1^{2p-2q}}{2p-2q}&=\sum_{q=p+1}^{2p-1} (-1)^q r_{q} \log \frac{l_2^{2q-2p}-l_1^{2q-2p}}{(2q-2p)(l_1l_2)^{2q-2p}}\\
&=\sum_{q=0}^{p-2} (-1)^{q+1} r_{q} \log  \frac{l_2^{2p-2-2q}-l_1^{2p-2-2q}}{(2p-2-2q)(l_1l_2)^{2p-2-2q}}\\
&=\sum_{q=0}^{p-2} (-1)^{q+1} r_{q} \log \frac{1-\left(\frac{l_1}{l_2}\right)^{2p-2-2q}}{2p-2-2q } +\sum_{q=0}^{p-2} (-1)^q r_q \log l_1^{2p-2-2q}.
\end{align*} 

Whence,
\begin{align*}
\log \prod_{q=0}^{2p-1} \Gamma_q
=& \sum_{q=0,q\not=p}^{2p-1} (-1)^{q}  m_{{\rm har},q}\log\frac{l_2^{2p-2q}-l_1^{2p-2q}}{2p-2q} + \frac{(-1)^p}{2} r_p \log\log\frac{l_2}{l_1}\\
=& \sum_{q=0}^{p-1} (-1)^{q}  m_{{\rm har},q}\log\frac{l_2^{2p-2q}-l_1^{2p-2q}}{2p-2q} + \frac{(-1)^p}{2} r_p \log\log\frac{l_2}{l_1}\\
+&\sum_{q=0}^{p-2} (-1)^{q+1} r_{q} \log \frac{1-\left(\frac{l_1}{l_2}\right)^{2p-2-2q}}{2p-2-2q } +\sum_{q=0}^{p-2} (-1)^q r_q \log l_1^{2p-2-2q}.
\end{align*}

When $m=2p$,
\begin{align*}
\log \prod_{q=0}^{2p} \Gamma_q
=& \sum_{q=0}^{2p} (-1)^{q}  m_{{\rm har},q}\log\frac{l_2^{2p-2q+1}-l_1^{2p-2q+1}}{2p-2q+1}\\
=& \sum_{q=0}^{p} (-1)^{q}  m_{{\rm har},q}\log\frac{l_2^{2p-2q+1}-l_1^{2p-2q+1}}{2p-2q+1}+ \sum_{q=p+1}^{2p} (-1)^q  m_{{\rm har},q}\log\frac{l_2^{2p-2q+1}-l_1^{2p-2q+1}}{2p-2q+1}\\
=& \sum_{q=0}^{p-1} (-1)^{q}  m_{{\rm har},q}\log\frac{l_2^{2p-2q+1}-l_1^{2p-2q+1}}{2p-2q+1}+ \sum_{q=p+1}^{2p} (-1)^q  m_{{\rm har},q}\log\frac{l_2^{2p-2q+1}-l_1^{2p-2q+1}}{2p-2q+1}\\
&+(-1)^p r_p \log (l_2-l_1).
\end{align*}

\begin{align*}
\sum_{q=p+1}^{2p} (-1)^q  m_{{\rm har},q}\log\frac{l_2^{2p-2q+1}-l_1^{2p-2q+1}}{2p-2q+1}=&\sum_{q=p+1}^{2p} (-1)^q  m_{{\rm har},q}\log\left(l_1^{2p-2q+1}-l_2^{2p-2q+1}\right)\\
&-\sum_{q=p+1}^{2p} (-1)^q  m_{{\rm har},q}\log(2q-1-2p).
\end{align*}

\begin{align*}
\sum_{q=p+1}^{2p} (-1)^q  m_{{\rm har},q}\log(2q-1-2p)=\sum_{q=0}^{p-1} (-1)^q  m_{{\rm har},q}\log(2p-1-2q).
\end{align*}

\begin{align*}
\sum_{q=p+1}^{2p} (-1)^q  m_{{\rm har},q}\log\left(l_1^{2p-2q+1}-l_2^{2p-2q+1}\right)
=&\sum_{q=p+1}^{2p} (-1)^q  m_{{\rm har},q}\log l_1^{2p-2q+1}\\
&+\sum_{q=p+1}^{2p} (-1)^q  m_{{\rm har},q}\log\left(1-\left(\frac{l_1}{l_2}\right)^{2q-2p-1}\right).
\end{align*}

\end{proof}

\begin{rem} Observe that the anomaly boundary term of the frustum is the sum of two equal terms each defined as an integral over on of the two boundaries, so we can write
\[
A_{\rm BS,abs}(\b C_{[l_1,l_2]}(W))=A_{\rm BS,abs}(\b C_{l_2}(W))+A_{\rm BS,abs}(\b C_{l_1}(W)),
\]
and this is true in all dimensions, see Lemma 4.1 of \cite{HS3}. 
\end{rem}

\begin{prop}\label{neg2} For small $l_1$, we have the expansions: 
\begin{align*}
\log T_{\rm abs}(C_{[l_1,l_2]}(W^{(2p-1)}))-\log  T_-(C_{l_1}(W^{(2p-1)}))=& \frac{(-1)^p}{2} r_p \log\log\frac{1}{l_1} +\sum_{q=0}^{p-1} (-1)^q r_q \log l_1^{2p-1-2q}\\
&+\log T_{\rm abs}(C_{l_2}(W))+O(l_1),\\
\log T_{\rm abs}(C_{[l_1,l_2]}(W^{(2p)}))-\log  T_-(C_{l_1}(W^{(2p)}))=&\log T_{\rm abs, ideal}(C_{l_2}(W))-\frac{1}{2}\chi(W) \log 2+O(l_1).
\end{align*}
\end{prop}
\begin{proof} Using the expansion in the previous lemma and the formulas in Theorem \ref{t1} we compute, in the odd case $m=2p-1$,
\begin{align*}
\log T_{\rm abs}(C_{[l_1,l_2]}(W))-\log  T_-(C_{l_1}(W))=&\log T(W,g)+\log \frac{{\rm Det} \ddot\alphas_F}{{\rm Det} \alphas}+ A_{\rm BS,abs}(\b C_{[l_1,l_2]}(W))\\
&-\frac{1}{2} \log T(W,g)+\log \frac{{\rm Det}^{p-1}_{0} \ddot\alphas_{C_{l_1}}}{{\rm Det}^{p-1}_{0,l_1} \alphas}
-A_{\rm BS,abs}(\b C_{l_1}(W))\\
&+\sum_{q=0}^{p-2}(-1)^{q}r_q\log\frac{p-q}{p-q-1}\\
=& \frac{(-1)^p}{2} r_p \log\log\frac{1}{l_1} +\sum_{q=0}^{p-1} (-1)^q r_q \log l_1^{2p-1-2q}\\
&+\frac{1}{2} \log T(W,g)+ \log \frac{{\rm Det}^{p-1}_0 \ddot\alphas_{C_{l_2}}}{{\rm Det}^{p-1}_0 \alphas}\\
&+A_{\rm BS,abs}(\b C_{l_2}(W))+O(l_1).
\end{align*}

In the even case $m=2p$,
\begin{align*}
\log T_{\rm abs}(C_{[l_1,l_2]}(W))-\log  T_-(C_{l_1}(W))=&\log \frac{{\rm Det} \ddot\alphas_F}{{\rm Det} \alphas}+\frac{1}{2}\chi(W)\log 2+A_{\rm BS,abs}(\b C_{[l_1,l_2]}(W))\\
&+\log \frac{{\rm Det}^{p-1}_0 \ddot\alphas_{C{l_1}}}{{\rm Det}^{p-1}_0 \alphas}-\frac{1}{4}\chi(W)\log 2-A_{\rm BS,abs}(\b C_{l_1}(W))\\
&+B^{(2p)}_1(C_{l_1}(W))-B^{(2p)}_2(C_{l_1}(W))-\sum_{q=0}^{p-1}(-1)^{q}r_q\log l_1\\
&+\frac{1}{2}\sum_{q=0}^{p-1}(-1)^{q}r_q\log (2p-2q-1)(2p-2q+1)\\
=&\log \frac{{\rm Det}^{p-1}_0 \ddot\alphas_{C_{l_2}}}{{\rm Det}^{p-1}_0 \alphas}+\frac{1}{4}(-1)^p r_p \log l_2\\
&+A_{\rm BS,abs}(\b C_{[l_1,l_2]}(W))-A_{\rm BS,abs}(\b C_{l_1}(W))\\
&+B^{(2p)}_1(C_{l_2}(W))+O(l_1).
\end{align*}
\end{proof}

\begin{theo}\label{t2}
\begin{align*}
\Rz_{l_1=0} \left(\log T_{\rm abs}(C_{[l_1,l_2]}(W))-\log  T_-(C_{l_1}(W))\right)=& \log T_{\rm abs, ideal}(C_{l_2}(W))-\frac{1}{2}\chi(W) \log 2.
\end{align*}
\end{theo}

As announced in the introduction, we reobtained the formula for the torsion written in Theorem \ref{t0} (or better \ref{t1}) as a limit of a regularisation of the torsion of the frustum (extending a result obtained in \cite{HS3} for the odd case $m=2p-1$). Beside the intrinsic interest of this result, this also shows that the anomaly term $B_1$ appearing in the formula for the analytic torsion of the cone (compare Theorem \ref{t1}) is due to the fact that in this approach to the problem of extending the definition of Ray and Singer of analytic torsion to spaces with conical singularities, a set of eigenfunctions of the Laplace Beltrami operator are missed (those that are not square integrable near the tip of the cone). As a consequence, the spectrum  changes, and the lost part of the spectrum is exactly the one that produces a counter term to the anomaly term $B_1$ in the analytic torsion (the term we called the negative part of the torsion). Due to the symmetry of the problem, that depends on the parity of the dimension, this cancelation happens  in odd dimension, but does not happens  in even dimension. This emerges clearly by comparison  of the formulas given in this section for the different terms composing the negative torsion with the formulas for the corresponding  terms composing  the regular (positive) torsion and given in the previous sections.

\appendix

\section{Formulas for the zeta function of the Hodge-Laplace oeprator}

Decomposing the zeta function of the Hodge-Laplace operator $\Delta$ on an $m$-dimensional oriented compact connected Riemannian manifold $(W,g)$, we have that 
\[
\zeta(s,\Delta^{(q)})=\zeta_{\rm ex}(s,\Delta^{(q)})+\zeta_{\rm cex}(s,\Delta^{(q)}).
\]

When $m=2p-1$ is odd, using duality, this gives
\beq\label{Todd}
\log T(W,g)=\sum_{q=0}^{p-2}(-1)^{q+1} \zeta'_{\rm cex}(0,\Delta^{(q)})+\frac{1}{2} (-1)^p \zeta'_{\rm cex}(0,\Delta^{(p-1)}).
\eeq

Moreover, observing that
\beq
\label{peppo}
\zeta_{\rm cex}(s,\Delta^{(q)})=(-1)^q\sum_{k=0}^q (-1)^k \zeta(s,\Delta^{(k)}),
\eeq
and that
\beq\label{zetazero}
\zeta(0,\Delta^{(q)})=-\dim\ker \Delta^{(q)},
\eeq
using duality we have 
\beq\label{forfor}
\begin{aligned}
\sum_{q=0}^{p-2} (-1)^q \zeta_{\rm cex}(0, \Delta^{(q)})+\frac{1}{2}(-1)^{p-1}\zeta_{\rm cex}(0,\Delta^{(p-1)})
&=-\frac{1}{2}\sum_{q=0}^{p-1} (-1)^q (2p-1-2q)\rk H_q(W).
\end{aligned}
\eeq

\end{document}